\documentclass[11pt]{amsart}

\usepackage{amsmath}
\usepackage{amssymb}
\usepackage{amsthm}
\usepackage{amscd, amsfonts}
\usepackage[dvips]{epsfig}
\usepackage{verbatim}
\usepackage{epsf}

\newcommand{\al}{\alpha}

\newcommand{\ga}{\gamma}
\newcommand{\Ga}{\Gamma}
\newcommand{\de}{\delta}
\newcommand{\ep}{\epsilon}
\newcommand{\la}{\lambda}
\newcommand{\La}{\Lambda}
\newcommand{\si}{\sigma}

\newcommand{\Om}{\Omega}

\newcommand{\cR}{\mathcal{R}}
\newcommand{\RR}{\mathbb R}

\newcommand{\NN}{\mathbb N}

\newcommand{\rar}{\rightarrow}

\newcommand{\tr}{\operatorname{tr}}

\newcommand{\mip}{M_i^{\frac{p_i-1}{2}}}
\newcommand{\mim}{M_i^{-\frac{p_i-1}{2}}}
\newcommand{\ve}{\varepsilon}
\newcommand{\vol}{\operatorname{vol}}

\theoremstyle{plain}
\newtheorem{theorem}{Theorem}[section]
\newtheorem{lemma}[theorem]{Lemma}
\newtheorem{prop}[theorem]{Proposition}
\newtheorem{coro}[theorem]{Corollary}
\newtheorem{conjecture}[theorem]{Conjecture}

\theoremstyle{definition}
\newtheorem{rema}[theorem]{Remark}
\newtheorem{defi}[theorem]{Definition}

\numberwithin{equation}{section}

\title[Compactness and Non-compactness for the Yamabe Problem]{Compactness and
Non-compactness for the Yamabe Problem on Manifolds With Boundary}

\author[Disconzi]{Marcelo M. Disconzi}
\address{Department of Mathematics\\
Stony Brook University\\ Stony Brook, NY 11794}
\curraddr{Department of Mathematics\\
Vanderbilt University\\ Nashville, TN 37240}
\email{marcelo.disconzi@vanderbilt.edu}

\author[Khuri]{Marcus A. Khuri}
\address{Department of Mathematics\\
Stony Brook University\\ Stony Brook, NY 11794}
\email{khuri@math.sunysb.edu}
\thanks{The second author is partially supported by
NSF Grants DMS-1007156, DMS-1308753 and a Sloan Research Fellowship.}

\begin{document}

\begin{abstract}
We study the problem of conformal deformation of Riemannian structure to
constant scalar curvature with zero mean curvature on the boundary. We
prove compactness for the full set of solutions when the boundary is
umbilic and the dimension $n \leq 24$. The Weyl Vanishing Theorem is
also established under these hypotheses, and we provide counter-examples
to compactness when $n \geq 25$. Lastly, our methods point towards a
vanishing theorem for the umbilicity tensor, which will be fundamental for
a study of the nonumbilic case.
\end{abstract}

\maketitle

\section{Introduction\label{introduction}}

The Yamabe problem consists of finding a constant scalar curvature metric
$\tilde{g}$ which is pointwise
conformal to
a given metric $g$ on an $n$-dimensional ($n\geq 3$) compact Riemannian manifold $M$
without boundary.
This is equivalent to producing a positive solution to the following semilinear
elliptic equation
\begin{equation}
 L_g u + K u^\frac{n+2}{n-2}  = 0,  \text{ on } M,
\label{Yamabe_eq}
\end{equation}
where $K$ is a constant, $L_g = \Delta_g  - c(n) R_g $ is the conformal Laplacian
for $g$ with scalar curvature $R_g$, and
$c(n) = \frac{n-2}{4(n-1)}$. If $u>0$ is a solution of (\ref{Yamabe_eq}) then the
new metric
$\tilde{g} = u^\frac{4}{n-2} g$ has scalar curvature $c(n)^{-1} K$. This problem was
solved in the
affirmative through the combined works
of Yamabe \cite{Ya}, Trudinger \cite{Tr}, Aubin \cite{Au} and Schoen \cite{S1} (see
also \cite{LP} for a complete overview).

From an analytic perspective the Yamabe problem has proven to be a rich source of
interesting ideas. The complete solution of
the problem
was the first instance of a satisfactory existence theory for equations involving a
critical exponent, where
the standard techniques of the calculus of variations fail to apply. When the first
eigenvalue of the conformal
Laplacian is positive, which is equivalent to the case of positive Yamabe invariant
$Y(M)$ (see \cite{SY1} for a definition),
solutions to (\ref{Yamabe_eq}) are not unique, and it is known that the set of
solutions can be quite large (\cite{Po,Ob}).

Therefore it becomes natural to ask what can be said about the full set of solutions
to (\ref{Yamabe_eq}) when
$Y(M) > 0$. While this set is noncompact in the $C^2$ topology when the
underlying manifold is $S^n$ with the round metric (see \cite{Ob}), when $M$ is not
conformally equivalent to the
round sphere compactness was established in various
cases, namely by Schoen \cite{S0} in the locally conformally flat case,
Schoen and Zhang \cite{SZ} in three dimensions,  Druet \cite{Dr} for $n \leq 5$,
Marques \cite{M} for
$n \leq 7$, Li and Zhang \cite{LZh1,LZh2} for $n \leq 11$. However in a surprising
turn of events, counterexamples to compactness
were found by Brendle \cite{Br} when $n \geq 52$ and subsequently by Brendle and
Marque \cite{BM} for $ 25 \leq n \leq 51$.
Finally, Khuri, Marques and Schoen \cite{KMS} proved that compactness does hold in
all remaining cases, that is,
for $n\leq 24$. See \cite{BM2} for a survey of various compactness and
non-compactness results for the Yamabe equation.

An obvious extension of such problems is to consider manifolds with boundary. In
this case one would like to conformally
deform a given metric to one which has not only constant scalar curvature but
constant mean curvature as well. This problem is equivalent
to showing the existence of a positive solution to the boundary value problem
\begin{align}
\begin{cases}
L_{g}u + K u^{\frac{n+2}{n-2}} = 0, & \text{ in } M, \displaybreak[0] \\
B_g u = \partial_{\nu_{g}}u + \frac{n-2}{2}\kappa_{g} u  = \frac{n-2}{2} c
u^\frac{n}{n-2}, & \text{ on } \partial M,
\end{cases}
\label{Yamabe_eq_bry}
\end{align}
where $\nu_g$ is the unit outer normal and $\kappa_g$ is the mean curvature. If such
a solution exists then the
metric $\tilde{g} = u^\frac{4}{n-2} g$ has scalar curvature $c(n)^{-1} K$ and the boundary has mean
curvature $c$. This Yamabe
problem on manifolds with boundary was initially investigated by Escobar
\cite{Es,Es2}, who solved the problem
 affirmatively in several cases.
With contributions from several authors (see \cite{Es3,M2,M3,HL,HL2,A,Al2,BC,C}),
most of the cases have now been solved.

Notice that if $K \neq 0$ and $c \neq 0$ then both the equation and the boundary
condition are nonlinear. In order to
simplify the problem, it is customary to assume then that one of them is linear,
that is, that either $K$
or $c$ is zero. Geometrically, this corresponds to deforming the manifold to one
with either constant nonzero scalar curvature and
zero mean curvature on the boundary ($K \neq 0$, $c=0$) or zero scalar curvature and
constant nonzero mean
curvature on the boundary ($K=0$, $c\neq 0$). In this paper we will focus on the
first of these two cases.

In analogy to the case of manifolds without boundary, where the round sphere
provides the canonical example of noncompactness, when the manifold has boundary and
is not conformally equivalent
to the round hemisphere, the question of compactness of solutions arises.
Compactness was proven
by Han and Li \cite{HL} when the scalar curvature is negative ($K<0$) and the mean
curvature is zero ($c=0$), and also when
the scalar curvature is positive ($K > 0$) with no restriction on the mean curvature
but with the extra hypotheses that
the manifold is locally conformally flat and the boundary is umbilic; by Felli and
Ahmedou \cite{FA} when the scalar curvature is
zero ($K=0$), the mean curvature positive ($c>0$), the manifold is locally
conformally flat and the boundary umbilic (see also\cite{FA2});
and by Almaraz \cite{Al} when the scalar curvature is zero $(K=0$), $n \geq 7$, and
a generic condition on the trace-free part
of the second fundamental form holds.

It is natural to consider subcritical approximations to equation
(\ref{Yamabe_eq_bry}), where a priori estimates are readily
available. Thus we define
\begin{equation*}
\Phi_p = \Big \{ u > 0 ~ \big | ~ L_gu + Ku^p = 0 \text{ in } M,~B_g u = 0 \text{ on
} \partial M \Big \},
\end{equation*}
for $p \in [1,\frac{n+2}{n-2}]$. Furthermore, as the case $K<0$ has already been
treated in \cite{HL}, we will assume from
now on that $K>0$.  Then our main result may be stated as follows.

\begin{theorem} (Compactness) Let $(M^n,g)$ be a smooth compact Riemannian manifold
of dimension $3 \leq n \leq 24$ with umbilic boundary,
and which is not conformally equivalent to the standard hemisphere $(S^n_+,g_0)$.
Then for any $\ve > 0$
there exists a constant $C>0$ depending only on $g$ and $\ve$ such that
\begin{equation}
 C^{-1} \leq u \leq C \text{   and   } \parallel u \parallel_{C^{2,\al}(M)} \leq C
\nonumber
\end{equation}
for all $u \in \cup_{1+\ve \leq p \leq \frac{n+2}{n-2}} \Phi_p$, where $0 < \al < 1$.
\label{compactness_theorem}
\end{theorem}

This theorem is established by a fine analysis of blow-up behavior at boundary
points; such a fine analysis was
carried out for interior blow-up points in \cite{KMS}. The entire problem is reduced to showing
the positivity of a certain quadratic form on a finite dimensional vector space,
which may be analyzed in a similar
manner as is done in the appendix of \cite{KMS}. Of course this theorem also relies
on the Positive Mass Theorem of
General Relativity, in its usual form. That is, although we are concerned with
manifolds having boundary, we are still
able to use the standard Positive Mass Theorem by employing a doubling procedure.

Another key feature of our approach
is to employ a version of conformal normal coordinates adapted to the boundary,
which elucidates the dependence of
various geometric quantities on the conformally invariant umbilicity tensor and Weyl
tensor. This coordinate system can be thought of as a good compromise between traditional
conformal normal coordinates \cite{LP} and the so-called conformal Fermi coordinates \cite{M2}.
This is because
although the latter has been shown to be a powerful tool to study the Yamabe problem on manifolds
with boundary, a critical part of the compactness result in \cite{KMS} is the proof of
the positivity of the quadratic form mentioned earlier. This proof makes substantial use
of the the radial symmetry coming from normal coordinates and we would like to preserve
as much as possible of that original argument.

In general, it is expected
that wherever blow-up occurs, these conformally invariant quantities will vanish to
high order because, up to a conformal change, the geometry of the manifold resembles that
of a sphere near the blow-up. As we are assuming
that the boundary is umbilic here, we focus on the Weyl tensor. In this regard we prove

\begin{theorem} (Weyl vanishing) Let $g$ be a smooth Riemannian metric defined in
the unit half $n$-ball $B^+_1$,
$6 \leq n \leq 24$. Suppose
that there is a sequence of positive solutions $\{ u_i \}$ of
\begin{equation}
\begin{cases}
L_{g} u_i + K u_i^{p_i} = 0, &  \text{ in } B_1^+ , \\
B_{g} u_i = 0, & \text{ on } \overline{B_1^+} \cap \RR^{n-1} ,
\end{cases}
\nonumber
\end{equation}
$p_i \in (1,\frac{n+2}{n-2}]$, such that for any $\ve > 0$ there exists a constant
$C(\ve)>0$ such that
$\sup_{B_1^+\backslash B_\ve^+} u_i \leq C(\ve)$ and $\lim_{i \rar \infty}
(\sup_{B_1^+} u_i) = \infty$. Assume also that
$\overline{B_1^+} \cap \RR^{n-1}$ is umbilic. Then
the Weyl tensor $W_g$ satisfies
\begin{equation}
 |W_g|(x) \leq C|x|^\ell \nonumber
\end{equation}
for some integer $\ell > \frac{n-6}{2}$.
\label{Weyl_vanishing_theorem}
\end{theorem}

\begin{rema}
It may appear that since the boundary is umbilic, the proofs of theorems \ref{compactness_theorem} and \ref{Weyl_vanishing_theorem}
should follow directly from \cite{KMS} by applying a reflection argument. However, the techniques employed in \cite{KMS}
require a higher degree of regularity than what is typically available from a simple reflection of the metric.
\end{rema}

In analogy to the case without boundary, one wonders if theorem \ref{compactness_theorem} is
false when $n\geq 25$. We have also been able to answer this question.

\begin{theorem} Assume that $n \geq 25$. Then there
exists a smooth Riemannian metric $g$ on the hemisphere $S^n_+$
and a sequence of positive functions $u_i \in C^\infty(S^n_+)$, such that:

\hskip 0.5cm  (a) $g$ is not conformally flat (so in particular $(S^n_+,g)$ is not conformally equivalent
to $(S^n_+,g_0)$, where $g_0$ is the round metric),

\hskip 0.5cm  (b) $\partial S^n_+$ is umbilic in the metric $g$,

\hskip 0.5cm  (c) for each $i$, $u_i$ is a positive solution of the boundary value problem:
\begin{align}
\begin{cases}
L_{g}u_i + K u_i^{\frac{n+2}{n-2}} = 0, & \text{ in } S^n_+, \displaybreak[0] \\
B_g u_i  = 0, & \text{ on } \partial S^n_+,
\end{cases}
\nonumber
\end{align}
where $K$ is a positive constant,

\hskip 0.5cm  (d) $\sup_{S^n_+} u_i \rar \infty$ as $i\rar \infty$.
\label{non_compactness_theorem}
\end{theorem}
Together, theorems \ref{compactness_theorem} and \ref{non_compactness_theorem} give a complete
answer to the question of compactness of solutions to the Yamabe problem on manifolds with
umbilic boundary in the positive scalar curvature setting (Almaraz has proven an analogue to theorem
\ref{non_compactness_theorem} for scalar-flat manifolds \cite{Al3}).

The proof of theorem \ref{non_compactness_theorem} relies heavily on \cite{Br,BM}. In fact, with \cite{Br,BM} at hand, the idea
to proof theorem \ref{non_compactness_theorem} is not complicated.
Brendle and Marques' construction is a perturbation of the round sphere $(S^n,g_0)$.
Although their solutions are constructed on $S^n$ rather than $S^n_+$, they ``almost''
satisfy the boundary condition. We can therefore slightly modify
Brendle and Marques' solutions in order to produce a blow-up sequence for the hemisphere.

One obvious consequence of theorem \ref{compactness_theorem} is to give an
alternative proof of the solution to the Yamabe problem, allowing us to compute the total Leray-Schauder
degree of all solutions to (\ref{Yamabe_eq_bry}) (with $c=0$), and to obtain
more refined existence theorems. This is discussed at the end of the paper (see
section \ref{Leray}).

As mentioned earlier, certain conformally invariant quantities are expected to vanish to high order at a
blow-up point. In particular such behavior is expected for the umbilicity tensor when the boundary is not
umbilic. In this regard, we expect the following.

\begin{conjecture} Let $g$ be a smooth Riemannian metric defined in
the unit half $n$-ball $B^+_1$,
$4 \leq n \leq 24$. Suppose
that there is a sequence of positive solutions $\{ u_i \}$ of
\begin{equation}
\begin{cases}
L_{g} u_i + K u_i^{p_i} = 0, &  \text{ in } B_1^+ , \\
B_{g} u_i = 0, & \text{ on } \overline{B_1^+} \cap \RR^{n-1} ,
\end{cases}
\nonumber
\end{equation}
$p_i \in (1,\frac{n+2}{n-2}]$, such that for any $\ve > 0$ there exists a constant
$C(\ve)>0$ such that
$\sup_{B_1^+\backslash B_\ve^+} u_i \leq C(\ve)$ and $\lim_{i \rar \infty}
(\sup_{B_1^+} u_i) = \infty$. Then
the umbilicity tensor $T_g$ satisfies
\begin{equation}
 |T_g|(x) \leq C|x|^m,~ x\in \overline{B^{+}_{1}}\cap\mathbb{R}^{n-1}, \nonumber
\end{equation}
for some integer $m > \frac{n-4}{2}$. Moreover, if $n \geq 6$ we also have
\begin{equation}
 |W_g|(x) \leq C|x|^\ell,~x\in B^{+}_{1},   \nonumber
\end{equation}
for some integer $\ell > \frac{n-6}{2}$.
\label{conjecture_non_umbilic}
\end{conjecture}

Proving this
conjecture would be a key step towards a compactness theorem for manifolds with non-umbilic boundary.
In fact, one of the main ingredients of our proofs is to estimate several relevant quantities in terms
of the umbilicity tensor and its derivatives at the origin. The vanishing of these terms should allow one, at least
in principle, to adapt the ideas presented here to the non-umbilic case.

\section{Setting, notation, and basic definitions\label{setting_and_notation}}

Let $M^n$ be a $n$-dimensional Riemannian manifold with smooth boundary,
and let $\{ g_i \}_{i=1}^\infty$ be a sequence of metrics on $M$ converging
in $C^k(M)$ to a metric $g$, where $k$ is large and depends only on $n$. Let $\{ u_i
\}$ be a sequence of
positive solutions of the boundary value problem
\begin{gather}
\begin{cases}
L_{g_i}u_i + K f_i^{-\de_i}u_i^{p_i} = 0, & \text{ in } M, \displaybreak[0] \\
B_{g_i}u_i = \partial_{\nu_{g_i}}u_i + \frac{n-2}{2}\kappa_{g_i} u_i = 0, & \text{
on } \partial M,
\end{cases}
\label{bvp}
\end{gather}
where $L_{g_i} = \Delta_{g_i} - c(n)R_{g_i}$, $c(n) = \frac{n-2}{4(n-1)}$, $R_{g_i}$
is the scalar curvature
 of the metric $g_i$,  $K = n(n-2)$, $\nu_{g_i}$ is the outer unit normal, $\kappa_{g_i}$ is the
mean curvature of the boundary, $\{ f_i \}$ is a sequence of smooth positive
functions converging in $C^2(M)$ to a smooth
positive function $f$, $1 < p_i \leq \frac{n+2}{n-2}$, $\delta_i = \frac{n+2}{n-2} -
p_i$. $L_{g}$ is referred to as the
conformal Laplacian, and the boundary value problem (\ref{bvp}) is conformally
invariant (see proposition \ref{transformation_laws}).

In conformal normal coordinates (see \cite{LP,SY1,Ca}) centered at a point $p$,
we write $g(x) = \exp (h(x))$, where $h$ is a smooth function taking values in the
space of
symmetric $n\times n$ matrices. From standard properties of conformal normal
coordinates it then follows that
$x^j h_{ij}(x) = 0$, and $\tr h_{ij}(x) = O(r^N)$,
where $r=\operatorname{dist}_g(x,p)$ and $N$ is arbitrarily large.
We also have
$\det g_{ij} = 1 + O(r^N)$.

In most of the text we will identify the center $p$ of normal coordinates with the
origin. We will write $u_i(x)$ instead
of $u_i(\exp_p(x))$ and $|x|$ instead of $\operatorname{dist}_g(x,p)$. Since $N$ in
the above expressions is as large as we want,
we will often ignore the $O(r^N)$ contribution in the volume element and write
$d\vol_g(x) = dx$.

The proofs of theorems \ref{compactness_theorem} and \ref{Weyl_vanishing_theorem}
depend crucially on finding
a good approximation to the scalar curvature in terms of polynomials. To this end we
define, in conformal normal coordinates
\begin{gather}
H_{ij}(x) = \sum_{2\leq |\al | \leq n-4} h_{ij,\al} x^\al
\label{def_H}
\end{gather}
where $h_{ij,\al}$ are the coefficients of the Taylor polynomials centered at the origin.
Notice that we will sometimes use $~_{,\al}$
to denote Taylor coefficients at the origin --- which are multiples of derivatives evaluated
at the origin rather than the derivatives themselves.

We then have $h_{ij} = H_{ij} +
O(|x|^{n-3})$, $H_{ij} = H_{ji}$, $x^j H_{ij}(x) = 0$, and $\tr H_{ij}(x) = 0$. Put also
\begin{gather}
H_{ij}^{(k)} = \sum_{|\al |=k} h_{ij,\al}x^\al ,\label{def_Hk} \displaybreak[0] \\
|H^{(k)}|^2 = \sum_{ij}\sum_{|\al |=k} |h_{ij,\al}|^2, \label{def_norm_Hk}
\end{gather}
and for $\ve > 0$, set $x=\ve y$ and define
\begin{gather}
\tilde{H}^{(k)}_{ij}(y) = H^{(k)}_{ij}(\ve y).
\label{def_tilde_H}
\end{gather}

We will make extensive use of the following standard rescaling argument.
Let $\{ \ve_i \}_{i=1}^\infty$ be a given sequence of positive numbers converging to
zero.
Define $M_i$ by $\mip = \ve_i^{-1}$ and in normal coordinates put $y=\mip x =
\ve_i^{-1}x$ and
\begin{gather}
v_i(y) = M_i^{-1}u_i(x) = M_i^{-1} u_i(\mim y) = \ve_i^{\frac{2}{p_i - 1}}u_i(\ve_i
y) \nonumber
\end{gather}
for $y \leq \si \mip = \varepsilon_i^{-1}\si$, where $|x| \leq \si$ belongs to the
domain of definition of the normal coordinates. Then $v_i$ satisfies
\begin{gather}
\begin{cases}
L_{\tilde{g}_i}v_i + K \tilde{f}_i^{-\de_i}v_i^{p_i} = 0, & \text{ for } |y| \leq
\si\mip, \displaybreak[0] \\
B_{\tilde{g}_i}v_i = \partial_{\nu_{\tilde{g}_i}}v_i +
\frac{n-2}{2}\kappa_{\tilde{g}_i} v_i = 0, & \text{ on } \partial M,
\end{cases}
\label{bvpy}
\end{gather}
where $\tilde{f}_i(y) = f_i(\mim y) = f_i(\ve_i y )$, $(\tilde{g}_i)_{kl}(y) =
(g_i)_{kl}(\mim y) = (g_i)_{kl}(\ve_i y)$
(see \cite{KMS, M}).

We recall some standard definitions (see \cite{KMS,HL}). Consider a sequence $\{ u_i
\}$ of
solutions of (\ref{bvp}). A point $\bar{x} \in M$ is called a
\emph{blow-up} point for $\{u_i\}$ if $u_i(x_i) \rar \infty$ for some $x_i \rar
\bar{x}$.

\begin{defi} A point $\bar{x} \in M$ is called an isolated blow-up point for
$\{u_i\}$ if
there exists a sequence $\{x_i\} \subset M$, $x_i \rar \bar{x}$, where each $x_i$ is
a local maximum for $u_i$ and

1) $u_i(x_i) \rar \infty$ as $i \rar \infty$,

2) $u_i(x) \leq C \operatorname{dist}_{g_i}(x,x_i)^{-\frac{2}{p_i - 1}}$ for $x \in B_{\si}(x_i)$ and
some constants $\si,C>0$.
\label{defi_isolated}
\end{defi}

Notice that the definition of isolated blow-up points is the same as for the
boundaryless case (\cite{HL}).

\begin{rema} If we change the metric by a uniformly bounded conformal factor $\phi >0$
 such that $\phi(x_i) = 1$ and $\nabla \phi(x_i) = 0$, then isolated blow-up
points are preserved.
\end{rema}

\begin{defi} (\cite{HL})
Let $\{ u_i \}$ and $\{x_i\}$ be as in definition \ref{defi_isolated}. $x_i \rar
\bar{x}$ is an isolated simple blow-up point if
for some $\rho \in (0,\si)$ and $C>1$, where $\si$ comes from the definition of isolated blow
up point,
the function
\begin{gather}
\hat{u}_i(r) = r^{\frac{2}{p_i-1}} \bar{u}_i(r) =
\frac{r^{\frac{2}{p_i-1}}}{\vol_{g_i}(M \cap \partial B_r(x_i) )} \int_{M \cap
\partial B_r(x_i) } u(z) dS(z) \nonumber
\end{gather}
satisfies, for large $i$, $\hat{u}_i^\prime < 0 $ for $r$ such that $C \mim \leq r
\leq \rho$.
\end{defi}
Observe that if $\bar{x}$ is an interior point then this definition agrees with
the standard one (compare with \cite{M}).

Throughout the paper we let $U:\RR^n \rar \RR$ be the function $U(y) =
(1+|y|^2)^{\frac{2-n}{2}}$.
$U$ is known as the ``standard bubble''. From \cite{KMS} we have the following.

\begin{defi} Let $\tilde{z}_\ve$ be the solution of
\begin{gather}
 \Delta \tilde{z}_{\varepsilon} + n(n+2)U^\frac{4}{n-2}\tilde{z}_{\ve} =
c(n)\sum_{k=4}^{n-4} \partial_i\partial_j \tilde{H}^{(k)}_{ij} U
\label{equation_z}
\end{gather}
constructed in \cite{KMS}. It is implicitly assumed that $\tilde{z}_\ve \equiv 0$ if
$n=3,4,5$.
\label{defi_tilde_z_i}
\end{defi}
We recall estimate (4.4) of \cite{KMS}
\begin{gather}
|\partial^\beta \tilde{z}_\ve (y) |
\leq C \sum_{|\al |=4}^{n-4} \sum_{\ell k} \varepsilon^{|\al|} | h_{\ell k, \al}| (1
+ |y| )^{|\al|+2 - n - |\beta |} ,
\label{basic_z_i_tilde_estimate}
\end{gather}
which implies
\begin{gather}
|\partial^\beta \tilde{z}_\ve (y) | \leq C (1+|y|)^{2-n-|\beta|},~\text{ for }
|y|\leq \si \ve^{-1}.
\label{bound_z_i_similar_U}
\end{gather}
The role of $\tilde{z}_\ve$ is to provide a sharp correction term for the usual approximation of
the (rescaled) solutions $u$ by $U$ around a blow-up point. $\tilde{z}_\ve$ was introduced in
the context of manifolds without boundary, and one of the main challenges in our paper is to establish
that the same $\tilde{z}_\ve$ can be used in our setting. In other words, we need to show that $\tilde{z}_\ve$ satisfies a natural
boundary condition. In order to accomplish this, we use one of the key results of the paper, theorem \ref{big_estimate},
to show that the umbilicity of the boundary implies severe constraints on the behavior of
the polynomials $\tilde{H}^{(k)}_{ij}$ on the boundary. Then we use the explicit
construction of $\tilde{z}_\ve$ in terms of $\tilde{H}^{(k)}_{ij}$ to show that it satisfies the desired boundary
condition.

\textbf{Notation and terminology used throughout the text:}

(i) $d=[\frac{n-2}{2}]$.

(ii) If $x_i \rar \bar{x}$ is an isolated blow-up point, we denote $M_i = u_i(x_i)$
and $\ve_i^{-1} = \mip$.

(iii) $x^\prime$ denotes the first $n-1$ coordinate functions.

(iv) We use $N$ to denote an integer that is arbitrarily large, coming typically from properties
of conformal normal coordinates, such as $\det (g) = 1 + O(r^N)$.

(v) Let $\Om$ be an open connected set that intersects $\partial M$.
We then set $\partial^\prime \Om = \overline{\Om} \cap \partial M$ and
$\partial^+\Om = \partial \Om \backslash \partial^\prime \Om$.

(vi) In a coordinate system near the boundary, define $B^G_\si(0) = \{ x \in
B_\si(0) ~|~ x^n > G(x) \}$ for some real valued
function $G$, then denote $\partial^\prime B_\si^G(0) = \{ (x^\prime,G(x)) \}$,
$\partial^+B_\si^G(0) = \partial B_\si^G(0) \backslash \partial^\prime G_\si^G(0)$
(see section \ref{boundary_estimates} and corollary
\ref{big_estimate_interior}).

(vii) We will always assume that the blow-up points $\bar{x}$ lie on the boundary
$\partial M$, since theorems
\ref{compactness_theorem} and \ref{Weyl_vanishing_theorem} would otherwise follow
from \cite{KMS} (see section
\ref{proof_compactness}).

(viii) We will switch back and forth between problems (\ref{bvp}) and (\ref{bvpy}),
referring to them as ``$x$-coordinates'' and ``$y$-coordinates''.

(ix) If $x_0 \in \partial M$, then $B_\si(x_0)$ is a ball of radius $\si$ and
center $x_0$, i.e.,
\begin{gather}
B_\si(x_0) = \{ x \in M ~|~ \operatorname{dist}(x,x_0) \leq \si \}.
\nonumber
\end{gather}
Notice that $B_\si(x_0)$ will usually
look more like a half-ball rather than like a full ball, but we will not denote
it by $B^+_\si(x_0)$,
reserving the latter for balls which explicitly satisfy the condition $x^n > 0$.

(x) For $T\geq 0$ we define $\RR^n_{-T} = \{ x \in \RR^n ~|~ x^n > -T \}$ and
$\RR^n_+ = \{ x \in \RR^n ~|~ x^n > 0 \}$.

\section{Estimates near the boundary and boundary conformal normal
coordinates\label{boundary_estimates}}

In this section we will derive estimates for the second fundamental form, mean
curvature, etc., in terms of the umbilicity tensor.
Then we will use these estimates to modify the standard conformal normal coordinates
construction in order to obtain
conformal normal coordinates at the boundary with zero mean curvature.

We first recall a result of Escobar.

\begin{prop} (Conformal normal coordinates at the boundary \cite{Es}) Assume
$\partial M$ is umbilic and
let $x_0 \in \partial M$.
For any $N>0$ there exists a metric $\tilde{g}$ conformal to $g$ such that, in
normal coordinates for $\tilde{g}$
centered at $x_0$
\begin{gather}
\det(\tilde{g}) = 1 + O(r^N) , \nonumber
\end{gather}
where $r=|x|$. If $N \geq 5$ then $R_{\tilde{g}} = O(r^2)$, and $\Delta R_{\tilde{g}}(0) =
-\frac{1}{6}|W_{\tilde{g}}|^2(0)$, and $\kappa_{\tilde{g}} = O(r^2)$. Here $W_{\tilde{g}}$ is
the Weyl tensor.
\end{prop}

Now take conformal normal coordinates at $x_0 \in \partial M$.
 We choose the coordinates so that $\partial_i$, $1 \leq i \leq n-1$, are tangent to
$\partial M$ and $\partial_n$ is
 normal (pointing inward) to $\partial M$ at $x_0=0$. Near $x_0$ the boundary
$\partial M$ may be
expressed as a graph $x^n = F(x^\prime)$, where $x^\prime = (x^1,\dots,x^{n-1})$ and
since normal coordinates
are defined up to a rotation we can assume that $F(0)=|\nabla F(0)|=0$ and the tangent
plane at $0$ is the ``horizontal''
hyperplane $\{ x^n = 0 \}$. Then a basis for the tangent space $T_x \partial M$ is
given by the
vectors $X_i = \partial_i +F_{,i} \partial_n$, $1 \leq i \leq n-1$. The normal is
given as a
covector by $(\nu_g)_n = -1, (\nu_g)_i = F_{,i}$, $1\leq i \leq n-1$, or as a vector
by
\begin{gather}
(\nu_g)^i = g^{ij}(\nu_g)_j = -g^{in} + \sum_{i=1}^{n-1} g^{ij}F_{,j}.
\label{comp_normal}
\end{gather}

If $g=e^h$ then we may write
\begin{gather}
 (\nu_g)^i = -\de^{in} + h_{in} + F_{,i} + O(|h|^2 + |h||\nabla F|). \nonumber
\end{gather}
Define the second fundamental form by
\begin{gather}
\kappa_{ij} = \kappa(X_i,X_j) = g(\nabla_{X_i}\nu_g,X_j),
\label{definition_sff_not_normalised}
\end{gather}
then
\begin{align}
\kappa_{ij} & =  g(\nabla_{i}\nu_g,\partial_j) + F_{,i}g(\nabla_{n}\nu_g,\partial_j)+
F_{,j}g(\nabla_{i}\nu_g,\partial_n) + F_{,i}F_{,j}g(\nabla_{n}\nu_g,\partial_n)
\label{second_ff_error}  \displaybreak[0] \\
& =  \Ga^n_{ij} + F_{;ij} - F_{,i}\Ga^l_{nj}(\nu_g)_l - F_{,j}\Ga^l_{in}(\nu_g)_l -
F_{,j}F_{,i} \Ga_{nn}^l(\nu_g)_l
\nonumber \displaybreak[0] \\
 & = \frac{1}{2}(-\partial_n h_{ij} + \partial_i h_{nj} + \partial_j h_{ni} ) + F_{,ij} +
O(|h||\nabla h| + |\nabla F||\nabla h|) .
\nonumber
\end{align}
The mean curvature is given by
\begin{align}
\kappa & = g^{-1}(X_i,X_j)\kappa_{ij} = \Delta F + \sum_{i=1}^{n-1}\partial_i h_{ni} +
\frac{1}{2} \partial_n h_{nn}
\label{mean_curv_error} \displaybreak[0] \\
& + O(|h||\nabla h| + |\nabla F||\nabla h|+|h||\nabla^2 F| + |\nabla F|^2|\nabla^2 F|
+ |x|^N) ,
\nonumber
\end{align}
where we have used $\sum_{i=1}^n h_{ii} = O(|x|^N)$.
Finally the umbilicity tensor is given by
\begin{gather}
T(X_i,X_j) = T_{ij} = \kappa(X_i,X_j) -\frac{1}{n-1}\kappa g(X_i,X_j).
\label{definition_T_ij}
\end{gather}
Notice that these quantities differ from the usual ones by a multiple of
$| \nu_g |$ (since $\nu_g$ is not necessarily a unit vector).
As we show below (see proposition \ref{bcnc} and corollary \ref{bcnc_int}),
this will immediately yield estimates for the
standard (i.e., defined with respect to a unit vector)
mean curvature, second fundamental form and umbilicity tensor, and it will suffice for
our purposes. In fact,
we will express all desired quantities in terms of $T_{ij}$, and the umbilicity
of the boundary implies that $T_{ij}$ defined with respect to (\ref{comp_normal})
vanishes as well. We remark also that our definition of the mean curvature in this section differs from the
standard one by a multiple of $(n-1)^{-1}$. However, in all other sections of the paper we
adopt the standard convention, unless otherwise specified.

The next theorem will be our main tool to produce estimates. Although its proof is
long,
the idea behind it is quite simple: from properties of conformal normal coordinates
we can derive several identities involving geometric quantities and the functions
$h_{ij}$. We restrict the obtained expressions to their Taylor polynomials, and
successively solve these equations for one quantity in terms of the others, until we
express
all quantities in terms of the umbilicity tensor and an error.
\begin{rema}
It should be noted that in (\ref{second_ff_error}) and (\ref{mean_curv_error}), as
well as in the proof below,
the expression $|h||\nabla h|$ appearing in the error only includes terms of the form
$|h||\partial_i h_{nj}|,~|h||\partial_n h_{ij}|,~|h||\partial_n h_{nn}|$ or
$|h_{ni}||\nabla h|$.
\label{important_remark}
\end{rema}
\begin{rema} Since we will eventually restrict all expressions to their Taylor
polynomials in theorem
\ref{big_estimate}, and $N$ is arbitrarily large,
we will ignore the $O(|x|^N)$ contributions.
\end{rema}
\begin{theorem}
Take conformal normal coordinates at $x_0 \in \partial M$ as described  above and
choose a large integer $N$.
Then there exists a constant $C$, depending only on $N$ such that
for any $\ve > 0$ sufficiently small:
\begin{align}
 \sum_{|\al|=2}^N |\kappa_{,\al}|\varepsilon^{|\al|} & \leq C
\sum_{|\al|=2}^N \sum_{i,j=1}^{n-1} |T_{ij,\al}| \varepsilon^{|\al|}  \nonumber
\displaybreak[0] \\
\sum_{|\al|=2}^N |\Delta F_{,\al}|\varepsilon^{|\al|} & \leq C
\sum_{|\al|=0}^{N-2} \sum_{i,j=1}^{n-1} |T_{ij,\al}| \varepsilon^{|\al|+2}
\nonumber \displaybreak[0]  \\
\sum_{|\al|=0}^N \sum_{i,j=1}^{n-1} |\kappa_{ij,\al}|\varepsilon^{|\al|} & \leq C
\sum_{|\al|=0}^N \sum_{i,j=1}^{n-1} |T_{ij,\al}| \varepsilon^{|\al|} \nonumber
\displaybreak[0]  \\
\sum_{|\al|=2}^N |F_{,\al}|\varepsilon^{|\al|} & \leq C
\sum_{|\al|=0}^{N-2} \sum_{i,j=1}^{n-1} |T_{ij,\al}| \varepsilon^{|\al|+2} \nonumber
\displaybreak[0] \\
\sum_{|\al|=2}^N \sum_{j=1}^{n-1} |h_{nj,\al}|\varepsilon^{|\al|} & \leq C
\sum_{|\al|=1}^{N-1} \sum_{i,j=1}^{n-1} |T_{ij,\al}| \varepsilon^{|\al|+1} \nonumber
\displaybreak[0] \\
\sum_{|\al|=1}^N \sum_{i,j=1}^{n-1} |\partial_n h_{ij,\al}|\varepsilon^{|\al|} & \leq C
\sum_{|\al|=1}^N \sum_{i,j=1}^{n-1} |T_{ij,\al}| \varepsilon^{|\al|} \nonumber
\displaybreak[0] \\
\sum_{|\al|=1}^N |\partial_n h_{nn,\al}|\varepsilon^{|\al|} & \leq C
\sum_{|\al|=1}^N \sum_{i,j=1}^{n-1} |T_{ij,\al}| \varepsilon^{|\al|} \nonumber
\end{align}
where $\al$ denotes partial derivatives in the variables $x^1,\dots,x^{n-1}$ evaluated at
the origin,
and $F$ is the local representation of the
boundary as a graph as explained at the beginning of this section.
Moreover $\kappa(0) = |\nabla \kappa|(0) = F(0) = |\nabla F|(0) = \Delta F(0) = 0$ and
$|\nabla \Delta F|(0) \leq C\sum_{ij}|\nabla T_{ij}|(0)$.
\label{big_estimate}
\end{theorem}
\begin{proof}
We first record several useful calculations.
When repeated indices $i$ or $j$ appear this signifies summation
from 1 to $n-1$.  Using familiar properties of conformal normal
coordinates and (\ref{second_ff_error}) we have
\begin{align}
x^{i}\kappa_{ij} & =\frac{1}{2}[-\partial_{n}(x^{i}h_{ij})+x^{i}\partial_{i}h_{nj}
-\delta_{j}^{i}h_{ni}+\partial_{j}(x^{i}h_{ni})]+x^{i}F_{,ij}
+O(|x||h||\nabla h|+|x||\nabla F||\nabla h|)
\label{one_contraction_kappa}
\displaybreak[0] \\
& =\frac{1}{2}[\partial_{n}(x^{n}h_{nj})+x^{i}\partial_{i}h_{nj}
-h_{nj}-\partial_{j}(x^{n}h_{nn})]+x^{i}\partial_{i}F_{,j}
+O(|x||h||\nabla h|+|x||\nabla F||\nabla h|)\nonumber\displaybreak[0] \\
& =\frac{1}{2}x^{i}\partial_{i}h_{nj}+x^{i}\partial_{i}F_{,j}
+O(|x^{n}||\nabla h|+|x||h||\nabla h|+|x||\nabla F||\nabla h|).
\nonumber
\end{align}
Furthermore
\begin{align}
x^{i}x^{j}\kappa_{ij} & = x^{i}x^{j}F_{,ij}-\frac{1}{2}x^{i}\delta_{i}^{j}h_{nj}
+\frac{1}{2}x^{i}\partial_{i}(x^{j}h_{nj})
\label{contraction_kappa} \displaybreak[0] \\
& +O(|x||x^{n}||\nabla h|+|x|^{2}|h||\nabla h|+|x|^{2}|\nabla F||\nabla
h|)\nonumber\displaybreak[0] \\
& =x^{i}x^{j}F_{,ij}+\frac{1}{2}x^{n}h_{nn}
-\frac{1}{2}x^{n}x^{i}\partial_{i}h_{nn}
 +O(|x||x^{n}||\nabla h|+|x|^{2}|h||\nabla h|+|x|^{2}|\nabla F||\nabla
h|)\nonumber\displaybreak[0] \\
& =x^{i}x^{j}F_{,ij}
 +O(|x^{n}||h|+|x||x^{n}||\nabla h| +|x|^{2}|h||\nabla h|+|x|^{2}|\nabla F||\nabla h|),
\nonumber
\end{align}
and
\begin{align}
x^{i}x^{j}\kappa g(X_{i},X_{j}) & =x^{i}x^{j}\kappa(g_{ij}+F_{,i}g_{jn}+F_{,j}g_{in}
+F_{,i}F_{,j}g_{nn}) \label{contraction_kappa_mean} \displaybreak[0] \\
& =|x^{\prime}|^{2}\kappa+O(|x|^{2}|h|||\kappa|+|x|^{2}||\nabla
F|^{2}|\kappa|) .
\nonumber
\end{align}
Recalling the definition of the umbilicity tensor together with
(\ref{contraction_kappa}) and (\ref{contraction_kappa_mean}) yields
\begin{align}
x^{i}x^{j}F_{,ij} & =\frac{1}{n-1}|x^{\prime}|^{2}\kappa+x^{i}x^{j}T_{ij}
\label{contraction_F}
+O(|x^{n}||h|+|x||x^{n}||\nabla h| +|x|^{2}|h||\nabla h|
\displaybreak[0] \\
&+|x|^{2}|\nabla F||\nabla h|+|x|^{2}|h||\kappa|+|x|^{2}|\nabla
F|^{2}|\kappa|).
\nonumber
\end{align}
Moreover since
\begin{gather}
x^{i}\kappa g(X_{i},X_{j})=x^{j}\kappa+O(|x||h||\kappa|+|x||\nabla F|^{2}|\kappa|) ,
\nonumber
\end{gather}
we find that (using (\ref{one_contraction_kappa}))
{\allowdisplaybreaks
\begin{eqnarray}
& & \frac{1}{2}x^{i}\partial_{i}h_{nj}+x^{i}\partial_{i}F_{,j}
 \label{one_contraction_F_H} \\
& & = \frac{1}{n-1}x^{j}\kappa+x^{i}T_{ij} + O(|x^{n}||\nabla h|+|x||h||\nabla h|+|x||\nabla F||\nabla h|
+|x||h||\kappa|+|x||\nabla F|^{2}|\kappa|).
\nonumber
\end{eqnarray} }
Eliminating $\kappa$ from (\ref{contraction_F}) and (\ref{one_contraction_F_H})
produces
\begin{gather}
x^{i}\partial_{i}F_{,j}+\frac{1}{2}x^{i}\partial_{i}h_{nj}-x^{i}T_{ij}=
|x^{\prime}|^{-2}x^{j}x^{i}x^{l}(F_{,il}-T_{il})+\Omega_{j} ,
\label{F_T_ij_before_restricting}
\end{gather}
where throughout this proof $\Omega_{j}$ denotes error which
satisfies
\begin{align}
\Omega_{j}& =O(|x^{n}||\nabla h|+|x||h||\nabla h|+|x||\nabla
F||\nabla h| +|x||h||\kappa|+|x||\nabla F|^{2}|\kappa|+|x^{\prime}|^{-1}|x^{n}||h|)
\label{error_omega}
\displaybreak[0] \\
& =O(|x^{n}||\nabla h|+|x||h||\nabla h|+|x||\nabla F||\nabla h|
+|x||h||\nabla^{2}F|+|x||\nabla F|^{2}|\nabla^{2}F|+|x^{\prime}|^{-1}|x^{n}||h|).\nonumber
\end{align}
Upon restricting attention to Taylor polynomials (\ref{F_T_ij_before_restricting})
simplifies
to
{ \allowdisplaybreaks
\begin{eqnarray}
& & (k-1)|x^{\prime}|^{2}F_{,j}^{(k)}-k(k-1)x^{j}F^{(k)} \label{F_and_H} \\
& & =  -\frac{1}{2}(k-1)|x^{\prime}|^{2}h_{nj}^{(k-1)}
+|x^{\prime}|^{2}x^{i}T_{ij}^{(k-2)}
-x^{j}x^{i}x^{l}T_{il}^{(k-2)}+|x^{\prime}|^{2}\Omega_{j}^{(k-1)}
\nonumber
\end{eqnarray} }
where $h_{nj}^{(k-1)}$ denotes the $(k-1)$-degree Taylor polynomial in the
variables $x^{1},\ldots,x^{n-1}$ and similarly for $F^{(k)}$,
$T_{ij}^{(k-2)}$.  We note that $h_{nj}^{(k-1)}$ is not the full
Taylor polynomial in all the variables
$x^{1},\ldots,x^{n}$ but rather just the portion involving the
first $n-1$ coordinates, and the remainder involving $x^{n}$ is
relegated to the error term.  Now apply $\partial_{j}$ to (\ref{F_and_H}) and sum
over $j$ to find an equation for $F^{(k)}$,
{ \allowdisplaybreaks
\begin{eqnarray}
& & |x^{\prime}|^{2}\Delta F^{(k)}+k(3-n-k)F^{(k)} =-\frac{1}{2}|x^{\prime}|^{2}\partial_{j}h_{nj}^{(k-1)}  \label{eq_F}  \\
& & +(k-1)^{-1}\partial_{j}(|x^{\prime}|^{2}x^{i}T_{ij}^{(k-2)}-x^{j}x^{i}x^{l}T_{il}^{(k-2)})
+\partial_{j}(|x^{\prime}|^{2}\Omega_{j}^{(k-1)})
\nonumber
\end{eqnarray} }
where we used $x^i h_{ij} = -x^nh_{nj}$ to absorb this term in the error.
Differentiate
(\ref{one_contraction_F_H}) with respect to $x^j$ and sum over $j$ to find
\begin{gather}
 \frac{1}{2}\partial_j h_{nj} + \frac{1}{2} x^i \partial_i \partial_j h_{nj} +
\Delta F + x^i\partial_i \Delta F
= x^i \partial_j T_{ij} + \kappa + \frac{1}{n-1} x^j \partial_j \kappa + \partial_j
\Om_j
\nonumber
\end{gather}
where we used that the error term in (\ref{one_contraction_F_H}) has the form
$\Om_j$. Then
\begin{gather}
 \frac{1}{2}(k-1)\partial_j h_{nj}^{(k-1)} + (k-1) \Delta F^{(k)} =
x^i \partial_j T_{ij}^{(k-2)} + \frac{n+k-3}{n-1} \kappa^{(k-2)} + \partial_j
\Om_j^{(k-1)}.
\label{partial_j_h_nj_F_T_kappa_error_1}
\end{gather}
On the other hand (\ref{contraction_F}) gives
\begin{gather}
k(k-1)F^{(k)}-x^{i}x^{j}T_{ij}^{(k-2)}=\frac{1}{n-1}|x^{\prime}|^{2}\kappa^{(k-2)}+x^{i}\Omega_{i}^{(k-1)}.
\label{F_k_terms_T_kappa_error}
\end{gather}
Therefore using (\ref{eq_F}) and (\ref{F_k_terms_T_kappa_error}) in
(\ref{partial_j_h_nj_F_T_kappa_error_1}) produces
\begin{gather}
\partial_j h_{nj}^{(k-1)} = -\frac{2}{k-1} x^i  \partial_j T_{ij}^{(k-2)}
+ \frac{x^j}{|x^\prime|^2} \Om_j^{(k-1)} + \partial_j \Om_j^{(k-1)}.
\label{partial_j_h_nj_umbilic_error}
\end{gather}

Let $B_{1}^{n-1}$ denote the unit ball with
respect to
$x^{1},\ldots,x^{n-1}$, and let $\phi\in C^{\infty}(S_1^{n-2})$.
Extend $\phi$ radially so that it is defined on $B_1^{n-1}\backslash\{0\}$ and
$\partial_{r}\phi=0$ on $S_1^{n-2}$,
where $r=|x^{\prime}|$. Notice that even though $\phi$ is not defined at the origin,
we can still integrate
by parts against functions which vanish at zero, and so in particular against
homogeneous polynomials.

Let $\phi$ be as above. From (\ref{one_contraction_F_H}) we have
\begin{gather}
 (k-1)\partial_i F^{(k)} = \frac{1}{n-1} x^i \kappa^{(k-2)} -
\frac{k-1}{2}h_{ni}^{(k-1)} + x^j T_{ij}^{(k-2)} + \Om_i^{(k-1)}.
\label{partial_i_F_k_kappa_h_ni_T_error}
\end{gather}
Multiply (\ref{partial_i_F_k_kappa_h_ni_T_error}) by $\partial_i \phi$, sum over $i$
and
integrate by parts to get
{ \allowdisplaybreaks
\begin{eqnarray}
& &-(k-1)\int_{B_1^{n-1}} \phi \Delta F^{(k)} + (k-1) \int_{S_1^{n-2}} \phi \nu^i
\partial_i F^{(k)}  =
-\frac{1}{n-1} \int_{B_1^{n-1}} \phi (n+k-3)\kappa^{(k-2)} \nonumber \\
& & + \frac{1}{n-1}\int_{S_1^{n-2}} \phi \nu^i x^i \kappa^{(k-2)}
+ \frac{k-1}{2} \int_{B_1^{n-1}} \phi \partial_i h_{ni}^{(k-1)}
- \frac{k-1}{2} \int_{S_1^{n-2}} \phi \nu^i h_{ni}^{(k-1)}
\nonumber  \\
& &- \int_{B_1^{n-1}} \phi x^j \partial_i T_{ij}^{(k-1)}
+ \int_{S_1^{n-2}} \phi \nu^i x^j T_{ij}^{(k-2)}
 -\int_{B_1^{n-1}} \phi \partial_i \Om_i^{(k-1)}
+ \int_{S_1^{n-2}} \phi \nu^i \Om_i^{(k-1)}.
\nonumber
\end{eqnarray} }
Integrating in polar coordinates produces
{\allowdisplaybreaks
\begin{eqnarray}
& &-(k-1)\int_0^1 r^{n+k-4}\int_{S_1^{n-2}} \phi \Delta F^{(k)} + (k-1)
\int_{S_1^{n-2}} \phi \nu^i \partial_i F^{(k)}
\nonumber  \\
& = & -\frac{1}{n-1} \int_0^1 r^{k+n-4}\int_{S_1^{n-2}} \phi (n+k-3)\kappa^{(k-2)}
+\frac{1}{n-1}
 \int_{S_1^{n-2}} \phi \nu^i x^i \kappa^{(k-2)}
 \nonumber  \\
& &+ \frac{k-1}{2} \int_0^1 r^{n+k-4} \int_{S_1^{n-2}} \phi \partial_i h_{ni}^{(k-1)}
 - \frac{k-1}{2} \int_{S_1^{n-2}} \phi \nu^i h_{ni}^{(k-1)}
- \int_0^1 r^{n+k-4} \int_{S_1^{n-2}} \phi x^j \partial_i T_{ij}^{(k-1)}
\nonumber   \\
& & + \int_{S_1^{n-2}} \phi \nu^i x^j T_{ij}^{(k-2)}
-\int_0^1 r^{n+k-4} \int_{S_1^{n-2}} \phi \partial_i \Om_i^{(k-1)} +
\int_{S_1^{n-2}} \phi \nu^i \Om_i^{(k-1)},
\nonumber
\end{eqnarray} }
which implies (notice that the mean curvature terms cancel out)
\begin{align}
\Delta F^{(k)} - \frac{k(n+k-2)}{|x^\prime|^2} F^{(k)}
& = -\frac{1}{2} \partial_i h_{ni}^{(k-1)} + \frac{1}{2}\frac{n+k -3}{|x^\prime|^2}
x^ih_{ni}^{(k-1)} +\frac{1}{k-1} x^j \partial_i T_{ij}^{(k-2)}
\label{Laplacian_F_and_F_error_1} \displaybreak[0] \\
 & - \frac{n+k-3}{k-1}
\frac{x^ix^j}{|x^\prime|^2} T_{ij}^{(k-2)}
+ \partial_i \Om_i^{(k-1)} + \frac{x^i}{|x^\prime|^2} \Om_i^{(k-1)},
\nonumber
\end{align}
where we have used that $\phi$ is an arbitrary smooth function on $S_1^{n-2}$ and
homogeneous polynomials
are determined by their values on the sphere.
Using (\ref{eq_F}) and (\ref{partial_j_h_nj_umbilic_error}) in
(\ref{Laplacian_F_and_F_error_1}) we find that
\begin{gather}
 x^i h_{ni}^{(k-1)} = -\frac{2}{(k-1)(n+k-3)}|x^\prime|^2 x^j \partial_i
T_{ij}^{(k-2)} + x^j \Om_j^{(k-1)}
+ |x^\prime|^2 \partial_j \Om_j^{(k-1)} .
\label{h_ni_umbilic_error}
\end{gather}
Similarly, multiplying (\ref{F_and_H}) by $\partial_j \phi$ and integrating by parts
yields
\begin{align}
 |x^\prime|^2 \Delta F^{(k)} + \frac{k(k-n-1)}{k-1} F^{(k)}
&=  \frac{1}{2} (n+k -3)
x^j h_{nj}^{(k-1)}
-\frac{1}{2}|x^\prime|^2 \partial_j h_{nj}^{(k-1)} + \frac{1}{k-1}|x^\prime|^2 x^i \partial_j T_{ij}^{(k-2)}
\label{Laplacian_F_and_F_error_2}
\displaybreak[0] \\
 & - \frac{n+k-3}{k-1} x^i
x^j T_{ij}^{(k-2)} +
x^j\Om_j^{(k-1)} + |x^\prime|^2\partial_j\Om_j^{(k-1)}. \nonumber
\end{align}
Solving for $\Delta F^{(k)} + \frac{1}{2} \partial_j h_{nj}^{(k-1)}$ in
(\ref{eq_F}) and using it along with (\ref{h_ni_umbilic_error}) in
(\ref{Laplacian_F_and_F_error_2}) we obtain
\begin{gather}
 F^{(k)} = \frac{n+k-3}{k(2n + 3k -nk -k^2 - 3)}x^ix^j T_{ij}^{(k-2)} + x^j
\Om_j^{(k-1)} +
|x^\prime|^2 \partial_j \Om_j^{(k-1)}.
\label{F_umbilic_error}
\end{gather}
Notice that the denominator of the first term on the right hand side is never zero
since $k\geq 2$.

From (\ref{contraction_F})  we have
\begin{gather}
k|x^{\prime}|^{-2}F^{(k)}=\frac{1}{(n-1)(k-1)}\kappa^{(k-2)}+\frac{1}{k-1}|x^{\prime}|^{-2}
x^{i}x^{j}T_{ij}^{(k-2)}+|x^{\prime}|^{-2}x^{j}\Omega_{j}^{(k-1)}.
\label{F_kappa_T_error}
\end{gather}
Using (\ref{F_umbilic_error}) in (\ref{F_kappa_T_error}) yields
\begin{gather}
 \kappa^{(k-2)} = c(n,k) \frac{x^i x^j}{|x^\prime|^2} T_{ij}^{(k-2)} +
\frac{x^j}{|x^\prime|^2} \Om_j^{(k-1)}
+ \partial_j \Om_j^{(k-1)} ,
\label{kappa_umbilic_error}
\end{gather}
where $c(n,k)$ is a numerical factor depending on $n$ and $k$ only.
Let $ \cR$ be the set of homogeneous polynomials that can be estimated
in terms of the umbilicity tensor and an error (of the same degree) in $\Om_j$.
Then (\ref{partial_j_h_nj_umbilic_error}), (\ref{h_ni_umbilic_error}),
(\ref{F_umbilic_error})
and (\ref{kappa_umbilic_error}) give that $\partial_i h_{ni}^{(k-1)}$, $x^i
h_{ni}^{(k-1)}$,
$F^{(k)}$, $\kappa^{(k-2)}$ $\in \cR$. From (\ref{Laplacian_F_and_F_error_1}) it
then follows that
$\Delta F^{(k)} \in \cR$ as well. From (\ref{F_and_H}) and $F^{(k)} \in \cR$ we
get $h_{nj}^{(k-1)}\in \cR$, and from (\ref{mean_curv_error}) and $\Delta
F^{(k)},~\partial_i h_{ni}^{(k-1)} \in \cR$
it follows that $(\partial_n h_{nn})^{(k-2)} \in \cR$. Using (\ref{definition_T_ij})
along with $\kappa^{(k-2)} \in \cR$ we get $\kappa_{ij}^{(k-2)} \in \cR$ and from this,
(\ref{second_ff_error}), $h_{nj}^{(k-1)}\in \cR$ and $F^{(k)} \in \cR$
we find that $(\partial_n h_{ij})^{(k-2)} \in \cR$.

The inequalities of theorem \ref{big_estimate} now follow  with the help of remark
\ref{important_remark}.

By our construction of $F$ and properties of conformal normal coordinates we have
$\kappa(0) = |\nabla \kappa|(0) = F(0) = |\nabla F|(0)$. Hence in order to finish the
theorem we only have to show that $\Delta F(0) = 0$ and $|\nabla \Delta F|(0) \leq C
\sum_{ij} \nabla |T_{ij}|(0)$.

Using the definition of $\kappa_{ij}$, and
recalling that $\nu_n = -1$ and $\nu_j = F_{,j}$, $1\leq j \leq n-1$, we obtain
\begin{align}
 \kappa_{ij} & = F_{,ij} - \Ga_{ij}^k F_{,k} + \Ga_{ij}^n - F_{,j}\Ga_{in}^k F_{,k} +
F_{,j}\Ga_{in}^n
\label{sff_F_Ga} \displaybreak[0] \\
& - F_{,i}\Ga_{nj}^k F_{,k} + F_{,i}\Ga_{nj}^n - F_{,i}F_{,j}\Ga_{nn}^kF_{,k} +
F_{,i}F_{,j}\Ga_{nn}^n
\nonumber
\end{align}
where we have used $F_{,n} =0$ and $\sum_{k=1}^{n-1}\Ga_{ij}^k F_{,k} = \Ga_{ij}^k
F_{,k}$ since $F$ does not depend on $x^n$.
Evaluating (\ref{sff_F_Ga}) at $0$  and using $\Ga_{ij}^k(0) = 0 = |\nabla F |(0)$,
we have
$\kappa_{ij}(0) = F_{ij}(0)$.
Taking a trace produces $\Delta F (0) = \kappa(0) = 0$. Finally
notice that (\ref{error_omega}) gives
\begin{gather}
 \Om_i = O(|x^n||x| + |x|^4 + |x|^3 + |x^\prime|^{-1}|x^n||x|^2), \nonumber
\end{gather}
so we can compute directly from (\ref{F_umbilic_error}) to find $|\nabla \Delta F(0)|
\leq C\sum_{ij}|\nabla T_{ij}(0)|$,
finishing the proof.
\end{proof}

Now with the help of theorem \ref{big_estimate} we improve the properties of
conformal normal coordinates at the boundary by showing that we can also require zero
mean curvature.
We call these coordinates \emph{boundary conformal normal coordinates} to avoid
confusion with the usual conformal normal
coordinates at a point on the boundary.

\begin{prop} (Boundary conformal normal coordinates) Let $(M,g_0)$ be a Riemannian
manifold with umbilic boundary and $x_0 \in \partial M$. Fix an integer $N\geq 5$.
Then there exists a metric $\tilde{g}$ conformal to $g_0$ such that, in
$\tilde{g}$-normal coordinates centered at $x_0$:
(i) $\det\tilde{g} = 1 + O(r^N)$, (ii) $R_{\tilde{g}} = O(r^2)$, (iii)
$\Delta_{\tilde{g}} R_{\tilde{g}}(0) = -\frac{1}{6}|W_{\tilde{g}}|^2(0)$
and (iv) $\kappa_{\tilde{g}} = 0$ near $x_0$, where $r=|x|$.
\label{bcnc}
\end{prop}
\begin{proof}
Using conformal normal coordinates at $x_0$ we obtain a metric $g$ which satisfies
properties (i)-(iii) in a ball $B_\si(0)$. Our task is to show that we can perform a
further
conformal change in the metric in order to obtain property (iv) while maintaining
(i)-(iii).

We write all quantities as explained above (see equation (\ref{comp_normal}) and
what follows); in particular we denote by
$\kappa_g$ the mean curvature defined as in (\ref{mean_curv_error}), and by
$\widehat{\kappa}_g$ the
mean curvature defined in the usual way, i.e., with respect to a unit vector.

If $\tilde{g} = e^{2f}g$ then $\widehat{\kappa}_{\tilde{g}} =
e^{-f}(\widehat{\kappa}_g + \frac{\partial f}{\partial \nu_g} )$. We will
choose $f$ appropriately.

Because the boundary is umbilic, $T_{ij}$ vanishes identically and therefore theorem
\ref{big_estimate}
gives $\kappa_{g,\al}(0) = 0$ for $|\al| =0,\dots,N$, where $\al$ denotes
derivatives with respect to $x^1,\dots,x^{n-1}$. In
other words, we obtain that $\kappa_g = O(|x^\prime|^N)$, from which it follows that
$\widehat{\kappa}_g = O(|x^\prime|^N)$ as well. Now we choose an extension of
$\widehat{\kappa}_g$ to
$\widetilde{\kappa}_g$ in a neighborhood of $x_0$, with $\widetilde{\kappa}_g$
satisfying
$\widetilde{\kappa}_g = O(|x|^N)$ and $\frac{ \partial
\widetilde{\kappa}_g}{\partial \nu_g} = 0$ . Such an extension is possible
because $\widehat{\kappa}_g= O(|x^\prime|^N)$.

Now pick a smooth function $\widetilde{f}$ such that $\frac{\partial
\widetilde{f}}{\partial \nu_g} = -1$ near $x_0$
and put $f= \widetilde{f} \widetilde{\kappa}_g$. With this choice of $f$ we then
have $\widehat{\kappa}_{\tilde{g}} = 0$
in a neighborhood of $x_0$.

By construction we have $f=O(|x|^N)$, and so
we obtain the desired result as the remaining properties all follow from
$\det(\tilde{g}) = 1 + O(r^N)$ (after choosing a smooth extension of $f$ to the
whole of $M$).
\end{proof}

\begin{rema}
We stress a point made in the introduction.
The so-called conformal Fermi coordinates \cite{M2} have been
used with great success in the study of the Yamabe problem for
manifolds with boundary (see references mentioned in the introduction).
This expresses the fact that cylindrical coordinates generally work better than spherical ones
for Neumann-type of problems. However, a critical part of the compactness result
of Khuri, Marques and Schoen  \cite{KMS} for boundaryless manifolds is the proof of
the positivity of a quadratic form on Taylor polynomials of the scalar curvature which
naturally arises in the problem. Their proof makes substantial use
of the the radial symmetry coming from normal coordinates and we would like to preserve
as much as possible of that original argument. Boundary conformal normal coordinates
preserve the radial symmetry while displaying features similar to the good properties
of Fermi coordinates, as it is shown below.
\end{rema}

Boundary conformal normal coordinates have the following useful property.

\begin{coro}
 In boundary conformal normal coordinates centered at $x_0 \in \partial M$ the
boundary is given by $x^n = 0$. Moreover, $g_{in}(x^\prime,0) = O(|x^\prime|^N)$, $1 \leq i \leq
n-1$.
\label{boundary_hyper}
\end{coro}
\begin{proof}
 Since the boundary is umbilic and $\kappa_{g} \equiv 0$, it is also totally
geodesic (i.e. $\kappa_{ij} \equiv 0$)
 for the metric $g$, and therefore the boundary is given by $x^n = 0$ in normal
coordinates. The second statement follows from theorem \ref{big_estimate} as $g = e^h$.
\end{proof}

Now we want to extend the previous results for the case of interior points.
Assume that $x_0 \in \mathring{M}$ is an interior point which is sufficiently close to $\partial M$, and take
 conformal normal coordinates at $x_0$. Denote by $\tilde{x}_0 \in \partial M$ the
closet point to $x_0$. We
 can still write the boundary as a graph $x^n=F(x^\prime)$, and since
 normal coordinates are defined up to a rotation we can assume
that $F(0)=-|\tilde{x}_0|$ where $|\tilde{x}_0| = \operatorname{dist}(x_0,\tilde{x}_0)$
 (so that $\tilde{x}_0=(0,\dots,0,-|\tilde{x}_0|)$) and the tangent plane is
 horizontal there, so $|\nabla F(0)|=0$. Moreover, by the Gauss lemma we also have
$\left.\frac{\partial }{\partial \nu_{g}} \right|_{\tilde{x}_0} = g^{nn} \left. \partial_n \right|_{\tilde{x}_0}$.

If we ``translate the boundary'', i.e., define
\begin{gather}
 G(x^\prime) = F(x^\prime) + |\tilde{x}_0| \nonumber
\end{gather}
we have $G(0) =|\nabla G(0)| =0$ and $\partial^\al G = \partial^\al F$.
Set $B_\si^G = \{ x \in B_\si(0) | x^n > G(x^\prime) \}$.
Notice then that a basis for the tangent space
at a point on $\partial^\prime B_\si^G = \{ (x^\prime,G(x^\prime)) \}$
is $X_i = \partial_i + G_{,i}\partial_n = \partial_i + F_{,i}\partial_n$, $1 \leq i
\leq n-1$
(since we are simply translating the boundary). We can then consider all geometric
quantities induced on
the boundary $\partial^\prime B_\si^G$. In this situation, theorem
\ref{big_estimate} holds with $G$
replacing $F$ and all quantities being defined with respect to the boundary
$\partial^\prime B_\si^G$,
except for the conclusions that depend on $\kappa = O(r^2)$, since the boundary
$\partial^\prime B_\si^G$ need not be umbilic. We state this as a corollary.

\begin{coro} Let $x_0 \in \mathring{M}$, take conformal normal coordinates at $x_0$
and assume that
$x_0$ is sufficiently close to $\partial M$ as to have $\partial M \cap B_\si(0)
\neq \emptyset$, where $B_\si(0)$ is the domain
of definition of the conformal normal coordinates. Let $F=F(x^\prime)$ be the local
representation
of the boundary as explained above. Define $G(x^\prime) = F(x^\prime) +
|\tilde{x}_0|$,
$B_\si^G = \{ x \in B_\si(0) | x^n > G(x^\prime) \}$, and  $\partial^\prime B_\si^G
= \{ (x^\prime,G(x^\prime)) \}$.
Then there exists a constant $C$, depending only on $N$, such that for any $\ve > 0$ sufficiently small:
\begin{align}
\sum_{|\al|=2}^N |\widetilde{\kappa}_{,\al}|\varepsilon^{|\al|} & \leq C
\sum_{|\al|=2}^N \sum_{i,j=1}^{n-1} |\widetilde{T}_{ij,\al}| \varepsilon^{|\al|}
\nonumber \displaybreak[0] \\
\sum_{|\al|=2}^N |\Delta G_{,\al}|\varepsilon^{|\al|} & \leq C
\sum_{|\al|=0}^{N-2} \sum_{i,j=1}^{n-1} |\widetilde{T}_{ij,\al}|
\varepsilon^{|\al|+2} \nonumber \displaybreak[0] \\
\sum_{|\al|=0}^N \sum_{i,j=1}^{n-1} |\widetilde{\kappa}_{ij,\al}|\varepsilon^{|\al|} &\leq C
\sum_{|\al|=0}^N \sum_{i,j=1}^{n-1} |\widetilde{T}_{ij,\al}| \varepsilon^{|\al|}
\nonumber \displaybreak[0] \\
\sum_{|\al|=2}^N |G_{,\al}|\varepsilon^{|\al|} & \leq C
\sum_{|\al|=0}^{N-2} \sum_{i,j=1}^{n-1} |\widetilde{T}_{ij,\al}|
\varepsilon^{|\al|+2} \nonumber \displaybreak[0] \\
\sum_{|\al|=2}^N \sum_{j=1}^{n-1} |h_{nj,\al}|\varepsilon^{|\al|} & \leq C
\sum_{|\al|=1}^{N-1} \sum_{i,j=1}^{n-1} |\widetilde{T}_{ij,\al}|
\varepsilon^{|\al|+1} \nonumber \displaybreak[0] \\
\sum_{|\al|=1}^N \sum_{i,j=1}^{n-1} |\partial_n h_{ij,\al}|\varepsilon^{|\al|} & \leq C
\sum_{|\al|=1}^N \sum_{i,j=1}^{n-1} |\widetilde{T}_{ij,\al}| \varepsilon^{|\al|}
\nonumber \displaybreak[0] \\
\sum_{|\al|=1}^N |\partial_n h_{nn,\al}|\varepsilon^{|\al|} & \leq C
\sum_{|\al|=1}^N \sum_{i,j=1}^{n-1} |\widetilde{T}_{ij,\al}| \varepsilon^{|\al|}
\nonumber
\end{align}
where $\widetilde{\kappa}$, $\widetilde{\kappa}_{ij}$ and $\widetilde{T}_{ij}$ are
respectively the mean curvature,
second fundamental form  and umbilicity tensor of $\partial^\prime B_\si^G$, all
defined with respect to the
outer normal
\begin{gather}
 (\nu_g)^i = g^{ij}(\nu_g)_j = -g^{in} + \sum_{i=1}^{n-1} g^{ij}G_{,j} \nonumber
\end{gather}
(which is not necessarily a unit normal) and $\al$ denotes partial derivatives in
the variables $x^1,\dots,x^{n-1}$ evaluated at the origin.
Moreover $G(0) = |\nabla G|(0) =0 $ and $|\nabla \Delta G|(0) \leq
C\sum_{ij}|\nabla \widetilde{T}_{ij}|(0)$.
\label{big_estimate_interior}
\end{coro}

As before, estimates on quantities defined with respect to $\nu_g$, with $\nu_g$ not
necessarily a unit vector, will
suffice for our purposes.

Because $\partial^\al G = \partial^\al F$, estimates for $G$ from corollary
\ref{big_estimate_interior} translate into estimates for $F$.
\begin{coro}
Let $x_0 \in \mathring{M}$ and $\tilde{x}_0 \in \partial M$ be the closest point to
$x_0$. Take conformal normal
coordinates at $x_0$, choose a large integer $N$ and let $F$ be the local
representation of
the boundary as a graph. Then there exists a constant $C$, depending only on $N$
such that for any $\ve > 0$ sufficiently small:
\begin{align}
\sum_{|\al|=2}^N |F_{,\al}|\varepsilon^{|\al|} & \leq C
\sum_{|\al|=0}^{N-2} \sum_{i,j=1}^{n-1} |\widetilde{T}_{ij,\al}|
\varepsilon^{|\al|+2} \nonumber\displaybreak[0] \\
\sum_{|\al|=2}^N |\Delta F_{,\al}|\varepsilon^{|\al|} & \leq C
\sum_{|\al|=0}^{N-2} \sum_{i,j=1}^{n-1} |\widetilde{T}_{ij,\al}|
\varepsilon^{|\al|+2}\nonumber
\end{align}
where $\widetilde{T}_{ij}$ is the umbilicity tensor of the boundary $\partial^\prime
B^G_\si$ as
in corollary \ref{big_estimate_interior}, and $\al$ denotes
partial derivatives in the variables $x^1,\dots,x^{n-1}$ evaluated at the origin.
Moreover $|\nabla F|(0) = 0$, $|\nabla \Delta F|(0) \leq C\sum_{ij}|\nabla
\widetilde{T}_{ij}(0)|$
and $F(0) = -|\tilde{x}_0|$.
\label{big_estimate_interior_F}
\end{coro}

The following corollary will finish the treatment of interior points in this section.

\begin{coro} (Boundary conformal normal coordinates for an interior point) Let
$(M,g_0)$ be a Riemannian manifold with
umbilic boundary and $x_0 \in \mathring{M}$. Fix an integer $N\geq 5$. If $x_0$ is
sufficiently close to $\partial M$, then
there exists a metric $\tilde{g}$ conformal to $g_0$ such that, in
$\tilde{g}$-normal coordinates centered at $x_0$:
(i) $\det\tilde{g} = 1 + O(r^N)$, (ii) $R_{\tilde{g}} = O(r^2)$, (iii)
$\Delta_{\tilde{g}} R_{\tilde{g}}(0) = -\frac{1}{6}|W_{\tilde{g}}|^2(0)$
and (iv) $\kappa_{\tilde{g}} = 0$ near $\tilde{x}_0$, where $r=|x|$ and $\tilde{x}_0 \in \partial M$ is
such that $\operatorname{dist}_{g_0}(x_0,\tilde{x}_0) = \operatorname{dist}_{g_0}(x_0,\partial M)$.
\label{bcnc_int}
\end{coro}
\begin{proof} Let $\tilde{x}_0 \in \partial M$ be the closet point to $x_0$.
Denote by $\{ \tilde{x}^i \}$ conformal normal coordinates centered at $\tilde{x}_0$, $\{
x^i \}$ conformal normal
 coordinates centered at $x_0$, $\widetilde{\kappa}$ the mean curvature of $\partial M$ in $\{
\tilde{x}^i \}$-coordinates,
$\kappa$ the mean curvature of $\partial M$ in $\{ x^i \}$-coordinates.
When $x_0 \rar \tilde{x}_0$ we have $x^i \rar \tilde{x}^i$,
and $\partial_{\al} \kappa(x_0) \rar \partial_{\tilde{\al}} \widetilde{\kappa}(\tilde{x}_0)$
where
$\al$ denotes partial derivatives with respect to
$x^1,\dots,x^{n-1}$ and
$\tilde{\al}$ denotes partial derivatives with respect to
to $\tilde{x}^1,\dots,\tilde{x}^{n-1}$.

By theorem \ref{big_estimate} we have that $\partial_{\tilde{\al}} \kappa(\tilde{x}_0) = 0$ for
$|\tilde{\al}|\leq N$ since
the boundary is umbilic. Because $\partial_{\al} \kappa(x_0)
\rar \partial_{\tilde{\al}} \kappa(\tilde{x_0})$
 as $x_0 \rar \tilde{x}_0$, if $x_0$ is sufficiently
close to $\tilde{x}_0$ we can choose an extension of $\kappa$ to $B_\si(x_0)$ (with
$\si$ small) which is
$O(|x-x_0|^N)$.
The rest of the argument now is similar to the proof of proposition \ref{bcnc}.
\end{proof}

We finish this section with several remarks.

\begin{rema}
One of the key ingredients of our proof is to show that the blow-up sequence $x_i$
lies on the boundary (possibly after passing to a subsequence, see section
\ref{section_estimate_distance}). Before showing that, however, we have to
deal with both the case of a blow-up
sequence belonging
to the boundary and the case of a blow-up sequence belonging to the interior of the manifold.
It will therefore be implicitly understood that
when $x_i \in \mathring{M}$, all quantities $\kappa,~\kappa_{ij}$ and $T_{ij}$ are for
the
boundary $\partial^\prime B_\si^G$, as described above, i.e., we will drop
$\widetilde{~}$ from
the interior quantities for the sake of notation. $F$, however, will
always be the representation
of $\partial M$ as a graph unless stated otherwise.
\label{remark_interior}
\end{rema}

\begin{rema}
Suppose that $x_0 \in \partial M$ or that it is sufficiently close to the boundary,
and in boundary conformal normal coordinates centered at
$x_0$ consider $x=(0,x^n)$.
If we translate the boundary
by $|x^n|$ instead of $|\tilde{x}_0|$,
\begin{gather}
 G(x^\prime) = F(x^\prime) + |x^n|, \nonumber
\end{gather}
we can, for each $|x^n|$, consider geometric quantities induced on the boundary
$\partial^\prime B_\si^G$ as before.
In another words, we
have a foliation of a small neighborhood of the boundary by copies of $\partial
M$. In particular,
we can then think of $T_{ij}$ as defined in
a neighborhood of $\partial M$, allowing us to take derivatives with respect to
$x^n$, Taylor
expand $T_{ij}$ in the $x^n$ direction, etc.
\label{T_ij_defined_neighborhood_partial_M}
\end{rema}

\begin{rema} Since boundary conformal normal coordinates are a special case of conformal normal coordinates,
the results of this section stated for conformal normal coordinates, in particular theorem \ref{big_estimate},
 are still valid if we choose
boundary conformal normal coordinates instead.
\label{bcnc_still}
\end{rema}

\section{Higher order estimates\label{higher}}

Our next goal is to extend the results of theorem \ref{big_estimate} to higher order derivatives
of $h$ in the normal direction. Throughout this section we will work with boundary conformal normal coordinates
centered at a point on the boundary; all definitions are as in section \ref{boundary_estimates}.
We will use Greek letters to denote indices running up to $n$,
Latin letters to denote indices running up to $n-1$, and $x^\prime$ to denote the first $n-1$
coordinates. Notice that in light of
corollary \ref{boundary_hyper} we have that the boundary is given by $x^n = 0$, and as
in section \ref{boundary_estimates}, by a ``translation'' we can consider quantities defined on the neighborhood
of the boundary, so that
\begin{align}
\frac{\partial}{\partial \nu_g} & = - g^{n\tau}\partial_\tau , \displaybreak[0] \label{partial_hyper} \\
\left. g_{n i} \right|_{\partial M} & = g_{n i}(x^\prime,0) = O(|x^\prime|^N) , \label{metric_restricted}  \\
\left.\frac{\partial }{\partial \nu_{g}}
\right|_{\partial M} & = -g^{nn}(x^\prime,0) \partial_n  + \sum_{\ell=1}^{n-1} O(|x^\prime|^N)\partial_\ell .
\label{partial_hyper_restricted}
\end{align}

\begin{theorem}
In boundary conformal normal coordinates at a point on the boundary,
\begin{gather}
\left.  h_{nn} \right|_{\partial M} = h_{nn}(x^\prime,0) =  O(|x^\prime|^N).
\end{gather}
\label{estimate_h_nn_boundary}
\end{theorem}
\begin{proof}
We will compute $\nabla_i \nu^n$ in two different ways. First,
\begin{gather}
\nabla_i \nu^n = - \partial_i g^{nn} - g^{nn} \Ga_{in}^n - g^{nl} \Ga_{il}^n .
\label{nabla_i_nu_n}
\end{gather}
Notice that (\ref{second_ff_error}) becomes in our coordinates
$\Ga_{ij}^n(x^\prime,0) = \kappa_{ij}(x^\prime,0)  = 0 $, and hence
(\ref{nabla_i_nu_n}) gives
\begin{align}
\left. \nabla_i \nu^n \right|_{\partial M} & = - \partial_i g^{nn} -\frac{1}{2} (g^{nn})^2 \partial_i g_{nn}
- \frac{1}{2} g^{nn}g^{nl} (\partial_i g_{nl} + \partial_n g_{li} - \partial_l g_{in} )
\label{nabla_i_nu_n_1} \\
& = - \partial_i g^{nn} -\frac{1}{2} (g^{nn})^2 \partial_i g_{nn} + O(|x^\prime|^{2N-1}) ,
\nonumber
\end{align}
where we used theorem \ref{big_estimate}. In order to simplify notation, here and in the rest of the section
we use the following convention. When an equality is restricted to the boundary we write
$\left. \cdot \right|_{\partial M}$ or $(\cdot)(x^\prime,0)$ on one side of the equation,
and it is implicitly understood that the remaining quantities on the other side are restricted as well.

Now differentiate $g^{nn} g_{nn} + g^{nl}g_{nl} = 1$ with respect to $i$ to obtain
$g^{nn} \partial_i g_{nn} = - g_{nn} \partial_i g^{nn} - \partial_i (g^{nl} g_{nl} )$ and  so (\ref{nabla_i_nu_n_1}) becomes
\begin{align}
\left. \nabla_i \nu^n \right|_{\partial M} & = - \partial_i g^{nn} + \frac{1}{2} g^{nn} g_{nn} \partial_i g^{nn}
- \partial_i (g^{nl} g_{nl} ) + O(|x^\prime|^{2N-1})
\label{nabla_i_nu_n_2} \\
& = - \partial_i g^{nn} + \frac{1}{2} (1 - g^{nl} g_{nl}) \partial_i g^{nn}
- \partial_i (g^{nl} g_{nl} ) + O(|x^\prime|^{2N-1})
\nonumber \\
& = - \frac{1}{2} \partial_i g^{nn} + O(|x^\prime|^{2N-1}) ,
\nonumber
\end{align}
where we used theorem \ref{big_estimate} again. Combining (\ref{nabla_i_nu_n_1}) and (\ref{nabla_i_nu_n_2})
gives
\begin{gather}
 \partial_i g^{nn} + (g^{nn})^2 \partial_i g_{nn} = O(|x^\prime|^{2N-1}) .
\nonumber
\end{gather}
Using $g=e^h$ this becomes
\begin{gather}
 -(1-(g^{nn})^2) \partial_i h_{nn} + \partial_i O_{nn}(-h) + (g^{nn})^2 \partial_i O_{nn}(h) = O(|x^\prime|^{2N-1})
\label{partial_i_h_nn_1}
\end{gather}
where
\begin{gather}
 O_{\mu\si}(h) = \sum_{\ell=2}^\infty \frac{(h^\ell)_{\mu\si}}{\ell !} .
\nonumber
\end{gather}
But
\begin{align}
 O_{nn}(-h) = \frac{1}{2}(h_{nn}h_{nn} + h_{nl}h_{nl}) + \sum_{\ell=2}^\infty \frac{(h^{2\ell})_{nn}}{(2\ell) !}
- \sum_{\ell=1}^\infty \frac{(h^{2\ell+1})_{nn}}{(2\ell+1) !}
\nonumber
\end{align}
and
\begin{align}
 O_{nn}(h) = \frac{1}{2}(h_{nn}h_{nn} + h_{nl}h_{nl}) + \sum_{\ell=2}^\infty \frac{(h^{2\ell})_{nn}}{(2\ell) !}
+ \sum_{\ell=1}^\infty \frac{(h^{2\ell+1})_{nn}}{(2\ell+1) !} ,
\nonumber
\end{align}
and since by theorem \ref{big_estimate} $h_{nl}(x^\prime,0)h_{nl}(x^\prime,0) = O(|x^\prime|^{2N})$,
(\ref{partial_i_h_nn_1}) becomes
\begin{gather}
 (-1 + (g^{nn})^2 + h_{nn} + (g^{nn})^2 h_{nn}) \partial_i h_{nn}
+ (1+(g^{nn})^2)\partial_i \sum_{\ell=2}^\infty \frac{(h^{2\ell})_{nn}}{(2\ell) !}
\label{partial_i_h_nn_induction} \\
+ (1-(g^{nn})^2) \partial_i \sum_{\ell=1}^\infty \frac{(h^{2\ell+1})_{nn}}{(2\ell+1) !} = O(|x^\prime|^{2N-1}) .
\nonumber
\end{gather}
We will now show inductively that $h_{nn}=O(|x^\prime|^k)$ implies $h_{nn} = O(|x^\prime|^{3k})$, $k \leq N$. Notice
that we already know that $h_{nn} = O(|x^\prime|^2)$. Also, as before, the terms $h_{nl}$ appearing
in $(h^{2\ell})_{nn}$, $\ell \geq 2$, and $(h^{2\ell+1})_{nn}$, $\ell \geq 1$, can be estimated
by theorem \ref{big_estimate} and hence they can be
absorbed in the error; in other words we can replace $(h^\ell)_{nn}$ by $(h_{nn})^\ell$ up to
and error $O(|x^\prime|^{2N-1})$ (notice that due to the rules of multiplication of matrices, the terms $h_{nl}$ appearing
in $(h^{2\ell})_{nn}$, $\ell\geq 2$, $(h^{2\ell+1})_{nn}$, $\ell \geq 1$, or in the expansion
of $g^{nn}$ must be multiplied by another
$h_{nl}$ and hence such errors are of the same order of the right hand side of (\ref{partial_i_h_nn_induction})).

Since $g^{nn} \geq C >0$ near the origin, if $h_{nn}=O(|x^\prime|^k)$ then
\begin{gather}
 -1 + (g^{nn})^2 = -(1+g^{nn})(1-g^{nn}) =  -(1+g^{nn})O(|x^\prime|^k) = O(|x^\prime|^k)
\nonumber
\end{gather}
and also
\begin{gather}
h_{nn} + (g^{nn})^2 h_{nn} = O(|x^\prime|^k),
\nonumber
\end{gather}
so
\begin{gather}
-1 + (g^{nn})^2 + h_{nn} + (g^{nn})^2 h_{nn} = O(|x^\prime|^k).
\label{error_h_nn_k_1}
\end{gather}
But
\begin{align}
 (1+(g^{nn})^2)\partial_i \sum_{\ell=2}^\infty \frac{(h^{2\ell})_{nn}}{(2\ell) !} & =
(1+(g^{nn})^2)\partial_i \sum_{\ell=2}^\infty \frac{(h_{nn})^{2\ell}}{(2\ell) !} + O(|x^\prime|^{2N-1})
\label{error_h_nn_k_2} \\
& = O(|x^\prime|^{4k-1}) + O(|x^\prime|^{2N-1} )
\nonumber
\end{align}
and
\begin{align}
 (1-(g^{nn})^2)\partial_i \sum_{\ell=2}^\infty \frac{(h^{2\ell})_{nn}}{(2\ell) !} & =
(1+g^{nn})(1-g^{nn})\partial_i \sum_{\ell=2}^\infty \frac{(h_{nn})^{2\ell}}{(2\ell) !} + O(|x^\prime|^{2N-1})
\label{error_h_nn_k_3} \\
& = (1+g^{nn})O(|x^\prime|^k)O(|x^\prime|^{3k-1}) +  O(|x^\prime|^{2N-1})
\nonumber \\
& = O(|x^\prime|^{4k-1}) +  O(|x^\prime|^{2N-1}).
\nonumber
\end{align}
Therefore (\ref{partial_i_h_nn_induction})-(\ref{error_h_nn_k_3}) give
\begin{align}
 \partial_i h_{nn} & =
 \frac{1}{(-1 + (g^{nn})^2 + h_{nn} + (g^{nn})^2 h_{nn}) }
(1+(g^{nn})^2)\partial_i \sum_{\ell=2}^\infty \frac{(h^{2\ell})_{nn}}{(2\ell) !}
\nonumber \\
& + \frac{1}{(-1 + (g^{nn})^2 + h_{nn} + (g^{nn})^2 h_{nn} ) }
(1-(g^{nn})^2) \partial_i \sum_{\ell=1}^\infty \frac{(h^{2\ell+1})_{nn}}{(2\ell+1) !}
\nonumber \\
& = O(|x^\prime|^{3k-1}) +  O(|x^\prime|^{2N-1 - k}),
\nonumber
\end{align}
provided that $(-1 + (g^{nn})^2 + h_{nn} + (g^{nn})^2 h_{nn})(x^\prime,0)$ is not zero and $k < N$. Since $h(0)=0$ we
conclude that $h_{nn}(x^\prime) = O(|x^\prime|^{3k})$. Repeating the argument we obtain the result.

Now we have to show that the result is still true if $(-1 + (g^{nn})^2 + h_{nn} + (g^{nn})^2 h_{nn})(x^\prime,0)$ vanishes
or is $O(|x^\prime|^N)$, and it is enough to consider this latter case. So suppose that
$(-1 + (g^{nn})^2 + h_{nn} + (g^{nn})^2 h_{nn})(x^\prime,0) = O(|x^\prime|^N)$. Multiply it by $g_{nn}$ and use
$1 = g^{nn} g_{nn} + g^{nl}g_{nl} =g^{nn} g_{nn} + O(|x^\prime|^{2N})$ to get
\begin{gather}
 -g_{nn} + g^{nn} + (g_{nn}+g^{nn})h_{nn} = O(|x^\prime|^N) ,
\nonumber
\end{gather}
which implies $-O_{nn}(h) + O_{nn}(h) + (O_{nn}(h) + O_{nn}(-h)) h_{nn} = O(|x^\prime|^N)$ and therefore
\begin{gather}
 -2 \sum_{\ell=1}^\infty \frac{(h^{2\ell+1})_{nn}}{(2\ell +1)!} + 2h_{nn} \sum_{\ell=1}^\infty \frac{(h^{2\ell})_{nn}}{(2\ell)!}
= O(|x^\prime|^N) .
\nonumber
\end{gather}
As before we can ignore contributions from $h_{nl}$ and replace $(h^\ell)_{nn}$ by $(h_{nn})^\ell$, which gives
\begin{gather}
 -\frac{1}{3}(h_{nn})^3 -2 \sum_{\ell=2}^\infty \frac{(h_{nn})^{2\ell+1}}{(2\ell +1)!} +
\Big ( (h_{nn})^2 + 2\sum_{\ell=2}^\infty \frac{(h_{nn})^{2\ell}}{(2\ell)!} \Big ) h_{nn}
= O(|x^\prime|^N) .
\nonumber
\end{gather}
This gives $(h_{nn})^3 = O((h_{nn})^5) + O(|x^\prime|^N)$. Since $h=O(|x|^2)$ we obtain
$(h_{nn})^3 = O(|x^\prime|^{10})$. But then
\begin{align}
 (h_{nn})^3 & = O((h_{nn})^5) + O(|x^\prime|^N) = O((h_{nn})^3 (h_{nn})^2) + O(|x^\prime|^N) \nonumber \\
& = O(|x^\prime|^{10} (h_{nn})^2) + O(|x^\prime|^N) = O(|x^\prime|^{14}) . \nonumber
\end{align}
Repeating the argument produces $(h_{nn})^3 = O(|x^\prime|^{N})$, which gives the result since
$N$ is as large as we want.
\end{proof}

\begin{theorem}
 In boundary conformal normal coordinates centered at a point on the boundary we have
\begin{align}
 \partial_n^2 h_{nj,\al^\prime}(0) = 0,~~~|\al^\prime| \leq N \label{higher_order_partial_n_2} \\
\partial_n^3 h_{nn,\al^\prime}(0) = 0,~~~|\al^\prime| \leq N \label{higher_order_partial_n_3}
\end{align}
where $\al^\prime$ denotes derivatives with respect to $x^1,\dots,x^{n-1}$. In other words
$\left. \partial_n^2 h_{nj} \right|_{\partial M} = O(|x^\prime|^N)$ and
 $\left. \partial_n^3 h_{nn} \right|_{\partial M} = O(|x^\prime|^N)$.
\label{big_estimate_higher_order}
\end{theorem}
\begin{proof}
Denote by $h_{nl}^{(m)}$ the $m^{th}$ Taylor polynomial of $h_{nl}$. Let $\phi \in C^\infty(S^{n-1}_+)$ and extend
it radially similarly to what was done in theorem \ref{big_estimate} (notice however that
here we have the full, i.e., including $x^n$, Taylor polynomial). Integration by parts yields
\begin{align}
\int_{B_+} \phi^2 \partial_n h_{nl}^{(m)}  & =
-2 \int_{B_+} \phi \partial_n \phi h_{nl}^{(m)}
+ \int_{S^{n-1}_+} \phi^2 x^n h_{nl}^{(m)}
 - \int_{B_1^{n-1}} \phi^2 h_{nl}^{(m)}
\nonumber
\end{align}
where $B_+$ is the half unit ball and $B^{n-1}$ the unit ball in $x^\prime$ coordinates.
Since $h_{nl}(x^\prime,0) = O(|x^\prime|^N)$ by theorem \ref{big_estimate}, we obtain
that the integral over $B^{n-1}$ vanishes. Integrating in polar coordinates as
in theorem \ref{big_estimate} shows that
\begin{align}
\int_{S^{n-1}_+} \phi^2 \partial_n h_{nl}^{(m)} & =
-2\frac{m+n-1}{m+n} \int_{S^{n-1}_+}  \phi \partial_n \phi h_{nl}^{(m)}
+ (m+n-1) \int_{S^{n-1}_+} \phi^2  x^n h_{nl}^{(m)}.
\nonumber
\end{align}
And since this is true for any $\phi \in C^\infty(S^{n-1}_+)$, we conclude
\begin{align}
 \phi \partial_n h_{nl}^{(m)} & =
-2\frac{m+n-1}{m+n} \partial_n \phi h_{nl}^{(m)}
+ \phi (m+n-1) x^n h_{nl}^{(m)}~\text{ on } S^{n-1}_+.
\nonumber
\end{align}
Using theorem \ref{big_estimate} again, or, alternatively, choosing a non-zero
test function such that $\partial_n \phi = 0$ on $S^{n-2} = \partial B^{n-1}$,
it follows that $\partial_n h_{nl}^{(m)} (x^\prime,0) = 0$, from which we conclude
\begin{gather}
\partial_n h_{nl}(x^\prime,0) = O(|x^\prime|^N).
\label{estimate_partial_n_h_nl_sum}
\end{gather}
Now with (\ref{estimate_partial_n_h_nl_sum}) in hand, we repeat the integration by parts argument
with $\partial^2_n h_{nl}$ in place of $\partial_n h_{nl}$ and conclude
(\ref{higher_order_partial_n_2}).

To obtain (\ref{higher_order_partial_n_3}), argue similarly to the above, integrate
$\partial_n^2 h_{nn}$ by parts and use theorem \ref{big_estimate} to conclude
$\partial_n^2 h_{nn}(x^\prime,0) = O(|x^\prime|^N)$; then repeat the argument, expressing
$\partial_n^3 h_{nn}$ in terms of $\partial_n^2 h_{nn}$.
\end{proof}

\begin{rema} Since theorem \ref{big_estimate} gives $\partial_n h_{nn}(x^\prime,0) = O(|x^\prime|^N)$,
using an argument similar to that of theorem \ref{big_estimate_higher_order} we can relate
$\partial_n h_{nn}(x^\prime,0)$ and $h_{nn}(x^\prime,0)$, obtaining in this way an alternative
proof of theorem \ref{estimate_h_nn_boundary}.
\end{rema}

\section{Boundary condition for the correction term\label{boundary_correction}}
In this section we use the results of sections \ref{boundary_estimates} and \ref{higher} to show that
the correction term $\tilde{z}_\ve$ (see definition \ref{defi_tilde_z_i}) satisfies the correct boundary condition.
The idea is to use the results of sections \ref{boundary_estimates} and \ref{higher} to show that
certain homogeneous polynomials that appear in the (explicit) construction of $\tilde{z}_\ve$
satisfy the boundary condition and so will $\tilde{z}_\ve$ itself.
Throughout this section we work with boundary conformal normal coordinates centered at a point on the boundary.
This section relies heavily on the appendix of \cite{KMS} and we will often refer to it.

Lemma (A.6) of \cite{KMS} gives the decomposition
\begin{gather}
 H^{(k)}_{ij} = W^{(k)}_{ij} + \sum_{q=1}^{[ \frac{k-2}{2} ]} (\widehat{H}^{(k)}_q)_{ij}
\label{decomposition_appendix}
\end{gather}
where $W^{(k)}_{ij}$ satisfies $\partial_{ij} W^{(k)}_{ij} = 0$. In \cite{KMS} it is also computed that
\begin{align}
 \partial_{ij} (\widehat{H}^{(k)}_q)_{ij} & = \frac{n-2}{n-1} (k-2q)(k-2q-1)(n+k-2q-1)(n+k-2q-2)|x|^{2q-2} p_{k-2q}
\label{double_div_H_hat} \\
& = C_{n,k,q} |x|^{2q-2} p_{k-2q} ,
\nonumber
\end{align}
where $p_{k-2q}$ is a harmonic polynomial of degree $k-2q$.

Since $H^{(k)}_{ij}(x) = \sum_{|\al|=k} h_{ij,\al}x^\al$, we obtain
\begin{gather}
 \partial_n \partial_{ij} H^{(k)}_{ij}(x) = \sum_{|\al|=k} h_{ij,\al}\partial_{nij} x^\al .
\nonumber
\end{gather}
Now we consider $\left.\partial_n \partial_{ij} H^{(k)}_{ij}\right|_{\partial M}$ and identify the
terms that do not necessarily vanish (recall that the boundary is given by $x^n = 0$).

Consider the case $i,j < n$. In this case if
\begin{gather}
 \sum_{|\al|=k} h_{ij,\al} \left. \partial_{ijn}x^\al \right|_{x^n=0} \neq 0 , \nonumber
\end{gather}
then the non-zero terms have $\al_n = 1$, i.e., we can write the multi-index  $\al$
(of the non-zero terms) as
$\al = (\al^\prime,1)$. Hence the coefficients of the non-vanishing terms are all of the form
\begin{gather}
 h_{ij,\al} = \frac{1}{\al!} \partial^\al h_{ij}(0) = \frac{1}{\al!}  \partial_n h_{ij,\al^\prime}(0),~~~~|\al^\prime| = k -1.
\label{h_ij_al}
\end{gather}

Similarly if $i=n$ and $j< n$ then the coefficients of the non-vanishing terms are all of the form
\begin{gather}
 h_{ij,\al} =  \frac{1}{\al!}  \partial^2_n h_{nj,\al^\prime}(0), ~~~~|\al^\prime| = k - 2,
\label{h_in_al}
\end{gather}
and if $i=j=n$ the coefficients of the non-vanishing terms are all of the form
\begin{gather}
 h_{ij,\al} =  \frac{1}{\al!} \partial^3_n h_{nn,\al^\prime}(0),~~~~|\al^\prime| = k-3,
\label{h_nn_al}
\end{gather}
where $\al^\prime$ is a multi-index with $\al^\prime_n = 0$. Since $k \leq n-4$, we have by
theorems \ref{big_estimate} and \ref{big_estimate_higher_order} that (\ref{h_ij_al}), (\ref{h_in_al})
and (\ref{h_nn_al}) all vanish, and therefore
\begin{gather}
\partial_n \partial_{ij} H^{(k)}_{ij}(x^\prime,0)=0 .
\label{vanishing_normal_double_div_H}
\end{gather}
Combining (\ref{decomposition_appendix}) and $\partial_{ij} W^{(k)}_{ij} = 0$ with
(\ref{double_div_H_hat}) and (\ref{vanishing_normal_double_div_H}) gives
\begin{gather}
 \sum_{q=1}^{[ \frac{k-2}{2} ]} C_{n,k,q}|x^\prime|^{2q-2} \partial_n p_{k-2q}(x^\prime,0) = 0,
\nonumber
\end{gather}
and it then follows from usual decomposition theorems for homogeneous polynomials (see e.g. \cite{ABR}) that each
$\partial_n p_{k-2q}(x^\prime,0)$ vanishes separately.

To compute $\left. \partial_n \tilde{z}_\ve \right|_{\RR^{n-1}}$ it is enough to compute
the derivative of $Z((\widehat{H}^{(k)}_q)_{ij})$ --- the solution to (\ref{equation_z})
with $(\widehat{H}^{(k)}_q)_{ij}$ instead of $\sum_k H^{(k)}_{ij}$. Such a solution takes the form (see \cite{KMS})
\begin{gather}
 Z((\widehat{H}^{(k)}_q)_{ij}) = c(n) \al_{k-2q} (1+|x|^2)^{-\frac{n}{2}} \sum_{j=1}^{q+1} \Ga(k,q,j) |x|^{2j} p_{k-2q}
\nonumber
\end{gather}
where $\al_{k-2q}$ and $\Ga(k,q,j)$ are numerical coefficients. Computing we find
\begin{gather}
\partial_n Z((\widehat{H}^{(k)}_q)_{ij})(x^\prime,0) =
c(n) \al_{k-2q} (1+|x^\prime|^2)^{-\frac{n}{2}} \sum_{j=1}^{q+1} \Ga(k,q,j) |x^\prime|^{2j} \partial_n p_{k-2q}(x^\prime,0).
\nonumber
\end{gather}
But we showed above that $\partial_n p_{k-2q}(x^\prime,0)=0$ and hence
$\partial_n Z((\widehat{H}^{(k)}_q)_{ij})(x^\prime,0)$. Therefore, we have proven
\begin{prop}
Take boundary conformal normal coordinates at a point on the boundary and let $\tilde{z}_\ve$ be as
in definition \ref{defi_tilde_z_i}. Then it satisfies
\begin{align}
\begin{cases}
 \Delta \tilde{z}_{\varepsilon} + n(n+2)U^\frac{4}{n-2}\tilde{z}_{\ve} =
c(n)\sum_{k=4}^{n-4} \partial_i\partial_j \tilde{H}^{(k)}_{ij} U
, & \text{ in } \RR^n_+, \displaybreak[0] \\
\partial_n \tilde{z}_{\ve} = 0, & \text{ on } \RR^{n-1} .
\end{cases}
\end{align}
\label{bvp_z}
\end{prop}

\section{Basic convergence results}

Here we prove some basic convergence results. Most of the results are either known
or modifications of similar results for manifolds
without boundary.

\begin{lemma} Suppose  $x_i \rar \bar{x}$ is an isolated blow-up point. Take normal
coordinates at $x_i$,
\and rescale coordinates to $y$-coordinates. Then in the limit $i \rar \infty$ the
boundary becomes a hyperplane.
\label{limit_hyperplane}
\end{lemma}
\begin{proof}
The metric $\tilde{g}_i$ is obtained from $g$ by (i) a rescaling
$g_{1_i} = M_i^{p_i-1} g_i = \varepsilon_i^{-2} g$ and then (ii) by the change of
coordinates $y=\varepsilon^{-1} x$.
 If we write $g_{1_i} = M_i^{p_i-1} g_i $ in the standard form $g_{1_i} =
\phi_i^\frac{4}{n-2} g_i $ we have
$\phi_i^\frac{2}{n-2} = M_i^\frac{p_i-1}{2} = \varepsilon_i^{-1}$, so transformation
law (\ref{conf_sec_fundamental_form}) gives
$(\kappa_{1_i})_{kj} = \varepsilon_i^{-1} (\kappa_i)_{kj}$. The second
fundamental form transforms as
\begin{align}
(\tilde{\kappa}_i)_{kj}(y) &= \frac{\partial x^p}{\partial y^k} \frac{\partial
x^q}{\partial y^j} (\kappa_{1_i})_{pq}(x) = \varepsilon_i^2 \delta^{kp} \delta^{qj}
(\kappa_{1_i})_{pq}(x) \nonumber \displaybreak[0] \\
&= \varepsilon_i^2  (\kappa_{1_i})_{kj}(x)  = \varepsilon_i^2  \varepsilon_i^{-1}
(h_{i})_{kj}(x)
= \varepsilon_i (\kappa_{i})_{kj}(x) \nonumber
\end{align}
when we change coordinates from $x$ to $y$ via $x = \varepsilon_i y$.
Now the sequence $(\kappa_1)_{kj}(x)$ is bounded because in
$x$-coordinates the metrics converge in $C^k$ ($k$ large), and therefore the second
fundamental
form goes to zero in $y$-coordinates. But $\tilde{g}_{ij}(y) \rar \de_{ij}$ since
\begin{gather}
\tilde{g}_{ij}(y) = g_{ij}(\ve_i y) = \de_{ij} + \ve_i^2 O(|y|^2) \nonumber
\end{gather}
and therefore in the limit the boundary is a hyperplane (see also \cite{HL}).
\end{proof}

\begin{lemma} Let $x_i \rar \bar{x} \in \partial M$ be an isolated blow-up point.
There exists a
constant $C>0$ such that for all $i$ and all $|y| \leq \si \mip$ we have $|v_i(y)|
\leq C$ (where $\si$ comes
from the definition of isolated blow-up point).
\label{boundness_v_i}
 \end{lemma}
\begin{proof}
The proof is similar to the first claim in proposition 4.3 of \cite{M} and it uses
the maximum principle, the
 Harnack inequality and the definition of isolated blow-up point. In fact,
from the definition of $v_i$ and isolated blow-up points we have
that
\begin{gather}
\begin{cases}
v_i(0) = 1, ~\nabla v_i(0) = 0 \displaybreak[0] \\
0 < v_i(y) \leq C |y|^{-\frac{2}{p_i - 1}} & \text{ for } |y| \leq \si
M_i^{\frac{p_i - 1}{2}} = l_i.
\end{cases}
\label{properties_v}
\end{gather}
From these properties it follows that $v_i(y) \leq C$ for $1 \leq |y| \leq l_i$.
Since $L_{\tilde{g}_i} v_i = - Kv_i^{p_i} \leq 0$, using
the maximum principle (corollary \ref{coro_max_principle}) we have
that there exists a constant $C>0$ such that for every $i$,
\begin{gather}
\min_{|y| \leq r}  v_i(y) \geq C^{-1} \min_{|y| = r} v_i(y) \nonumber
\end{gather}
with $0<r\leq 1$. Using the Harnack inequality (lemma \ref{Harnack}) we get
\begin{gather}
\max_{|y| = r} v_i(y) \leq C \min_{|y| = r} v_i(y) , \nonumber
\end{gather}
so that
\begin{gather}
\max_{|y| = r} v_i(y) \leq C \min_{|y| = r} v_i(y)  \leq C \min_{|y| \leq r}  v_i(y)
\leq C v_i(0) \leq C \nonumber
\end{gather}
for $0 < r \leq 1$, and the claim follows.
\end{proof}

The next proposition is the analogue of proposition 4.3 of \cite{M} and of
proposition 1.4 of \cite{HL}.
\begin{prop} Let $x_i \rar \bar{x} \in \partial M$ be an isolated blow-up point and
assume that $R_i \rar \infty$
and $\ep_i \rar 0$ are given. Then $p_i \rar \frac{n+2}{n-2}$ and,
after passing to a subsequence
\begin{gather}
\parallel v_i - U \parallel_{C^2(B_{R_i}(0))} \leq \ep_i \nonumber
\end{gather}
and
\begin{gather}
\frac{R_i}{\log M_i} \rar 0. \nonumber
\end{gather}
\label{conv_C2_small_balls}
\end{prop}
\begin{proof}
Let $R>0$ and $\epsilon > 0$ be given.
From lemma \ref{boundness_v_i} we have $|v_i(y)| \leq C$.

Therefore, by standard elliptic estimates there exists a subsequence of $v_i$
converging in $C^2_{loc}$ to a limit $v$ which satisfies
\begin{gather}
\begin{cases}
\Delta v + K v^p = 0 & \text{ in } \RR^n_{-T}, \displaybreak[0] \\
\frac{\partial v}{\partial y^n} = 0 & \text{ on } \partial \RR^n_{-T} \text{ if } T
< \infty,  \displaybreak[0] \\
v(0) = 1 & \text{ and $y=0$ is a local maximum of } v,
\end{cases} \nonumber
\end{gather}
where $T$ is the limit of a subsequence of
$T_i  = \mip \operatorname{dist}_{g_i}(x_i,\partial M)= \mip |\tilde{x}_i| =|\tilde{y_i}|$. If $T =
\infty$ then the proposition follows
from the well known result of Caffarelli, Gidas, and Spruck (\cite{CGS}). If $T <
\infty$ then
the boundary converges to a hyperplane when $i\rar \infty$ by lemma
\ref{limit_hyperplane}, and the result
follows from proposition \ref{Euclidean_prop_bubble}.
\end{proof}

The following lemma is analogous to lemma 2.1 of \cite{HL}. As in the the proof of
\cite{HL} --- where they assume
conformal flatness --- the idea is to show that if the
$\mip \operatorname{dist}_{g_i}(x_i,\partial M)$ does not stay bounded, then after rescaling the
solutions we obtain an
\emph{interior} blow-up point, in which
case the machinery of \cite{KMS} can be applied (of course, in \cite{HL} they could
not use \cite{KMS} since
such results had not yet been known, but they could still apply whatever was known
about blow-up
points in conformally flat manifolds without boundary; the idea here is similar).
The proof does not require change to $y$-coordinates but we will keep track of the
expression in $y$-coordinates for future use.

\begin{lemma} Let $x_i \rar \bar{x} \in \partial M$ be an isolated simple blow-up
point, with $x_i \in \mathring{M}$. Then
\begin{gather}
\mip \operatorname{dist}_{g_i}(x_i,\partial M)
\nonumber
\end{gather}
 stays bounded.
\label{boundess_T}
\end{lemma}
\begin{proof} Let $\tilde{x}_i$ be such that $\operatorname{dist}_{g_i}(x_i,\partial M) =
\operatorname{dist}_{g_i}(x_i,\tilde{x}_i)$.
The proof is by contradiction.
Consider
a subsequence such that
\begin{gather}
 \mip \operatorname{dist}_{g_i}(x_i,\partial M)= \mip |\tilde{x}_i| \rar \infty
\nonumber
\end{gather}
i.e., $|\tilde{y_i}| \rar \infty$. Put $T_i = \mip |\tilde{x}_i| = |\tilde{y}_i|$
and take normal coordinates at $x_i$.
For $|z| \leq |\tilde{x}_i|^{-1}\si$ (where $\si$ comes from the definition of
isolated blow-up point) define
\begin{gather}
\xi_i(z) = N_i^{-1} u_i(N_i^{-\frac{p_i-1}{2}} z) \nonumber
\end{gather}
where $N_i^{-1} = |\tilde{x}_i|^{\frac{2}{p_i-1}}$.

Notice that $\xi_i$ has the same form as $v_i$ with $N_i$ in place of $M_i$, so if
 $(\tilde{g}_i)(z)_{kl} = (g_i)_{kl}(N_i^{-\frac{p_i-1}{2}}z)$ we see that $\xi_i$
satisfies
\begin{gather}
\begin{cases}
L_{\tilde{g}_i}\xi_i + K \tilde{f}_i^{-\de_i}\xi_i^{p_i} = 0 & \text{  for  } |z| \leq
|\tilde{x}_i|^{-1}\si , \nonumber  \displaybreak[0] \\
B_{\tilde{g}_i}\xi_i = \partial_{\nu_{\tilde{g}_i}}\xi_i +
\frac{n-2}{2}\kappa_{\tilde{g}_i} \xi_i = 0 & \text{  on  } \partial M , \nonumber
\end{cases}
\nonumber
\displaybreak[0]
\end{gather}
where $\tilde{f}_i(z) = f_i(N_i^{-\frac{p_i-1}{2}}z)$. Since $x_i$ is an isolated
simple blow-up point for $u_i$ we have
$u_i(x) \leq C |x|^{-\frac{2}{p_i-1}}$
and then
$\xi_i(z) \leq C  |z| ^{-\frac{2}{p_i-1}}$.
This, together with the fact that
$\xi(0)= |\tilde{x}_i|^{\frac{2}{p_i-1}} u_i(0) = |\tilde{x}_i|^{\frac{2}{p_i-1}}
M_i = T_i^\frac{2}{p_i-1} \rar \infty$
as $i \rar \infty$ implies that $\{ 0 \}$ is an \emph{interior} isolated blow-up
point for $\xi_i$,
hence we can use corollary 2.6 of \cite{KMS}
(with $N_i$ instead of $M_i$ and $\xi_i$ instead of $u_i$) and conclude
that $\xi_i(0) \xi_i \rar w$ in $C^2_{loc}(\RR^n_{-1} - \{0\})$, where  $w > 0$ is the
Euclidean Green's function for the Laplacian centered at $0$ (Euclidean  because
$\tilde{g}_i$ converges to the Euclidean metric) and
$\RR^n_{-1} = \{ z^n > -1 \}$. It also follows that
$B_{\tilde{g}_i}\xi_i = \partial_{\nu_{\tilde{g}_i}}\xi_i +
\frac{n-2}{2}\kappa_{\tilde{g}_i} \xi_i = 0 $
becomes in the limit $\frac{\partial w}{\partial z^n} = 0$ on $\partial \RR^n_{-1}$.
We have (see for instance \cite{HL})
\begin{gather}
w(z) = a|z|^{2-n} + A + O(|z|), ~A > 0. \nonumber
\end{gather}
Define
\begin{gather}
B(r,z,\xi,\nabla \xi) =
\frac{n-2}{2} \xi \frac{\partial \xi}{\partial \nu} -
 \frac{r}{2} |\nabla \xi |^2 + r \big ( \frac{\partial \xi}{\partial \nu} \big )^2.
\nonumber
\end{gather}
Because $0$ is an interior blow-up point, we can use theorem 7.1 of
\cite{KMS} to get
\begin{gather}
\liminf_{r\rar 0} \int_{|z|=r} B(r,z,w,\nabla w) \geq 0 \nonumber
\end{gather}
and a direct computation gives
\begin{gather}
\liminf_{r\rar 0} \int_{|z|=r} B(r,z,w,\nabla w) =-\frac{n-2}{2}A |S^{n-1}|, \nonumber
\end{gather}
contradicting $A > 0$.
\end{proof}

Suppose $x_i \rar \bar{x} \in \partial M$ is an isolated simple blow-up point.
In the notation of lemma \ref{boundess_T}, write $T_i  = \mip \operatorname{dist}_{g_i}(x_i,\partial
M)= \mip |\tilde{x}_i|$.
In $y$-coordinates this becomes $T_i = \mip |\tilde{x}_i| =|\tilde{y}_i|$.
By lemma \ref{boundess_T} we cannot have $T_i \rar \infty$, and passing to a
subsequence we have $T_{i_j} \rar T < \infty$.
Corresponding to the subsequence $\{ T_{i_j} \}$
there is a subsequence $\{ v_{i_j} \}$. Applying proposition
\ref{conv_C2_small_balls} to the
$\{ v_{i_j} \}$ yields $T=0$. Hence, we can hereafter assume that
\begin{gather}
|\tilde{y}_i| \rar 0.
\label{limit_T_0}
\end{gather}

\begin{prop} Let $x_i \rar \bar{x} \in \partial M$ be an isolated simple blow-up
point for the sequence $\{ u_i \}$
 of positive solutions to (\ref{bvp}). Then there exist constants $C>0$, $\si >0$
independent of $i$ such that
\begin{align}
 M_i u_i(x) \geq C^{-1}G_i(x,x_i), & \,\,\,\, \mip \leq |x| \leq \si, \nonumber
\displaybreak[0] \\
M_i u_i(x) \leq C |x|^{2-n}, & \,\,\,\, |x| \leq \si, \nonumber
\end{align}
where $G_i(x,x_i)$ is the Green's function for $L_{g_i}$ centered at $x_i$ with
boundary condition
$B_{g_i} G_i(x,x_i) = 0$ on $\partial^\prime B_\si(x_i)$.
Moreover, after passing to a subsequence
$M_i u_i(x) \rar G(x,\bar{x})$ in $C^2_{loc}(B_\si(\bar{x}) \backslash \{\bar{x}\}
)$, where $G(x,\bar{x})$
 is the Green's function for $L_{g}$ centered at $\bar{x}$ with boundary condition
$B_g G(x,\bar{x}) = 0$ on $\partial^\prime B_\si(\bar{x})$.
\label{u_conv_Green_fc}
\end{prop}
\begin{proof}
The proof is an adaptation of the ideas from \cite{M} and \cite{HL} using
lemma \ref{Harnack} and
proposition \ref{conv_C2_small_balls}.
\end{proof}

\section{A further estimate on $\operatorname{dist}_{g_i}(x_i,\partial M)$.\label{section_estimate_distance}}
Let $x_i \rar \bar{x} \in \partial M$ be an isolated simple blow-up point.
As in the boundaryless case, one of the main features of our proofs is the usual use of coordinates centered at
the points $x_i$. If $x_i \in \mathring{M}$, lemma \ref{boundess_T} then gives an estimate
for the distance of $x_i$ to the boundary. Unfortunately this estimate is not enough for our purposes.
In fact the results of sections \ref{higher} and \ref{boundary_correction} require the center of the
coordinate system to be on the boundary. We therefore have to prove that we can pass to a subsequence such that
$x_i \in \mathring{M}$.

\begin{prop} Suppose $x_i \rar \bar{x}$ is an isolated simple
blow-up point. Then in boundary conformal normal
coordinates at $x_i$, there exist
constants $\si, C > 0$, independent of $i$, such that
\begin{gather}
|v_i - U|(y) \leq C \ve_i
 \nonumber
\end{gather}
for every $|y| \leq \si \mip$.
\label{simple_symmetry}
\end{prop}
\begin{proof}
The idea of the proof is as follows. Using the fact that the boundary is totally
geodesic in boundary conformal normal coordinates,
we can reflect all quantities across the boundary and then mimic the proofs of \cite{M}.
In order to simplify notation the index $i$ will be dropped from all quantities when no confusion arises,
and the metric $\tilde{g}$ in
$y$-coordinates will simply be denoted as $g$. Similarly $\tilde{f}$ will be denoted by $f$.
We will use Greek letters to denote indices running up to $n$ and Latin letters for
indices running up to $n-1$, and write as usual $y=(y^\prime,y^n)$. Let $l = \si \ve^{-1}$.

If $y_i \in \partial M$, take Fermi coordinates $(z^1,\dots,z^n)$ at $y_i$.
If $y_i \in \mathring{M}$ then take Fermi coordinates $(z^1,\dots,z^n)$ at
$\tilde{y}_i$, where $\tilde{y}_i \in \partial M$ is the closest point to $y_i$. Then in these coordinates
$g_{nn} \equiv 1$ and $g_{ni} \equiv 0$. Shrinking the domain if necessary, we can assume that
the domain of definition of the Fermi coordinates contains the domain of definition of the boundary
conformal normal coordinates. Define the extensions
\begin{gather}
 \overline{g}(z^\prime,z^n) =
\begin{cases}
 g(z^\prime, z^n), & z^n \geq 0 \\
 g(z^\prime, -z^n), & z^n < 0
\end{cases}
\;\;\;\; \text{and} \;\;\;\;
 \overline{v}(z^\prime,z^n) =
\begin{cases}
 v(z^\prime, z^n), & z^n \geq 0 \\
 v(z^\prime, -z^n), & z^n < 0 .
\end{cases}
\label{extension}
\end{gather}
Recall that in boundary conformal normal coordinates the mean curvature vanishes and hence
the boundary condition for $v$ is just a Neumann condition. Moreover the umbilicity of $\partial M$
gives that the second fundamental form vanishes as well. Therefore the
above extensions are $C^2$, and are in fact
smooth in the $z^\prime$ direction. Notice also that
we are performing a change of coordinates to Fermi coordinates, but we are not making a
conformal change of the metric, and hence the vanishing of $\kappa$ and $\kappa_{ij}$ are still true
in Fermi coordinates. Mimicking a standard one-dimensional argument then shows that
$\partial_n (\partial^2_n \overline{v})$ and $\partial_n (\partial^2_n \overline{g})$ exist in the weak
sense, so in particular the extensions are $C^{2,\al}$. Of course, the extended metric also satisfies
$\overline{g}_{nn} \equiv 1$ and $\overline{g}_{ni} \equiv 0$.

A simple calculation shows that
\begin{gather}
 R_{\overline{g}}(z^\prime,z^n) = R_g(z^\prime,-z^n),~~z^n < 0 ,
\label{scalar_extends}
\end{gather}
and
\begin{gather}
 \Delta_{\overline{g}} \overline{v}(z^\prime,z^n) = \Delta_{g} v(z^\prime,-z^n),~~ z^n < 0.
\label{Laplacian_v_extends}
\end{gather}

Now extend the function $f$ across the boundary by $\overline{f}(z^\prime,z^n) = f(z^\prime,-z^n)$
if $z^n < 0$. Notice that $\overline{f}$ and $R_{\bar{g}}$ are $C^{0,\al}$.
Combining (\ref{scalar_extends}) and (\ref{Laplacian_v_extends})
produces
\begin{align}
 L_{\overline{g}} \overline{v} (z^\prime,z^n) = L_g v(z^\prime,-z^n) = -Kf^{-\de}(z^\prime,-z^n) v^p(z^\prime,-z^n)
= -K\overline{f}^{-\de}(z^\prime,z^n) \overline{v}^p(z^\prime,z^n)
\nonumber
\end{align}
for $z^n<0$, i.e., the extended quantities also satisfy the equation. It follows that the extended
equation holds in the original $y$-coordinates,
\begin{gather}
L_{\overline{g}} \overline{v} (y)+ K\overline{f}^{-\de} \overline{v}^p (y) = 0 ~~\text{in}~~ \widetilde{B}_{l}(0) ,
\end{gather}
where $\widetilde{B}_{l}(0)$ is a full ball in $\RR^n$, i.e.,
$\widetilde{B}_{l}(0) = \{ y \in \RR^n ~|~|y| < l \}$. From
$\det{g} = 1 + O(r^N)$ in $B_l(0)$ we obtain $\det{\overline{g}} = 1 + O(r^N)$ in
$\widetilde{B}_l(0)$ as well.

Now that the problem is defined in the full ball $\widetilde{B}_l(0)$, to prove
the proposition, proceed with almost identical arguments as in the proofs of lemmas 5.1, 5.2 and 5.3 of \cite{M}.
There are, however, three differences that we now discuss.

First, unlike in \cite{M} the coefficients, of the PDE are not smooth.
However, they are sufficiently regular to apply elliptic estimates.

Second, we need the estimate $\overline{v} \leq C U$. In \cite{M} this arises from the fact that the
blow-up is isolated simple. In the current situation it is not necessarily
true that $0$ is an isolated simple blow-up point for $\overline{v}$ on $\widetilde{B}_l$(0).
Nevertheless, we will show
that  $\overline{v}(y) \leq C U(y)$ still holds for all $y \in \widetilde{B}_l(0)$.
Notice that we do not need to make an extension of $U$
since it is a priori defined on the whole of $\RR^n$.

To see why this is the case, first notice that since $y_i$ is an isolated simple blow-up point
for $v$ on $B_l(0)$, we have $v \leq C U$ there.
For $p \in B_l(0)$, let $\overline{p} \in \widetilde{B}_l(0) \backslash \overline{B_l(0)}$,
be the reflected point.
If $y_i \in \partial M$ then $d_{\overline{g}}(y_i,\overline{p}) = d_{\overline{g}}(y_i,p)$,
where $d_{\overline{g}}$ means $\operatorname{dist}_{\overline{g}}$.
If $y_i \notin \partial M$,
then in $y$ coordinates the boundary is given by a graph $y^n = F(y^\prime)$, but
$F(y^\prime) \rar 0$ as $i \rar 0$ (see (\ref{limit_T_0}) and lemma \ref{limit_hyperplane}),
which then implies $d_{\overline{g}}(y_i,\overline{p}) = d_{\overline{g}}(y_i,p) + o(1)$. Therefore
\begin{gather}
\label{v_U_simple_blow_up_extended}
\overline{v}(\overline{p}) =v(p)
 \leq C U(p) = C (1 + d_{\overline{g}}(y_i,p)^2)^\frac{2-n}{2}
\leq C_1 (1 + d_{\overline{g}}(y_i,\overline{p})^2)^\frac{2-n}{2}
= C_1 U(\overline{p}),
\nonumber
\end{gather}
as desired.

Finally, the third difference with \cite{M} is that there, the scalar curvature
satisfies $R_g=O(r^2)$, which comes from the Taylor formula and properties of conformal normal
coordinates. Here, since $R_{\overline{g}}$ is $C^{0,\al}$ only, we avoid the Taylor expansion. Without
$R_{\overline{g}}=O(r^2)$ the proof in \cite{M} yields a weaker estimate, but since we only need
$|v - U|(y) \leq C \ve$, the hypothesis
$R_{\overline{g}}=O(r^2)$ is not necessary. In \cite{M} the better estimate
$|v - U|(y) \leq C \ve^s$, with $s>1$, is established.
\end{proof}
\begin{rema}
Observe that as in \cite{M}, the proof of proposition \ref{simple_symmetry} produces
the estimate $\de_i \leq C \ve_i$.
\label{estimate_de_i_simple}
\end{rema}

In the proof of the next proposition, we retain the notation for the reflected quantities that appears in the proof of proposition
\ref{simple_symmetry}.

\begin{prop} Under the same hypotheses of proposition
\ref {simple_symmetry}, there exists a constant $C_0$, independent
of $i$, such that
\begin{gather}
 \parallel v_i - U \parallel_{C^{2,\al}(\widetilde{B}_{\frac{l_i}{4}}(0))} \leq C_0 \ve_i ,
\nonumber
\end{gather}
where $l_i = \si \ve_i^{-1}$.
\label{simple_symmetry_C_2_al}
\end{prop}
\begin{proof}
It is sufficient to establish the desired estimate for $w_i = \overline{v}_i - U$. We have
\begin{gather}
 L_{\tilde{g}_i} w_i + b_i w_i = Q_i \text{ in } \widetilde{B}_{l_i}(0)
\nonumber
\end{gather}
with
\begin{align}
 b_i(y) &= K\overline{f}^{-\delta_i} \frac{\overline{v}_i^{p_i} - U^{p_i}}{\overline{v}_i - U
}(y), \nonumber \displaybreak[0] \\
Q_i(y) &= \Big ( c(n) \varepsilon_i^2  R_{g_i} (\varepsilon_i y)U(y) +
K ( U^\frac{n+2}{n-2} - \overline{f}^{-\delta_i}U^{p_i} )
+ \ve_i^{N+1} O(|y|^N)|y|(1+|y|^2)^{-\frac{n}{2}}
\Big ). \nonumber
\end{align}
Use
(\ref{v_U_simple_blow_up_extended}) to find
$|b_i(y)| \leq c(1+|y|)^{-4}$.
Then the representation formula gives, for any and $|y| \leq \frac{l_i}{4}$,
\begin{gather}
 w_i(y) = \int_{\widetilde{B}_{l_i}(0)} G_i(y,z) (b_i w_i - Q_i)(z) dz
- \int_{\partial \widetilde{B}_{l_i}(0)} \frac{\partial G_i}{\partial \nu_{\tilde{g}_i}}(y,z)w_i(z) dS(z) ,
\label{rep_Newtonian_analog}
\end{gather}
where $G_i$ is the Green's function for the conformal Laplacian with Dirichlet boundary condition.
The proof is now similar to standard estimates for the Newtonian potential,
and therefore we will only indicate
the main steps (see for example \cite{GT}).

First notice that unlike the Newtonian potential case, there is a boundary
integral in the representation formula (\ref{rep_Newtonian_analog}). Nevertheless, this boundary
integral is easily estimated using standard properties of the Green's function and $\overline{v}_i \leq C U$,
since the singularities occur within the radius $\frac{l_i}{4}$.

For the interior integral, write $\ga_i = b_i w_i - Q_i$. This quantity
plays the role of the inhomogeneous term in potential theory.
Therefore standard potential theoretic arguments yield
\begin{gather}
 [D^2 w_i]_{\al,\widetilde{B}_{\frac{l_i}{4}}(0)} \leq \frac{C}{l_i^\al}
\Big ( \parallel \ga_i \parallel_{C^0(\widetilde{B}_{\frac{l_i}{2}}(0))} +
l_i^\al [\ga_i]_{\al,\widetilde{B}_{\frac{l_i}{2}}(0)} \Big )
\label{potential_estimate_1}
\end{gather}
where $[\cdot]_{\al,\Om}$ is the H\"older semi-norm on $\Om$.
Next, observe that by interpolation
\begin{align}
 [\ga_i]_{\al,\widetilde{B}_{\frac{l_i}{2}}(0)} &
\leq C \Big ([b_i]_{\al,\widetilde{B}_{\frac{l_i}{2}}(0)}[w_i]_{\al,\widetilde{B}_{\frac{l_i}{2}}(0)} +
[Q_i]_{\al,\widetilde{B}_{\frac{l_i}{2}}(0)} \Big ) .
\label{estimate_ga_i_potential}
\end{align}
In order to estimate $[w_i]_{\al,\widetilde{B}_{\frac{l_i}{2}}(0)}$
the representation formula (\ref{rep_Newtonian_analog}) may again be employed along with standard
properties of $G_i$ and proposition \ref{simple_symmetry}.
However, control of the boundary term relies on $y$ staying away from the boundary, that is why we choose
an estimate on $\widetilde{B}_{\frac{l_i}{2}}(0)$ (giving then a final estimate on
$\widetilde{B}_{\frac{l_i}{4}}(0)$). Moreover, using remark \ref{estimate_de_i_simple} and
$U^\frac{n+2}{n-2} - \overline{f}^{-\delta_i}U^{p_i} = U^\frac{n+2}{n-2}O( (|\log f| + |\log U|)\de_i )$,
it follows that
\begin{align}
[Q_i]_{\al,\widetilde{B}_{\frac{l_i}{2}}(0)} & \leq
C \Big ( \varepsilon_i^2 [ R_{g_i} (\varepsilon_i y)U(y) ]_{\al,\widetilde{B}_{\frac{l_i}{2}}(0)}
+ [ U^\frac{n+2}{n-2} - \overline{f}^{-\delta_i}U^{p_i} ]_{\al,\widetilde{B}_{\frac{l_i}{2}}(0)}
\label{estimate_Q_i_potential} \\
& + \ve_i [ \ve_i^N O(|y|^N)|y|(1+|y|^2)^{-\frac{n}{2}} ]_{\al,\widetilde{B}_{\frac{l_i}{2}}(0)}  \Big )
\nonumber \\
& \leq C \ve_i .
\nonumber
\end{align}
Finally, the term $\parallel \ga_i \parallel_{C^0(\widetilde{B}_{\frac{l_i}{2}}(0))}$ is estimated in
a similar manner
\begin{gather}
\parallel \ga_i \parallel_{C^0(\widetilde{B}_{\frac{l_i}{2}}(0))} \leq C \ve_i .
\nonumber
\end{gather}
Combining this
with (\ref{potential_estimate_1}), (\ref{estimate_ga_i_potential}) and (\ref{estimate_Q_i_potential}) yields
$[D^2 w_i]_{\al,\widetilde{B}_{\frac{l_i}{4}}(0)} \leq C\ve_i$.
The remaining lower order terms of the $C^{2,\al}$ norm may be estimated in an analogous way.
\end{proof}

The analogous of the following result is already known for scalar-flat manifolds \cite{Al}.

\begin{theorem} Suppose $x_i \rar \bar{x}$ is an isolated simple
blow-up point. Then in boundary conformal normal
coordinates at $x_i$, for all $i$ sufficiently large and possibly after passing to
a subsequence, we have
$x_i \in \partial M$.
\label{blow_up_on_boundary}
\end{theorem}
\begin{proof}
The proof is by contradiction. Therefore assume that $x_i \in \partial M$ occurs only for finitely
many $i$. Hence passing to a subsequence,
still denoted $x_i$, we can assume that
\begin{gather}
 x_i \in \mathring{M}~~~\text{for all}~i .
\label{assume_x_i_interior}
\end{gather}

Take boundary conformal normal coordinates at $x_i$ (see corollary \ref{bcnc_int}),
rescale all quantities to $y$ coordinates as explained at the beginning of the text,
and denote by $\tilde{y}_i \in \partial M$ the closest point to $y_i$, where $y_i$ is identified
with the origin. The closure
of the ball of radius $|\tilde{y}_i|$ will be denoted by $\overline{B_{|\tilde{y}_i|}(0)}$.
Furthermore, for any domain $\Om$, denote
by $[\cdot]_{1+\al,\Om}$ the $C^{1,\al}$ H\"older semi-norm, and by $[\cdot]_{1,\Om}$ the
$C^1$ H\"older semi-norm.

Let $w_i = v_i - U$, then
\begin{gather}
 \frac{| \partial_n( v_i - U)(0) - \partial_n(v_i - U)(\tilde{y}_i)  |}{|\tilde{y}_i - 0|^\beta}
\leq [w_i]_{1+\beta,\overline{B_{|\tilde{y}_i|}(0)}} .
\label{estimate_dist_1}
\end{gather}
As explained in section \ref{boundary_estimates}, the coordinates
may be arranged such that
$\left.\frac{\partial }{\partial \nu_{\tilde{g}_i}} \right|_{\tilde{y}_i} = g^{nn} \left. \partial_n \right|_{\tilde{y}_i}$.
Observe that the boundary condition for $v_i$ implies
that $\partial_n v_i(\tilde{y}_i) = 0$, since the mean curvature vanishes.
Notice also that
we have $\nabla v_i(0) = \nabla U(0) = 0$. On the other hand a direct calculation gives
\begin{gather}
\partial_n U(\tilde{y}_i) = (2-n) (1+|\tilde{y}_i|^2)^{-\frac{n}{2}}\tilde{y}_i^n.
\label{partial_n_U_y_tilde}
\end{gather}
Hence
(\ref{estimate_dist_1}) becomes
\begin{align}
|\tilde{y}_i^n| = |\tilde{y}_i| \leq \frac{1}{n-2}(1+|\tilde{y}_i|^2)^{\frac{n}{2}}
|\tilde{y}_i|^\beta  [w_i]_{1+\beta,\overline{B_{|\tilde{y}_i|}(0)} },
\nonumber
\end{align}
since $\tilde{y}_i = (0,\dots,0,\tilde{y}_i^n)$. By (\ref{limit_T_0}),
$\frac{1}{n-2}(1+|\tilde{y}_i|^2)^{\frac{n}{2}}  \leq C_1$ for a constant $C_1$ independent of $i$, so
\begin{align}
|\tilde{y}_i| \leq C_1 |\tilde{y}_i|^\beta  [w_i]_{1+\beta,\overline{B_{|\tilde{y}_i|}(0)} }.
\label{estimate_dist_2}
\end{align}
By proposition \ref{simple_symmetry_C_2_al}, $w_i$ converges to zero in $C^{2,\al}$, so there exists a small
$r>0$, independent of $i$, such that the Taylor formula for $w_i$ holds in $B_r(0)$ for all $i$. By
(\ref{limit_T_0}) we can assume that $\overline{B_{|\tilde{y}_i|}(0)} \subset B_r(0)$. Therefore for any
$y \in \overline{B_{|\tilde{y}_i|}(0)}$,
\begin{gather}
\partial_k w_i(y) = \partial_k w_i(0) + \mathcal{R}_l(y) y^l = \mathcal{R}_l(y) y^l ,
\nonumber
\end{gather}
where we used $\nabla w_i(0) = 0$. The remainder term satisfies, for each $l=1,\dots,n$,
\begin{gather}
 |\mathcal{R}_l(y)| \leq \sup_{z \in \overline{B_{|\tilde{y}_i|}(0)}} |\nabla^2 w(z)|
\leq \parallel w_i \parallel_{C^{2,\al}(\overline{B_{|\tilde{y}_i|}(0)})} \leq C_0 \ve_i,
\nonumber
\end{gather}
where proposition \ref{simple_symmetry_C_2_al} has been used.
Hence $|\partial_k w_i(y)| \leq |\mathcal{R}_l(y)y^l| \leq n C_0 \ve_i |y|$, and therefore
\begin{gather}
 [w_i]_{1,\overline{B_{|\tilde{y}_i|}(0)}} \leq nC_0 \ve_i |\tilde{y}_i|.
\label{estimate_w_1_ve_dist}
\end{gather}

Let $\Om$ be a convex domain. The following inequality is standard (see e.g. \cite{GT})
\begin{gather}
 [ u ]_{1+\beta,\Om} \leq \La \rho^{\al-\beta}[u]_{1+\al,\Om} + \La \rho^{-\beta} [u]_{1,\Om} ~,
\label{interpolation_1}
\end{gather}
where the constant $\La$ depends only on the dimension,
$0 < \beta < \al < 1$,
and $\rho >0$ is any positive number.
Also, using the mean value inequality,
there is a constant $A$ depending only on the dimension, such that
$|\partial_i u(p) - \partial_i u(q) | \leq A |p-q|[u]_{2,\Om} = A |p-q|^\al |p-q|^{1-\al} [u]_{2,\Om}$.
From this it follows that
\begin{gather}
 [u]_{1+\al,\Om} \leq A \operatorname{diam}(\Om)^{1-\al} [u]_{2, \Om} .
\label{interpolation_5}
\end{gather}

Because the constants $C_0,C_1,A$ and $\La$ do
not depend on $i$, we can, with the help of (\ref{limit_T_0})
and the definition of $\ve_i$,
choose $i$ so large that
\begin{gather}
 |\tilde{y}_i| < \frac{1}{2} ,
\label{choice_i_1} \\
\ve_i < 1,
\label{choice_i_2} \\
\max \big \{ C_0 C_1 \La A, ~n C_0 C_1 \La \big \}  \ve_i^\frac{\al^2}{p} < \frac{1}{2} ,
\label{choice_i_3} \\
C_0 C_1 \ve_i^{1-\al} < 1,
\label{choice_i_4}
\end{gather}
where $p>1$ is a large number chosen such that
\begin{gather}
 \al\frac{p+1}{p} < 1.
\label{choice_p}
\end{gather}
This is possible since $\al < 1$; notice
that $p$ does not depend on $i$.

Now fix an $i_0=i_0(n,C_0,C_1,A,\La)$ such that (\ref{choice_i_1})-(\ref{choice_i_4}) hold.
From (\ref{interpolation_5}) we have
\begin{gather}
 [w_{i_0}]_{1+\al,\overline{B_{|\tilde{y}_{i_0}|}(0)}} \leq A  [w_{i_0}]_{2,\overline{B_{|\tilde{y}_{i_0}|}(0)}} .
\label{interpolation_5_b}
\end{gather}
Moreover the constants $C_0,C_1,A$ and $\La$ do not depend on the choice of $\beta$, as can be seen
from the derivation of inequalities (\ref{estimate_dist_2}), (\ref{interpolation_1}), (\ref{interpolation_5}),
and the proof of proposition \ref{simple_symmetry}. Therefore the inequalities
(\ref{estimate_dist_2})-(\ref{interpolation_5_b}) hold
for any $\beta$ such that
$0 < \beta < \al$.

We are now in a position to prove the theorem. It will show by induction that
\begin{gather}
|\tilde{y}_{i_0}| \leq \ve_{i_0}^{k\al}
\label{induction_hypothesis_x_i_boundary}
\end{gather}
for all $k=0,1,2,3,\dots$. Since $\ve_{i_0} < 1$ this
would imply $|\tilde{y}_{i_0}| = 0$ so that $x_{i_0} \in \partial M$, contradicting
(\ref{assume_x_i_interior}).

For $k=0$ (\ref{induction_hypothesis_x_i_boundary}) is true by
(\ref{choice_i_1}). For $k=1$, recall that $\partial_n v_{i_0}(\tilde{y}_{i_0}) = 0$, and observe
that
\begin{gather}
|\partial_n U(\tilde{y}_{i_0})|  = |\partial_n(v_i - U)(\tilde{y}_{i_0}) | \leq
[w]_{1,\overline{B_{|\tilde{y}_{i_0}|}(0)} }.
\nonumber
\end{gather}
Then (\ref{partial_n_U_y_tilde}), proposition \ref{simple_symmetry} and (\ref{choice_i_4})
give
\begin{gather}
|\tilde{y}_{i_0}| \leq C_1 [w_{i_0}]_{1,\overline{B_{|\tilde{y}_{i_0}|}(0)} } \leq C_0 C_1 \ve_{i_0}
= C_0 C_1 \ve_{i_0}^{1-\al} \ve_{i_0}^\al < \ve_{i_0}^\al .
\nonumber
\end{gather}
So assume that (\ref{induction_hypothesis_x_i_boundary}) holds for some $k\geq 1$.
Combining (\ref{estimate_dist_2}) and (\ref{interpolation_1}) gives
\begin{align}
 |\tilde{y}_{i_0}| \leq C_1 |\tilde{y}_{i_0}|^\beta [w_{i_0}]_{1+\beta,\overline{B_{|\tilde{y}_{i_0}|}(0)} }
\leq C_1 \La |\tilde{y}_{i_0}|^\beta \Big ( \rho^{\al-\beta}[w_{i_0}]_{1+\al,\overline{B_{|\tilde{y}_{i_0}|}(0)}}
+ \rho^{-\beta} [w_{i_0}]_{1,\overline{B_{|\tilde{y}_{i_0}|}(0)}} \Big ) .
\label{ineq_x_i_boundary_proof_1}
\end{align}
Choose $\beta = \frac{\al}{pk}$ (which is less than $\al$ by the choice of $p$). If we also choose
$\rho = \ve_{i_0}^k$ then (\ref{ineq_x_i_boundary_proof_1}) becomes
\begin{align}
 |\tilde{y}_{i_0}|
\leq C_1 \La |\tilde{y}_{i_0}|^\frac{\al}{pk} \Big ( \ve_{i_0}^{k\al-\frac{\al}{p}} [w_{i_0}]_{1+\al,\overline{B_{|\tilde{y}_{i_0}|}(0)}}
+ \ve_{i_0}^{-\frac{\al}{p}} [w_{i_0}]_{1,\overline{B_{|\tilde{y}_{i_0}|}(0)}} \Big ) .
\nonumber
\end{align}
By (\ref{estimate_w_1_ve_dist}), (\ref{interpolation_5_b}),
proposition \ref{simple_symmetry}, the induction hypothesis
(\ref{induction_hypothesis_x_i_boundary})
and the fact that
\begin{gather}
[w_{i_0}]_{2,\overline{B_{|\tilde{y}_{i_0}|}(0)}} \leq
\parallel w_{i_0} \parallel_{C^{2,\al}(\overline{B_{|\tilde{y}_{i_0}|}(0)}) },
\nonumber
\end{gather}
we obtain
\begin{align}
 |\tilde{y}_{i_0}|
& \leq
\max \big \{ C_0 C_1 \La A, ~n C_0 C_1 \La \big \}
\ve_{i_0}^\frac{\al^2}{p}  \Big ( \ve_{i_0}^{k\al + 1 -\frac{\al}{p}}
+ \ve_{i_0}^{-\frac{\al}{p} +1 + k\al}  \Big )
\nonumber \\
& =
2 \max \big \{ C_0 C_1 N A, ~n C_0 C_1 N \big \} \ve_{i_0}^\frac{\al^2}{p}
  \ve_{i_0}^{(k+1)\al} \ve_{i_0}^{1- \al -\frac{\al}{p}}
\nonumber \\
& \leq \ve_{i_0}^{(k+1)\al} \ve_{i_0}^{1- \al -\frac{\al}{p}} .
\nonumber
\end{align}
where (\ref{choice_i_3}) has been employed.
Finally, $\ve_{i_0}^{1- \al -\frac{\al}{p}} < 1$ by (\ref{choice_i_2}) and (\ref{choice_p}).
\end{proof}

\section{Symmetry estimates\label{sec_symmetry_estimates}}

In this section we derive sharp estimates for the behavior of solutions $u_i$ in the
neighborhood of an isolated simple
blow-up point. The proofs are an adaptation of the results of \cite{KMS} and we will
often refer the reader to it for details.

Throughout this section, let $x_i \rar \bar{x} \in \partial M$ be an isolated simple
blow-up point. By theorem \ref{blow_up_on_boundary} we can assume that $x_i \in \partial M$.
 We will be using boundary conformal normal coordinates at $x_i$ (see proposition
\ref{bcnc})
 and rescale all the quantities to $y$-coordinates as explained at the beginning of
the text.
Notice that because in boundary conformal normal coordinates we have $\kappa_g = 0$
in the neighborhood of the origin, the boundary condition becomes a Neumann condition.
Moreover, since the boundary is
umbilic we obtain that
it is totally geodesic in the neighborhood of the origin.

Also, by (\ref{partial_hyper_restricted}), proposition \ref{bvp_z} gives that
for $|y| \leq \si \ve_i^{-1}$ we have, with $\tilde{z}_i = \tilde{z}_{\ve_i}$,
\begin{align}
\begin{cases}
 \Delta \tilde{z}_{i} + n(n+2)U^\frac{4}{n-2}\tilde{z}_{i} =
c(n)\sum_{k=4}^{n-4} \partial_i\partial_j \tilde{H}^{(k)}_{ij} U
, & \text{ for } |y| \leq \si \ve_i^{-1}, \displaybreak[0] \\
\frac{ \partial \tilde{z}_{i}}{\partial \nu_{\tilde{g}_i}} = \sum_{l=1}^{n-1} g^{nl}\partial_l \tilde{z}_i
= \ve_i^N O(|y^\prime|^N (1+|y^\prime|)^{1-n}), & \text{ on } \partial M
\end{cases}
\label{bvp_z_ball}
\end{align}
where (\ref{bound_z_i_similar_U}) has also been used.
Notice that since $\partial_n U(y^\prime,0) = 0$, in these coordinates $U$ also satisfies the boundary condition
\begin{gather}
\frac{ \partial U}{\partial \nu_{\tilde{g}_i}} = \sum_{l=1}^{n-1} g^{nl}\partial_l U
= \ve_i^N O(|y^\prime|^N (1+|y^\prime|)^{1-n}) \text{ on } \partial M .
\label{bc_U}
\end{gather}

\begin{prop} Suppose $x_i \rar \bar{x}$ is an isolated simple
blow-up point. Then in boundary conformal normal
coordinates at $x_i$, there exist
constants $\si, C > 0$ such that
\begin{gather}
|v_i - U -\tilde{z}_i| \leq C \max_{2\leq k \leq d-1}
\{\varepsilon_i^{2k}|H^{(k)}|^2(x_i),\varepsilon_i^{n-3},\de_i \} \nonumber
\end{gather}
for every $|y| \leq \si \mip$.
\label{symmetry_1}
\end{prop}
\begin{proof}
Put $\La_i = \max_{|y| < l_i} | v_i - U -\tilde{z}_i |= | v_i - U -\tilde{z}_i
|(y_i)$. Then as in the boundaryless case we get a
 stronger inequality if there exists a constant $c$ such that $|y_i| \geq cl_i$ for
every $i$. In fact, using that
$\bar{x}$ is an isolated simple blow-up point we get the inequality $v \leq C U \leq
C|y|^{2-n}$, and using
estimate (\ref{bound_z_i_similar_U})
 we get $\La_i =  | v_i - U -\tilde{z}_i |(y_i) \leq C |y_i|^{2-n} \leq
\varepsilon_i^{n-2}$.
Hence we can assume $|y_i| \leq \frac{l_i}{2}$.

If the proposition is false we have
\begin{gather}
  \frac{1}{\La_i} \max_{2\leq k \leq d-1} \{\varepsilon_i^{2k}|H^{(k)}|^2(x_i) \}
\rar 0,~~
\frac{1}{\La_i}\varepsilon_i^{n-3} \rar 0,~~~ \frac{1}{\La_i} \de_i \rar 0.
\label{if_false}
\end{gather}

Define
\begin{gather}
w_i(y) = \frac{1}{\La_i}( v_i - U - \tilde{z}_i )(y).  \nonumber
\end{gather}
Then $|w_i(y)| \leq 1$, and
\begin{gather}
\begin{cases}
 L_{\tilde{g}_i} w_i + b_i w_i = Q_i & \text{ in } B_{l_i}(0)  \displaybreak[0] \\
w_i = O(\La^{-1}\ve_i^{n-2}) & \text{ on } \partial^+B_{l_i}(0) \displaybreak[0] \\
\frac{\partial w_i}{\partial \nu_{\tilde{g}_i}} = \La^{-1} \ve_i^N O(|y^\prime|^N (1+|y^\prime|)^{1-n}) & \text{ on } \partial^\prime B_{l_i}(0),
 \end{cases}
\label{bvp_w_i}
\end{gather}
where (\ref{bvp_z_ball}), (\ref{bc_U}) and the boundary condition
for $v_i$ have been used; $Q_i$ and $b_i$ are as in the boundaryless case
\begin{align}
 b_i(y) &= K\tilde{f}^{-\delta_i} \frac{v_i^{p_i} - (U + \tilde{z}_i)^{p_i}}{v_i - U
-\tilde{z}_i}(y), \nonumber \displaybreak[0] \\
Q_i(y) &= \frac{1}{\La_i} \Big \{ c(n) \varepsilon_i^2 \big ( R_{g_i} -
\sum_{\ell=2}^{n-6}(\partial_j\partial_k  H_{jk})^{(\ell)} \big ) (\varepsilon_i
y)U(y) +
(\Delta - L_{\tilde{g}_i} )(\tilde{z}_i) \nonumber \displaybreak[0] \\
 &+ O(|\tilde{z}_i|^2 U^\frac{6-n}{n-2} ) +
K ( ( U+\tilde{z}_i)^\frac{n+2}{n-2} - \tilde{f}^{-\delta_i}(U+\tilde{z})^{p_i} )
\nonumber \displaybreak[0] \\
&+ M_i^{-(1+N)\frac{p_i -1 }{2} } O(|y|^N)|y|(1+|y|^2)^{-\frac{n}{2}}
\Big \}, \nonumber
\end{align}
and they satisfy the estimates (see \cite{KMS})
\begin{align}
 |b_i(y)| \leq c(1+|y|)^{-4} \nonumber \displaybreak[0]
\end{align}
and
\begin{align}
|Q_i(y) | &\leq C\frac{1}{\La_i} \Big \{ \max_{2\leq k \leq d-1}
\{\varepsilon_i^{2k}|H^{(k)}|^2(x_i) \}
(1+|y|)^{2d-2-n}  \nonumber \displaybreak[0] \\
& +\varepsilon_i^{n-3}(1+|y|)^{-3} +
 M_i^{-(1+N)\frac{p_i -1 }{2} } O(|y|^N)|y|(1+|y|^2)^{-\frac{n}{2}} \nonumber
\displaybreak[0] \\
&+ \de_i( |\log(U+\tilde{z}_i)| + |\log\tilde{f}_i|)(1+|y|)^{-n-2} \Big \}. \nonumber
\end{align}

Let $G_i$ be the Green's function for the conformal Laplacian with boundary condition
$G_i=0$ on $\partial^+ B_{l_i}(0)$ and $B_{\tilde{g}_i} G_i = \frac{\partial G_i}{\partial \nu_{\tilde{g}_i}} = 0$
on $\partial^\prime B_{l_i}(0)$. The representation formula then gives
\begin{align}
w_i(y) &= \int_{B_{l_i}(0)} G_i(y,\eta)(b_i w_i - Q_i)(\eta)d\eta -
 \int_{\partial^+ B_{l_i}(0)} w_i(\eta)\frac{\partial G_i(y,\eta)}{\partial
\nu_{\tilde{g}_i}}dS(\eta) \nonumber  \displaybreak[0]  \\
&+ \int_{\partial^\prime B_{l_i}(0)} G_i(y,\eta) \frac{\partial w_i(\eta) }{\partial
\nu_{\tilde{g}_i}}dS(\eta)
\nonumber
\end{align}
for $|y| \leq \frac{l_i}{2}$. The first two integrals are estimated as in the
boundaryless case (see \cite{KMS}). For the third one we use (\ref{bvp_w_i}) to find
\begin{gather}
\Big | \int_{\partial^\prime B_{l_i}(0)} G_i(y,\eta) \frac{\partial w_i(\eta) }{\partial
\nu_{\tilde{g}_i}} d \eta^\prime \Big | \leq C \ve_i^{n-2} .
\end{gather}
Hence,
\begin{gather}
|w_i(y)| \leq C \Big ( (1+|y|)^{-2} + \frac{1}{\La_i}
\max_{2\leq k \leq d-1}
\{\varepsilon_i^{2k}|H^{(k)}|^2(x_i),\varepsilon_i^{n-3},\de_i \} \Big ).
\label{big_estimate_theta_i}
\end{gather}
It then follows from (\ref{if_false}), (\ref{big_estimate_theta_i}),
and standard elliptic estimates
that $w_i$ is bounded in $C^2_{loc}$ and has a subsequence, still denoted $w_i$,
converging to a limit
$w_\infty$, which satisfies
\begin{gather}
 \begin{cases}
  \Delta w_\infty + n(n+2) U^\frac{4}{n-2} w_\infty = 0 & \text{ on } \RR^n_+
\displaybreak[0] \\
\frac{\partial w_\infty}{\partial y^n} = 0 & \text{ on } \RR^{n-1} \displaybreak[0] \\
\lim_{|y|\rar\infty} w_\infty(y) = 0.
 \end{cases}\nonumber
\end{gather}
Note that for the boundary condition we used that
$\La^{-1} \ve_i^N |y^\prime|^N (1+|y^\prime|)^{1-n} \leq C\ve_i^{n-3}$
for $|y^\prime| \leq \si \ve_i^{-1}$. Lemma \ref{class_sol} then gives
\begin{gather}
w= c_0 \Big( \frac{n-2}{2} U(y) + y \cdot \nabla U \Big ) + \sum_{j=1}^{n-1} c_j
\partial_j U.  \nonumber
\end{gather}
However $w_i(0) = |\nabla w_i|(0) = 0$ implies $w_\infty(0) = |\nabla w_\infty|(0)=0$,
from which we conclude that $w_\infty \equiv 0$.
It then follows that $|y_i| \rar \infty$.
This combined with (\ref{if_false}) contradicts (\ref{big_estimate_theta_i}), as
$w_i(y_i) = 1$.
\end{proof}

The proofs of the next two results are similar to those in \cite{KMS}, making the
necessary adaptations to the boundary case
with ideas described in proposition \ref{symmetry_1}.

\begin{prop} Under the hypotheses of proposition \ref{symmetry_1},
\begin{gather}
\delta_i \leq C \max_{2\leq k \leq d-1}
\{\varepsilon_i^{2k}|H^{(k)}|^2(x_i),\varepsilon_i^{n-3} \} \nonumber
\end{gather}
for every $|y| \leq \si \mip$.
\label{symmetry_2}
\end{prop}
\begin{proof}
If the proposition is false we have
\begin{gather}
  \frac{1}{\de_i} \max_{2\leq k \leq d-1} \{\varepsilon_i^{2k}|H^{(k)}|^2(x_i) \}
\rar 0,~~
\frac{1}{\de_i}\varepsilon_i^{n-3} \rar 0.
\label{if_false_2}
\end{gather}
 Hence from proposition \ref{symmetry_1},
\begin{gather}
 |v_i - U - \tilde{z}_i|(y) \leq C\de_i. \nonumber
\end{gather}
Define
\begin{gather}
 w_i = \frac{1}{\de_i}(v_i - U - \tilde{z}_i), \nonumber
\end{gather}
and argue as in proposition \ref{symmetry_1}, with $\de_i$ replacing $\La_i$,
to obtain $w_i \rar w_\infty$ in $C^2_{loc}$, where $\partial_n w_\infty = 0$ on
$\RR^{n-1}$. Define
$\Psi(y) = \frac{n-2}{2} U(y) + y^j\partial_jU(y)$. Now we argue as in \cite{KMS},
except possibly for the extra boundary terms
\begin{gather}
 \int_{\partial^\prime B_{\frac{l_i}{2}}(0) } \Psi \frac{\partial w_i}{\partial
\nu_{\tilde{g}_i}} ~~~\text{ and }~~~
 \int_{\partial^\prime B_{\frac{l_i}{2}}(0) } w_i \frac{\partial \Psi}{\partial
\nu_{\tilde{g}_i} }.  \nonumber
\end{gather}
But as before, $\frac{\partial w_i}{\partial
\nu_{\tilde{g}_i}} = \ve_i^N O(|y^\prime|^N (1+|y^\prime|)^{1-n})$, and a direct computation gives
$\frac{\partial \Psi}{\partial \nu_{\tilde{g}_i}} = \ve_i^N O(|y^\prime|^N (1+|y^\prime|)^{1-n})$,
which is enough to handle the boundary integrals as in proposition \ref{symmetry_1}.
\end{proof}

\begin{prop} Under the hypotheses of proposition \ref{symmetry_1},
\begin{gather}
|\nabla^m(v_i - U - \tilde{z}_i )|(y) \leq C \sum_{k=2}^{d-1}
\varepsilon_i^{2k}|H^{(k)}|^2(x_i)(1+|y|)^{2k+2-n-m}
+\varepsilon_i^{n-3}(1+|y|)^{-1-m} \nonumber
\end{gather}
for every $|y| \leq \si \varepsilon^{-1}$, $m=0,1,2$.
\label{symmetry_3}
\end{prop}
\begin{proof}
Arguing similarly to \cite{KMS} with the necessary modifications as in propositions
\ref{symmetry_1} and \ref{symmetry_2}, we obtain the result with $m=0$. To obtain
the result for the derivatives, we invoke standard elliptic theory, which
gives the estimate provided that we can bound the $C^{1,\al}$ norm of $\partial_{\nu_{\tilde{g}_i}} (v_i - U - \tilde{z}_i)$
on the boundary. Since $\partial_{\nu_{\tilde{g}_i}} v_i = 0$ and
$\left.\partial_n \tilde{z}_i\right|_{y^n = 0} = 0 = \left.\partial_n U \right|_{y^n = 0} $, it is enough to show that
\begin{gather}
\parallel \sum_{l=1}^{n-1} g^{nl}\partial_l (\tilde{z}_i + U) \parallel_{C^{1,\al}(\partial^\prime B_{l_i}(0))}
\leq C \ve_i^{n-3} .
\label{C_1_alpha_bound_boundary}
\end{gather}
From (\ref{bound_z_i_similar_U}), (\ref{bvp_z}), properties
of boundary conformal normal coordinates (in particular corollary \ref{boundary_hyper}) and the explicit form of $U$ we have
\begin{gather}
\Big | \sum_{l=1}^{n-1} g^{nl}\partial_l (\tilde{z}_i + U)(y^\prime,0) \Big |
\leq C \ve_i^N |y^\prime|^N (1+|y^\prime|)^{1-n} ,
\nonumber
\end{gather}
which is bounded by $C\ve_i^{n-3}$ for $|y^\prime| \leq \si \ve_i^{-1}$.

Differentiating $\sum_{l=1}^{n-1} g^{nl}\partial_l (\tilde{z}_i + U)$ with respect to $y^k$, $k\leq n-1$,
using again (\ref{bound_z_i_similar_U}), (\ref{bvp_z}),  and properties
of boundary conformal normal coordinates yields
\begin{gather}
\Big | \partial_k ( \sum_{l=1}^{n-1} g^{nl}\partial_l (\tilde{z}_i + U) ) (y^\prime,0) \Big |
\leq C \ve_i^{n-3}~\text{ for }~|y^\prime| \leq \si \ve_i^{-1} .
\nonumber
\end{gather}
Differentiating again and repeating the argument gives (\ref{C_1_alpha_bound_boundary}). Now the pointwise
estimate follows by standard arguments.
\end{proof}

\section{Weyl vanishing\label{weyl_vanishing_section}}

In this section we will work mostly in $x$-coordinates and take boundary
conformal normal coordinates at $x_i$.
In these coordinates, estimate (\ref{basic_z_i_tilde_estimate}) and the estimate of
proposition \ref{symmetry_3} become,
for $|x|\leq \si$,
\begin{align}
 |\nabla^m z_i(x)| &\leq \varepsilon_i^\frac{n-2}{2}
 \sum_{|\al|=4}^{n-4} \sum_{jl}|h_{jl,\al}|(\varepsilon_i + |x|)^{|\al|+2-n-m}
\label{z_est_x} \displaybreak[0] \\
|\nabla^m(u_i - u_{\varepsilon_i} - z_i)(x)| & \leq C
\varepsilon_i^{\frac{n-2}{2}} \sum_{k=2}^{d-1}
|H^{(k)}|^2(x_i)(\varepsilon_i+|x|)^{2k+2-n-m}
+\varepsilon_i^\frac{n-2}{2}(\varepsilon_i+|x|)^{-m-1}, \label{sym_est_x}
\end{align}
where both $z_i$ and the sum with $|H^{(k)}|^2(x_i)$ appear only when $n \geq 6$, and
\begin{align}
 z_{\varepsilon_i} &= z_i(x) = \varepsilon^\frac{2-n}{2}
\tilde{z}_i(\varepsilon_i^{-1} x) \nonumber \displaybreak[0] \\
u_{\varepsilon_i}(x) & = \varepsilon_i^\frac{n-2}{2}( \varepsilon_i^2 +
|x|^2)^\frac{2-n}{2}. \nonumber
\end{align}
Throughout this section it will be assumed that $(M^n,g)$ is a Riemannian manifold of dimension
$3 \leq n \leq 24$ with umbilic boundary.
The index $i$ will be dropped from all quantities in
several estimates below. Note also that by theorem \ref{blow_up_on_boundary} we can assume
that $x_i \in \partial M$, therefore the boundary is given by
$\partial M = \{ x^n = 0 \}$.
We will use the notation $B_\rho^+ = \{ x \in B_\rho(x_i) ~ \big | ~
x^n \geq 0 \}$, where $\rho \leq \si$ ---
of course, $B_\rho^+$ is the same as $B_\rho(0)$, but the first notation will be emphasized since
it better suits the Pohozaev identity. Furthermore, the unit normal will be denoted by
$\nu = \nu_g = \nu_{g_i}$ when no
confusion arises, and $\nu_\de$ will denote the Euclidean normal.

We can now state one of the main estimates of the paper.

\begin{prop} Suppose $6\leq n\leq 24$ and that
$x_{i}\rightarrow\overline{x}\in\partial M$ is an isolated simple
blow-up point. Then
\begin{align}
\sum_{|\alpha|=2}^{d}\sum_{i,j=1}^{n}|h_{ij,\alpha}|^{2}
\varepsilon^{2|\alpha|}|\log\varepsilon|^{\theta_{|\al|}} \leq
C\varepsilon^{n-2}, \nonumber
\end{align}
where $\theta_k = 1$ if $k=\frac{n-2}{2}$ and $\theta_k =0$ otherwise.
\label{key_estimate_bry}
\end{prop}

Before giving a proof of proposition \ref{key_estimate_bry}, some consequences are derived,
in particular the Weyl vanishing theorem.

\begin{theorem} (Weyl vanishing) Let $x_i \rar \bar{x}$ be an isolated simple blow-up
point and $6 \leq n \leq 24$, then
\begin{gather}
|\nabla_{g_i}^l W_g|^2(x_i) \leq C \ve_i^{n-6-2l}|\log \ve_i|^{-\theta_{l+2}}, \nonumber
\end{gather}
for every $0\leq l \leq \big [ \frac{n-6}{2} \big ]$, where $\theta_k = 1$ if
$k=\frac{n-2}{2}$ and $\theta_k =0$ otherwise.
In particular
$|\nabla_{g}^l W_g|^2(\bar{x}) = 0$ for $0\leq l \leq \big [ \frac{n-6}{2} \big ]$.
\label{weyl}
\end{theorem}
\begin{proof}
Proposition \ref{key_estimate_bry} gives the same estimate as in the boundaryless case,
the argument then is similar (see \cite{KMS}).
\end{proof}

\begin{coro}
Under the same hypotheses of the Weyl vanishing theorem,
\begin{gather}
 |\nabla^m(v_i-U-\tilde{z}_i)(y)| \leq C \ve_i^{n-3}(1+|y|)^{-m-1} \nonumber
\end{gather}
or, in $x$-coordinates
\begin{gather}
 |\nabla^m(u_i-u_{\ve_i} - z_{\ve_i})(x)| \leq C
\ve_i^{\frac{n-2}{2}}(\ve+|x|)^{-m-1}. \nonumber
\end{gather}
\label{coro_weyl}
\end{coro}
\begin{proof}
This is straightforward from proposition \ref{symmetry_3} and theorem \ref{weyl}.
\end{proof}

We now proceed with the proof of proposition \ref{key_estimate_bry}.
The proof will involve an application of the Pohozaev identity
(\ref{Pohozaev}) in a half ball $B_\rho^+$.

Write $\phi = \frac{n-2}{2} u + x^k \partial_k u$ and $\phi_\varepsilon =
\frac{n-2}{2} u_\varepsilon + x^k \partial_k u_\varepsilon$.
In the proofs below extensive use will be made of the inequalities
$|\nabla^m u| \leq C \ve^\frac{n-2}{2} |x|^{2-n-m}$ and
$|\nabla^m \phi| \leq C \ve^\frac{n-2}{2} |x|^{2-n-m}$, which follow from proposition
\ref{u_conv_Green_fc}.

\noindent \emph{Proof of proposition \ref{key_estimate_bry}:} First it will be shown that
there exists a constant $C$ such that
{\allowdisplaybreaks
\begin{eqnarray}
& &C \Big( \sum_{|\al|=2}^d \sum_{ij=1}^n |h_{ij,\al}|^2 \varepsilon^{2|\al|+1}
+ \sum_{|\al^\prime|=0}^{N} \sum_{i,j=1}^{n-1} |T_{ij,\al^\prime}|\ve^{|\al^\prime|+1}
+ \ve^{n-2} \Big )
\label{big_weyl_1}  \\
& & \geq - \int_{\partial^\prime B_\rho^+} \phi_\varepsilon H_{in} \partial_i
z_\varepsilon dx +
 \int_{B^+_\rho} c(n) \phi_\varepsilon u_\varepsilon \partial_{ij} h_{ij} dx
\nonumber  \\
& & +
\int_{B^+_\rho} c(n)
( \phi_\varepsilon z_\varepsilon + u_\varepsilon(\frac{n-2}{2} z_\varepsilon + x^k
\partial_k z_\varepsilon ) )\partial_{ij}h_{ij} dx
\nonumber  \\
& & + \int_{B^+_\rho} c(n) \phi_\varepsilon u_\varepsilon ( -\partial_j
(H_{ij}\partial_l H_{il}) +
\frac{1}{2}\partial_j H_{ij} \partial_l H_{il} - \frac{1}{4} \partial_l
H_{ij}\partial_l H_{ij} ) dx.
\nonumber
\end{eqnarray}}
where by $\al^\prime$ we mean derivatives along $x^\prime$ only.

Start with the Pohozaev identity (\ref{Pohozaev}).
On its left hand-side the integrals over the hemisphere $S_+^{n-1}(\rho)$ are of order
$\varepsilon^{n-2}$ and the boundary terms with $x^k
\nu_\de^k$ vanish on $\partial^\prime B_\rho^+$,
so the remaining term on the left hand side of (\ref{Pohozaev}) is
\begin{gather}
\int_{\partial^\prime B_\rho^+} (\frac{n-2}{2} u + x^k \partial_k u )\frac{\partial
u}{\partial \nu_\de} d x^\prime .
\nonumber
\end{gather}
Since $\partial_\nu u = 0$ and $g^{nn}$ is bounded away from zero near the origin
\begin{gather}
 \frac{\partial u}{\partial \nu_\de} = - \partial_n u = \frac{1}{g^{nn}}\sum_{l=1}^{n-1} g^{nl} \partial_l u,
\nonumber
\end{gather}
therefore
\begin{align}
\Big | \int_{\partial B_\rho^+} (\frac{n-2}{2} u + x^k \partial_k u )\frac{\partial
u}{\partial \nu_\de} d x^\prime \Big | & \leq
C \Big | \int_{\partial B_\rho^+} (\frac{n-2}{2} u + x^k \partial_k u )
\sum_{l=1}^{n-1} g^{nl}\partial_l (u - u_\ve - z_\ve ) d x^\prime \Big |
\nonumber \\
& + C\Big | \int_{\partial B_\rho^+} (\frac{n-2}{2} u + x^k \partial_k u )
\sum_{l=1}^{n-1} g^{nl}\partial_l (u_\ve +  z_\ve ) d x^\prime \Big | .
\nonumber
\end{align}
Using (\ref{metric_restricted}), (\ref{z_est_x}), (\ref{sym_est_x}),  and
theorem \ref{big_estimate},
we
find that the above integrals are bounded by $C \ve^{n-2}$ and terms
involving the umbilicity tensor.
Now
(\ref{big_weyl_1}) follows from (see \cite{KMS})
\begin{gather}
 |R_g - \partial_{ij}h_{ij} + \partial_j(H_{ij}\partial_l H_{il}) -\frac{1}{2}
\partial_j H_{ij}\partial_l H_{il}
+ \frac{1}{4} \partial_l H_{ij} \partial_l H_{ij} |
\leq C \sum_{|\al|=2}^d \sum_{ij=1}^n |h_{ij,\al}|^2 |x|^{2|\al|}  +C |x|^{n-3}. \nonumber
\end{gather}

The next step is to show that, as in the standard case of a full ball, the first interior term
on the right hand side of (\ref{big_weyl_1}) may be absorbed into the error.  To see
this observe that theorem \ref{big_estimate} implies
\begin{align}
\int_{B_\rho^+}\phi_{\varepsilon}u_{\varepsilon}\partial_{ij}h_{ij}
& =-\int_{\partial^{\prime}B_\rho^+}\phi_{\varepsilon}u_{\varepsilon}
(\sum_{j=1}^{n-1}\partial_{j}H_{nj}+\partial_{n}H_{nn})+O(\varepsilon^{n-2}) \nonumber  \\
& =O \Big (\sum_{|\al^\prime|=0}^{N}\sum_{i,j=1}^{n-1}|T_{ij,\al^\prime}|\varepsilon^{|\al^\prime|+1}
+ \varepsilon^{n-2} \Big) . \nonumber
\end{align}
Therefore after an integration by parts (\ref{big_weyl_1}) becomes
\begin{eqnarray}
 & &C \Big( \sum_{|\al|=2}^d \sum_{ij=1}^n |h_{ij,\al}|^2 \varepsilon^{2|\al|+1} +
 \sum_{|\al^\prime|=0}^{N} \sum_{i,j=1}^{n-1} |T_{ij,\al^\prime}|\ve^{|\al^\prime|+1}
+ \varepsilon^{n-2} \Big )    \label{big_weyl_2}  \\
& & \geq \int_{\partial^{\prime}B_\rho^+}(c(n)\phi_{\varepsilon}u_{\varepsilon}
H_{in}\partial_{l}H_{il}-\phi_{\varepsilon}H_{in}\partial_{i}z_{\varepsilon})
\nonumber  \\
& & -2\int_{B_\rho^+}c(n)u_{\varepsilon}z_{\varepsilon}(1+\frac{1}{2}x^{k}\partial_{k})
\partial_{ij}h_{ij}
+\int_{B_\rho^+}c(n)\phi_{\varepsilon}u_{\varepsilon}(\frac{1}{2}\partial_{j}H_{ij}
\partial_{l}H_{il}-\frac{1}{4}\partial_{l}H_{ij}\partial_{l}H_{ij}).
\nonumber
\end{eqnarray}
Furthermore the boundary integral on
the right hand side of (\ref{big_weyl_2}) may be absorbed into the
error term with the help of theorem \ref{big_estimate},
\begin{gather}
\int_{\partial^{\prime}B_\rho^+}(c(n)\phi_{\varepsilon}u_{\varepsilon}
H_{in}\partial_{l}H_{il}-\phi_{\varepsilon}H_{in}\partial_{i}z_{\varepsilon})\nonumber
=O \Big (\sum_{|\al^\prime|=0}^{N}\sum_{i,j=1}^{n-1}|T_{ij,\al^\prime}|\varepsilon^{|\al^\prime|+1} +
\ve^{n-2} \Big ). \nonumber
\end{gather}

The remaining interior integrals are the same as those that appear
in the original Weyl vanishing proof \cite{KMS}, except that the domain of
integration is a half ball instead of the full ball.  At this
point we may follow the original proof to obtain the desired
conclusion, as long as the
following two facts hold: (i) the
necessary integration by parts may be
performed with the extra boundary integrals (along
$\partial^{\prime}B^+_\rho$) being absorbed into the error,
(ii) an orthogonality condition among harmonic polynomials holds on the half ball.

An
inspection of the original proof shows that (i) is valid,
since any integrand along $\partial^{\prime}B^+_\rho$ will contain
quantities that either appear in theorem \ref{big_estimate} (and
thus may be
estimated by the umbilicity tensor) or involve
$\partial_{n}z_{\varepsilon}$ --- which vanishes by proposition \ref{bvp_z}.
Furthermore consider the decomposition (\ref{decomposition_appendix}), then
in the notation of \cite{KMS}
\begin{gather}
(\widehat{H}_{q})_{ij}^{(k)}=\mathrm{Proj}(\partial_{i}\partial_{j}p_{k-2q}|x|^{2q+2}) .
\nonumber
\end{gather}
In section \ref{boundary_correction} it was shown	
that $\partial_{n}p_{k-2q}|_{\partial^{\prime}B_\rho^+}$ vanishes,
therefore it follows
that $(\widehat{H}_{q})_{in}$,
$\partial_{n}(\widehat{H}_{q})_{nn}$, and
$\partial_{n}(\widehat{H}_{q})_{ij}$ can be estimated in terms of the
umbilicity tensor. This implies that the corresponding elements of $W_{ij}$ can also be estimated in terms
of the umbilicity tensor
(since the corresponding elements of $H_{ij}$ have this property
by consequence of theorem \ref{big_estimate}).
Hence $((\widehat{H}_{q})_{ij},W_{ij})$ can be absorbed into the
error,
where the inner product is taken over the half sphere.  Similarly
\begin{gather}
\int_{S_{+}^{n-1}(\rho)}(l-k)p_{l}p_{k}=\int_{\partial^{'}B_\rho^+}
(p_{k}\partial_{n}p_{l}-p_{l}\partial_{n}p_{k}) = 0 , \nonumber
\end{gather}
so that $p_{l}\perp p_{k}$, $l\neq k$, that is, (ii) is valid.
This finishes the proof of proposition \ref{key_estimate_bry}.

\section{Sign restriction\label{sec_sign_restriction}}

Define
\begin{gather}
 P^\prime(r,w) = \int_{\partial B_r^+(x_i)} \big( \frac{n-2}{2}w \frac{\partial
w}{\partial \nu_\de}
+ x^k\partial_k w\frac{\partial w}{\partial \nu_\de} - \frac{1}{2}x^k\nu_\de^k|\nabla w|^2 ) ds.
\nonumber
\end{gather}

\begin{prop} (Sign restriction) Let $x_i \rar \bar{x}$ be an isolated simple blow-up
point and assume that $3 \leq n \leq 24$.
If $M_i u_i(x) \rar w$ away from the origin then
\begin{gather}
 \liminf_{r\rar 0} P^\prime(r,w) \geq 0. \nonumber
\end{gather}
\label{sign_restriction}
\end{prop}
\begin{proof}
Define
\begin{gather}
 P(r,u_i) = \int_{\partial B_r^+(x_i)} \big( \frac{n-2}{2}u_i \frac{\partial
u_i}{\partial \nu_\de}
+ x^k\partial_k u_i\frac{\partial u_i}{\partial \nu_\de} - \frac{1}{2}x^k\nu_\de^k|\nabla
u_i|^2 +
\frac{1}{p_i+1}K(x) x^k\nu_\de^k u_i^{p_i+1} \big ) ds.
\nonumber
\end{gather}
If $r$ is sufficiently small, the Pohozaev identity (proposition
\ref{Pohozaev_prop}) gives
\begin{align}
 P(r,u_i) & \geq -  \int_{B_r^+(x_i)} ( \frac{n-2}{2}u_i +  x^k\partial_k u_i )\big(
(g_i^{lj} -
\de^{lj})\partial_{lj}u_i + \partial_l g_i^{lj} \partial_j u_i ) dx \nonumber
\displaybreak[0] \\
& +
\int_{B_r^+(x_i)} ( \frac{n-2}{2}u_i + x^k\partial_k u_i )R_{g_i}u_i dx. \nonumber
\displaybreak[0]
\end{align}
Notice that
{\allowdisplaybreaks
\begin{eqnarray}
& & \int_{B_r^+(x_i)} x^k\partial_k u_i R_{g_i}u_i dx
 \\
& & = -\int_{B_r^+(x_i)} (x^k \partial_k R_{g_i} + n R_{g_i})u_i^2 dx
-\int_{B_r^+(x_i)}x^k\partial_k u_i R_{g_i} u_i dx
+\int_{\partial B_r^+(x_i)} x^k \nu_\de^k R_{g_i} u_i^2. \nonumber
\end{eqnarray} }
Since  $x^k\nu_\de^k = 0$ on $\partial^\prime B^+_r(x_i)$ and $\nu_\de^k=x^k/r$ on
$\partial^+ B_r^+(x_i)$ we obtain
\begin{gather}
\int_{B_r^+(x_i)} x^k\partial_k u_i R_{g_i}u_i dx = -\frac{1}{2}
\int_{B_r^+(x_i)} (x^k \partial_k R_{g_i} + n R_{g_i})u_i^2 dx
+\frac{r}{2} \int_{\partial^+ B_r^+(x_i)} R_{g_i} u_i^2 ds, \nonumber \displaybreak[0]
\end{gather}
so
{\allowdisplaybreaks
\begin{eqnarray}
 & & c(n)\int_{B_r^+(x_i)}(\frac{n-2}{2}u_i + x^k\partial_ku_i)R_{g_i} u_i dx \\
& & = - c(n)\int_{B_r^+(x_i)}(\frac{1}{2}x^k\partial_k R_{g_i} + R_{g_i} ) u_i^2 dx
 +  c(n) \frac{r}{2} \int_{\partial^+ B_r^+(x_i)} R_{g_i} u_i^2 ds, \nonumber
\end{eqnarray} }
and then
{\allowdisplaybreaks
\begin{eqnarray}
 P(r,u_i) & \geq & -  \int_{B_r^+(x_i)} ( \frac{n-2}{2}u_i +  x^k\partial_k u_i )\big(
(g_i^{lj} - \de^{lj})\partial_{lj}u_i + \partial_l g_i^{lj} \partial_j u_i \big ) dx
\nonumber  \\
& - & c(n)\int_{B_r^+(x_i)}(\frac{1}{2}x^k\partial_k R_{g_i} + R_{g_i} ) u_i^2 dx
+ c(n) \frac{r}{2} \int_{\partial^+ B_r^+(x_i)} R_{g_i} u_i^2  ds \nonumber
 \\
& =& A_i(r) +  c(n) \frac{r}{2} \int_{\partial^+ B_r^+(x_i)} R_{g_i} u_i^2  ds
\nonumber
\end{eqnarray} }
where $A_i(r)$ is defined by the above equality. Now observe
that $\int_{\partial B_r^+(x_i)} K(x) M_i^2 x^k\nu_\de^k u_i^{p_i+1} \rar 0$.
In fact,
the integral over $\partial^\prime B_r^+(x_i)$ vanishes as $x^k\nu_\de^k=0$ there. On
$\partial^+ B_r^+(x_i)$ we have $x^k\nu_\de^k=r$, hence, using the equation satisfied by
$u_i$ produces
\begin{gather}
\int_{\partial^+ B_r^+(x_i)} K(x) M_i^2 u_i^{p_i+1} = -\int_{\partial^+ B_r^+(x_i)}
M_iu L_{g_i} M_i u_i \rar
-\int_{\partial^+ B_r^+(x_i)} w L_{g}w = 0.
\nonumber
\end{gather}
Therefore $M_i^2 P(r,u_i) \rar P^\prime(r,w)$, so
\begin{align}
 P^\prime(r,w) & = \lim_{i\rar \infty} M_i^2 P(r,u_i) \geq  \lim_{i\rar \infty} M_i^2
A_i(r) +
 \lim_{i\rar \infty} c(n) \frac{r}{2} \int_{\partial^+ B_r^+(x_i)} R_{g_i} (M_i
u_i)^2  ds \nonumber \displaybreak[0] \\
& = \lim_{i\rar \infty} M_i^2 A_i(r) + c(n)\frac{r}{2}\int_{\partial^+ B_r^+(x_i)}
R_{g_i} w^2 ds. \nonumber
\end{align}

We now proceed to analyze $M_i^2 A_i(r)$, noticing that since theorem \ref{weyl} and
corollary \ref{coro_weyl}
give the same estimates as in the boundaryless case,
the same analysis can be carried out, except for an extra boundary term that
appears in $\hat{A}_i(r)$ when integration by parts is performed, where
\begin{align}
 \hat{A}_i(r) & = -c(n) \int_{B_r^+(x_i)}(\frac{1}{2}x^k\partial_kR_{g_i} + R_{g_i}
)(u_{\ve_i}+z_{\ve_i})^2 dx
\label{def_hat_A_sign}
\displaybreak[0] \\
& - \int_{B_r^+(x_i)}( \frac{n-2}{2}(u_{\ve_i}+z_{\ve_i}) +
x^k\partial_k(u_{\ve_i}+z_{\ve_i})  )(\Delta_{g_i}-\Delta_\de)(u_{\ve_i}+z_{\ve_i}) dx.
\nonumber
\end{align}
Corollary \ref{coro_weyl} implies that $\ve_i^{2-n}|A_i(r) - \hat{A}_i(r) | \leq Cr$, so
$\lim_{i\rar\infty} \ve_i^{2-n}(A_i(r) - \hat{A}_i(r) ) \geq -Cr$. Notice that since
$M_i = \ve_i^{-\frac{2}{p_i-1}}$ and $-\frac{4}{p_i-1} \rar 2-n$ we can replace
$M_i^2$ by $\ve_i^{2-n}$ and obtain
\begin{gather}
 P^\prime(r,w) \geq -Cr + c(n)\frac{r}{2}\int_{\partial^+ B_r^+(x_i)} R_{g_i} w^2 ds
+ \lim_{i\rar\infty} \ve_i^{2-n}\hat{A}_i(r).
\label{inequality_P_prime}
\end{gather}
Using the symmetries of $u_{\ve_i}$
\begin{eqnarray}
 & & \int_{B_r^+(x_i)}( \frac{n-2}{2}(u_{\ve_i}+z_{\ve_i}) +
x^k\partial_k(u_{\ve_i}+z_{\ve_i})  )(\Delta_{g_i} -\Delta_\de)(u_{\ve_i}+z_{\ve_i}) dx
\label{int_using_sym_sign} \\
& & =
\int_{B_r^+(x_i)}( \frac{n-2}{2}z_{\ve_i} + x^k\partial_kz_{\ve_i}
)(\Delta_{g_i}-\Delta_\de)z_{\ve_i} dx \nonumber  \\
& & + \int_{\partial B_r^+(x_i)}( \frac{n-2}{2}u_{\ve_i} + x^k\partial_k u_{\ve_i} )
(\frac{\partial z_{\ve_i}}{\partial \nu_{g_i}} - \frac{\partial z_{\ve_i}}{\partial
\nu_{\de}})ds \nonumber  \\
& & -\int_{\partial B_r^+(x_i)}z_{\ve_i}(\frac{\partial L u_{\ve_i}}{\partial \nu_{g_i}} -
\frac{\partial L u_{\ve_i}}{\partial \nu_{\de}})ds, \nonumber
\end{eqnarray}
where $L=\frac{n-2}{2} + x^k\partial_k$. The integrals over $\partial^+ B_r^+(x_i)$ vanish by properties
of normal coordinates, so consider the integrals over $\partial^\prime B_r^+(x_i)$. Observe that
$\frac{\partial z_{\ve_i}}{\partial \nu_{\de}} = 0 = \partial_n z_{\ve_i}$
by proposition \ref{bvp_z}. Then using (\ref{metric_restricted}), the definition of $u_{\ve_i}$,
and (\ref{z_est_x}), we obtain
\begin{align}
 \Big | \int_{\partial^\prime B_r^+(x_i)}( \frac{n-2}{2}u_{\ve_i} + x^k\partial_k u_{\ve_i} )
\frac{\partial z_{\ve_i}}{\partial \nu_{g_i}} ds \Big |  &=
 \Big | \int_{\partial^\prime B_r^+(x_i)}( \frac{n-2}{2}u_{\ve_i} + x^k\partial_k u_{\ve_i} )
\sum_{l=1}^{n-1}g_i^{nl} \partial_l z_{\ve_i} ds \Big |
\label{int_using_sym_sign_bry_1} \\
& \leq \int_{\partial^\prime B_r^+(x_i)} (\ve_i + |x^\prime|)^{2-n} (\ve_i + |x^\prime|)^{6-n-1} |x^\prime|^N dx^\prime \nonumber \\
& \leq C \ve_i^{n-2} r . \nonumber
\end{align}
For the other boundary integral notice that
\begin{gather}
 \frac{\partial L u_{\ve_i}}{\partial \nu_{g_i}} -
\frac{\partial L u_{\ve_i}}{\partial \nu_{\de}} =
(-g_i^{n \si} \partial_\si + \partial_n )L u_{\ve_i} = (-g_i^{nn} + 1)L u_{\ve_i}  - \sum_{l=1}^{n-1} g_i^{nl} \partial_l L u_{\ve_i} .
\nonumber
\end{gather}
Since $|\nabla L u_{\ve_i} | \leq \ve_i^\frac{n-2}{2} (\ve_i + |x|)^{-n}$, using
(\ref{z_est_x}), theorem \ref{estimate_h_nn_boundary}, theorem \ref{big_estimate}, and (\ref{metric_restricted}),
it follows that
\begin{align}
\Big | \int_{\partial^\prime B_r^+(x_i)}z_{\ve_i}(\frac{\partial L u_{\ve_i}}{\partial \nu_{g_i}} -
\frac{\partial L u_{\ve_i}}{\partial \nu_{\de}})ds \Big | \leq C \ve_i^{n-2} r .
\label{int_using_sym_sign_bry_2}
 \end{align}
Combining (\ref{def_hat_A_sign}), (\ref{inequality_P_prime}),  (\ref{int_using_sym_sign}),
(\ref{int_using_sym_sign_bry_1}), and (\ref{int_using_sym_sign_bry_2}) yields
\begin{align}
 P^\prime(r,w) & \geq -Cr
\int_{B_r^+(x_i)}( \frac{n-2}{2}z_{\ve_i} + x^k\partial_kz_{\ve_i}
)(\Delta_{g_i}-\Delta_\de)z_{\ve_i} dx
\label{integral_sign_rest_similar_boundaryless} \displaybreak[0] \\
&  + c(n)\frac{r}{2}\int_{\partial^+ B_r^+(x_i)} R_{g_i} w^2 ds
\nonumber \displaybreak[0] \\
& - c(n)\lim_{i\rar\infty} \ve_i^{2-n}  \int_{B_r^+(x_i)}(\frac{1}{2}x^k\partial_k R_{g_i} +
R_{g_i} )(u_{\ve_i}+z_{\ve_i})^2 dx. \nonumber
\end{align}

We can now proceed as in the boundaryless case.
The first integral on the right hand side of (\ref{integral_sign_rest_similar_boundaryless})
as well as $\frac{r}{2}\int_{\partial^+ B_r^+(x_i)} R_{g_i} w^2 ds$ are estimated using
theorem \ref{weyl}.
 Theorem
\ref{weyl} and corollary \ref{coro_weyl}
may be used to estimate $\int_{B_r^+(x_i)}(\frac{1}{2}x^k\partial_k R_{g_i} + R_{g_i} )z_{\ve_i}^2 dx$. Finally
the estimate
of proposition \ref{key_estimate_bry} is used to handle
$\int_{B_r^+(x_i)}(\frac{1}{2}x^k\partial_k R_{g_i} + R_{g_i} )(u^2_{\ve_i} +2u_{\ve_i} z_{\ve_i}
)dx$.
\end{proof}

\section{Blow-up set\label{blow-up_set}}

In this section we show that the set of blow-up points is finite and consists only
of isolated simple blow-up points.
The proofs are very similar to the boundaryless case (\cite{KMS}) and the
locally conformally flat case with boundary (\cite{HL}), and therefore we will
go through them rather quickly, indicating the necessary modifications.

The following proposition is proven in \cite{HL} (proposition 1.1, see also
\cite{KMS,M}).

\begin{prop} Given $\de>0$ sufficiently small and $R>0$ sufficiently large, there
exists a constant $C=C(\de,R)>0$
such that if $u$ is a positive solution of (\ref{bvp}) with $\max u > C$, then there
exists $\{x_1,\dots,x_N\} \subset M$, $N=N(u)>1$,
where $\frac{n+2}{n-2} - p < \de$ and each $x_i$ is a local maximum of $u$ such that:

1) $\{B_{r_i}(x_i)\}_{i=1}^N$ is a disjoint collection if $r_i = R
u(x_i)^{-\frac{p-1}{2}}$,

2) in normal coordinates centered at $x_i$
\begin{gather}
 \parallel u_i(x_i)^{-1} u(u_i(x_i)^{-\frac{p-1}{2}}y) - U(y)
\parallel_{C^2(B_R(0))} < \de  \nonumber
\end{gather}
where $y=u(x_i)^{\frac{p-1}{2}}x$,

3) $u(x) \leq C d_g(x,\{x_1,\dots,x_n\})^{-\frac{2}{p-1}}$ for all $x \in M$ and
\begin{gather}
 d_g(x_i,x_j)^\frac{2}{p-1}u(x_j) \geq C^{-1} \nonumber
\end{gather}
for $x_i \neq x_j$.
\label{no_accumulaion_1}
\end{prop}

\begin{lemma} Let $x_i \rar \bar{x}$ be an isolated blow-up point for the sequence
$\{u_i\}$ of
 positive solutions of (\ref{bvp}). Then $\bar{x}$ is an isolated simple blow-up point.
\label{isolated_is_isolated_simple}
\end{lemma}
\begin{proof}
We argue as in \cite{KMS} to obtain a subsequence $w_i$ such that
\begin{gather}
 w_i(0)w_i(y) \rar h(y) = a|y|^{2-n} + b(y) \text{ in } C^2_{loc}(\RR^n_+ \backslash
\{0\}), \nonumber
\end{gather}
where $b(y)$ is harmonic in $\RR^n_+$ and satisfies $\partial_n b = 0$ on $\RR^{n-1}$.
Therefore, extending
$b$ across $\RR^{n-1}$ and using Liouville's theorem shows that $a=b>0$. Arguing as
in \cite{KMS} this leads to a contradiction with proposition \ref{sign_restriction}.
\end{proof}

\begin{prop}
  Let $\de$, $R$, $u$, $C(\de,R)$, and $\{x_1,\dots,x_N\}$ be as in proposition
\ref{no_accumulaion_1}. If $\de$ is sufficiently small and $R$ sufficiently
large, then there exists a constant
$\bar{C}(\de,R)>0$ such that if $\max_M u \geq C$ then $d_g(x_j,x_l) \geq \bar{C}$
for all $1\leq j \neq l \leq N$.
\end{prop}
\begin{proof}
Again we argue as in \cite{KMS}, making the necessary modifications along the lines
of \cite{HL} as in
lemma \ref{isolated_is_isolated_simple}.
\label{no_accumulaion_2}
\end{proof}

The following is an immediate consequence.

\begin{coro}
 Let $\{ u_i \}$ be a sequence of solutions of (\ref{bvp}) with $\max_M u_i \rar
\infty$. Then $p_i \rar \frac{n+2}{n-2}$
and the set of blow-up points is finite and consists only of isolated simple blow-up
points.
\label{no_accumulaion_3}
\end{coro}

\section{Compactness\label{proof_compactness}}

Now that we have the Weyl vanishing theorem and sign restriction, the remaining arguments
for the proof of theorems \ref{compactness_theorem} and \ref{Weyl_vanishing_theorem}
are similar to those of the boundaryless case. In fact, the results of this section
will be an adaptation of
\cite{KMS,BC,Es}, and therefore as in section \ref{blow-up_set},
we will go through the proofs very briefly.

\emph{Proof of theorem \ref{compactness_theorem}}: From the results of
section \ref{blow-up_set} $p_i \rar \frac{n+2}{n-2}$, and
there exists a finite number $N>0$ of isolated simple blow-up points
$x_i^{(1)} \rar \bar{x}^{(1)},\dots,x_i^{(N)} \rar \bar{x}^{(N)}$. If none of the
$\bar{x}_\ell$ belong
to the boundary then the compactness result follows from \cite{KMS}, so
assume that at least one of
them belongs to $\partial M$. It may also be assumed without loss of generality that
$\bar{x}_\ell \in \partial M$, $\ell=1,\dots,N-k$ and $\bar{x}_\ell \notin \partial
M$, $\ell= N-k+1,\dots,N$,
for some $k \leq N-1$. Furthermore let
\begin{gather}
 u_i(x_i^{(1)}) = \min \{ u_i(x_i^{(1)}),\dots,u_i(x_i^{(N-k)}) \} \nonumber
\end{gather}
for all $i$.

Set $w_i = u_i(x_i^{(1)})u_i$. A standard estimate gives that away from the blow-up
points
$w_i \rar \sum_{j=1}^N a_j G_{\bar{x}^{(j)}}$, where $a_j \geq 0$, $a_1 > 0$ and
$G_{\bar{x}^{(j)}}$ is the Green's function for the conformal Laplacian with
singularity at $\bar{x}^{(j)}$.
Now argue as in \cite{SY1} (see \cite{BC,Es} as well) to obtain the asymptotic
expansion
\begin{gather}
 G(x,\bar{x}^{(1)}) = |x|^{2-n}\Big (1+\sum_{k=d+1}^{n-2} \psi_k \Big ) + A +
O(|x|\log|x|),
\label{asym_exp_Green}
\end{gather}
where $G=G_{\bar{x}^{(1)}}$, $\psi_k$ are homogeneous polynomials of degree $k$ and
$A$ is a constant. The sum between parenthesis
starts at $k=d+1$ because $h_{ij,\al}(\bar{x}) = 0$ at a blow-up point $\bar{x} \in
\partial M$, by the Weyl
vanishing theorem. We remark that
when the boundary is not umbilic an extra singular term appears in this expansion
(see e.g. \cite{L}).
Also notice that standard properties of conformal normal coordinates, theorem
\ref{big_estimate}, and the umbilicity of the boundary, imply that $\int_{S^{n-1}_+}
\partial_{ij}H_{ij} = 0$, from which
it follows that
\begin{gather}
 \int_{S^{n-1}_+} \psi_k = 0 = \int_{S^{n-1}_+} x_i \psi_k.
\label{int_hom_pol_Green}
\end{gather}
Now put $\widehat{g} = G^{\frac{4}{n-2}} g$. Then
$(M\backslash\{\bar{x}^{(1)}\},\widehat{g})$ is scalar
flat and its boundary is totally geodesic. If we introduce the asymptotic
coordinates $y=|x|^{-2}x$, then
the expansion (\ref{asym_exp_Green}) and the Weyl vanishing theorem give
$\widehat{g}_{ij} = \de_{ij} + O(|y|^{-d-1})$.
Therefore the doubling of $(M\backslash\{\bar{x}^{(1)}\},\widehat{g})$ is
asymptotically flat and has a well defined ADM
mass (\cite{B}, compare also with \cite{BC}).

The rest of the argument now is standard. The positive mass theorem (see remark below) along with
(\ref{int_hom_pol_Green}) and
the Weyl vanishing give that $A>0$ (as in \cite{KMS}, using the hypothesis that
the manifold is not conformally equivalent to the round hemisphere we can rule out the $A=0$ case).
This contradicts the sign restriction
of theorem \ref{sign_restriction}, finishing the proof.

\begin{rema}$~$ \\ 1) Strictly speaking, we did not show how to prove a positive
mass theorem (PMT) for manifolds with boundary, as
the mass of such manifolds was never defined. What is referred to as the PMT for manifolds with
boundary is actually the statement
that the constant term in the asymptotic expansion of the Green's function is
non-negative, which in turn is implied by the positivity
of the mass of the doubled manifold (see \cite{M}). \\2) The PMT is known to
hold up to dimension $7$ \cite{S2,SY1,SY2} and in arbitrary dimensions if the
manifold is spin \cite{W,LP}.
Therefore, our result for $n \geq 8$ in the case of non-spin manifolds is true provided
that the PMT holds
under such hypotheses.
\end{rema}

\emph{Proof of theorem \ref{Weyl_vanishing_theorem}:} This follows from lemma
\ref{isolated_is_isolated_simple} and
theorem \ref{weyl}.

\section{Blow-up of solutions for $n\geq25$}

In this section we prove theorem \ref{non_compactness_theorem}. We assume $n\geq 25$ throughout.
As we mention in the introduction,
the proof relies heavily on the constructions of Brendle \cite{Br} and Brendle and Marques \cite{BM},
and we refer the reader to them on several occasions.

We start collecting facts from \cite{Br,BM} that will be of direct use in our proof.
Their main results is
\begin{theorem} (Brendle and Marques, \cite{Br,BM}) Assume that $n \geq 25$. Then there exists a
metric $g$ on $S^n$ (of class $C^\infty$) and a sequence of positive functions $u_i \in C^\infty(S^n)$
with the following properties:

\hskip 0.5cm  (a) $g$ is a small perturbation of the round metric $g_0$ which is not conformally flat,
and $g=g_0$ near and beyond the equator,

\hskip 0.5cm  (b) for each $i$, $u_i$ is a solution of the Yamabe equation
\begin{align}
L_{g}u_i + K u_i^{\frac{n+2}{n-2}} = 0,
\nonumber
\end{align}
where $K=n(n-2)$ is a positive constant,

\hskip 0.5cm  (c) $E_{g}(u_i) < Y(S^n)$ for all $i \in \NN$, and $E_{g}(u_i) \rar Y(S^n)$ as
$i \rar \infty$, where $E_{g}(u_i)$ is the Yamabe energy of $u_i$ and $Y(S^n)$ is the Yamabe invariant
of the round sphere,

\hskip 0.5cm  (d) $\sup_{S^n} u_i \rar \infty$ as $i\rar \infty$.
\label{Brendle_Marques_blow_sphere}
\end{theorem}

The scalar curvature of the metric $g$ satisfies
\begin{equation}
R_{g}\geq c>0 .
\label{positivity_R}
\end{equation}
for some constant c, since $g$ is a small perturbation of the
round metric. In particular
this guarantees the coercivity of $L_{g}$, which allows us to use the
the $C^0$-blow-up theory developed by Druet, Hebey and Robert \cite{DHR}. From their results and
estimate $(c)$ of theorem \ref{Brendle_Marques_blow_sphere} it then follows
(theorem 5.2 of \cite{DHR}, see also discussion at the end of section 5.1) that $u_i$
has only one blow-up point, and it is apparent from \cite{Br,BM}
that this is the south pole (from the point of view of stereographic projection).
Moreover, up to a subsequence the following estimate holds (again theorem 5.2 of \cite{DHR})
\begin{gather}
 Q^{-1} u_{\ve_i,x_i} (x) \leq u_i(x) \leq Q u_{\ve_i,x_i}(x),
\label{C_0_estimate}
\end{gather}
for some constant $Q>1$ independent of $i$ and for all $x \in S^n$;
here $\ve_i = (\sup_{S^n} u_i)^{-\frac{2}{n-2}} = u_i(x_i)^{-\frac{2}{n-2}}$,
and $u_{\ve_i,x_i} = \ve_i^{\frac{n-2}{2}} ( \ve_i^2 + |x-x_i|^2 )^{\frac{2-n}{2}}$,
$|x-x_i| = \operatorname{dist}_{g}(x,x_i)$.

Consider now the south hemisphere $S^n_{-}$, which we
identify with the unit ball in $\RR^n$ via stereographic projection. Since $g = g_0$
on a neighborhood $\partial S^n_{-}$, we have
that $\partial S^n_{-}$ is totally geodesic, and in particular
$B_{g} = \partial_{\nu_g}$. Combining (\ref{C_0_estimate}) with the Harnack inequality
implies that away from the south pole, $\ve_i^\frac{2-n}{2} u_i$ converges in $C^2$ to
a positive Green's function for the conformal Laplacian (possibly after passing to a subsequence).
We claim that for large $i$
\begin{gather}
\frac{ \partial u_i}{\partial \nu_{g}} \leq 0 .
\label{normal_der_u_i_negative}
\end{gather}
To see this, denote by $\de$ the Euclidean metric so that
$g_0 = 4 U^\frac{4}{n-2} \de$. Let $G_{g_0}$ and $G_\de$ be the corresponding Green's functions
with singularity
at zero. Their relation is given by
$G_{g_0} = 4^{-\frac{n-2}{2}} U^{-1} G_\de$. Using (\ref{conf_inv_bc}) and the fact that
the mean curvature
of $\partial S^n_{-}$ vanishes , we have
\begin{gather}
\frac{ \partial G_{g_0} }{\partial \nu_{g_0} } = B_{g_0} G_{g_0} =
4^{-\frac{n-2}{2} } U^{-\frac{n}{n-2}} B_\de G_\de < 0
\nonumber
\end{gather}
on $\partial S^n_{-}$, where the inequality follows by direct calculation. Therefore
$\frac{ \partial G_{g} }{\partial \nu_{g}} < 0$
by theorem \ref{Brendle_Marques_blow_sphere}(a), so that
(\ref{normal_der_u_i_negative}) holds.

We conclude that
\begin{align}
\begin{cases}
L_{g}u_i + K u_i^{\frac{n+2}{n-2}} = 0 ~, & \text{ in } S^n_{-}, \displaybreak[0] \\
B_{g} u_i  \leq 0, & \text{ on } \partial S^n_{-}.
\end{cases}
\nonumber
\end{align}
That is, $u_i$ is a sub-solution of the boundary value problem
\begin{align}
\begin{cases}
L_{g} v + K v^{\frac{n+2}{n-2}} = 0 , & \text{ in } S^n_{-}, \displaybreak[0] \\
B_{g} v  = 0, & \text{ on } \partial S^n_{-}.
\end{cases}
\label{bvp_lower_hemisphere}
\end{align}

Actual solutions to (\ref{bvp_lower_hemisphere}) will be constructed by finding appropriate super-solutions.
The super-solutions will satisfy the equation with a different constant $K$,
and this will require a slight modification of the standard sub-super-solutions argument.

\begin{theorem}
 For all sufficiently large $i$ there exists a solution $v_i$ of
(\ref{bvp_lower_hemisphere}) satisfying $u_i \leq v_i$. In particular $\sup_{S^n_{-}} v_i \rar \infty$ as
$i \rar \infty$.
\end{theorem}
\begin{proof}
Because of (\ref{positivity_R}), we can choose $\de > 0$ so small that
\begin{gather}
 L_{g} \de + K \de^\frac{n+2}{n-2} = -c(n)R_{g} \de + K \de^\frac{n+2}{n-2} \leq 0.
\nonumber
\end{gather}
Put $w_i = A_i \de$,
where $A_i > 1$ is a constant chosen so large that
\begin{gather}
 u_i \leq A_i \de,
\label{choice_A_i_1}
\end{gather}
and
\begin{gather}
u_i^\frac{n+2}{n-2} - A_i \de^\frac{n+2}{n-2} \leq 0
\label{choice_A_i_2}
\end{gather}
By the choice of $\de$
\begin{align}
\begin{cases}
L_{g} w_i + \widetilde{K}_i w_i^{\frac{n+2}{n-2}}  \leq 0 , & \text{ in } S^n_{-}, \displaybreak[0] \\
B_{g} w_i = 0, & \text{ on } \partial S^n_{-},
\end{cases}
\end{align}
where $\widetilde{K}_i =  A_i^{-\frac{4}{n-2}}  K$. So
$w_i$ is a super-solution of the problem with constant $\widetilde{K}_i$.
As pointed out before, $(L_{g},B_{g})$
is invertible and therefore the operators $T$ and $P_i$ given by
\begin{align}
T z= t \Leftrightarrow
\begin{cases}
L_{g} t  = - K z^{\frac{n+2}{n-2}}, & \text{ in } S^n_{-}, \displaybreak[0] \\
B_{g} t = 0, & \text{ on } \partial S^n_{-},
\end{cases}
\end{align}
and
\begin{align}
P_i w = p_i \Leftrightarrow
\begin{cases}
L_{g} p_i  = - \widetilde{K}_i w^{\frac{n+2}{n-2}}, & \text{ in } S^n_{-}, \displaybreak[0] \\
B_{g} p_i = 0, & \text{ on } \partial S^n_{-},
\end{cases}
\end{align}
are well defined. By the maximum principle $T$ and $P_i$ are monotone in the sense
that $z_0 \leq  z_1 \Rightarrow T z_0 \leq T z_1$, and analogously for $P_i$.

Now we put $u^0_i = u_i$, $w^0_i = w_i$ and define inductively
$u^{\ell+1}_i = T u^{\ell}_i$ and $w_i^{\ell+1} = P_i w^{\ell}_i$.
Since $u_i$ is a sub-solution we obtain $u^0_i \leq u^1_i$ and inductively
$u^\ell_i \leq u^{\ell+1}_i$. Analogously $w^{\ell}_i \geq w^{\ell+1}_i$
since $w_i$ is a super-solution.

We have $u^0_i \leq w^0_i$ by (\ref{choice_A_i_1}),
and claim that
$u^\ell_i \leq w^\ell_i$ for every $\ell$ (the difference from the
standard sub-super-solutions argument is that the equations involved in the definition of $T$ and $P_i$ are
not exactly the same due to the different constants $K$ and $\widetilde{K}_i$). The difference
$u^{\ell+1}_i- w^{\ell+1}_i$ satisfies
\begin{align}
\begin{cases}
L_{g} (u^{\ell+1}_i- w^{\ell+1}_i)
= - ( K (u^{\ell}_i)^{\frac{n+2}{n-2}} - \widetilde{K}_i (w^{\ell}_i)^{\frac{n+2}{n-2}} ), & \text{ in } S^n_{-}, \displaybreak[0] \\
B_{g} (u^{\ell+1}_i- w^{\ell+1}_i) = 0, & \text{ on } \partial S^n_{-}.
\end{cases}
\label{monotonicity_P_T}
\end{align}
In order to apply the maximum principle we need the right hand side of
(\ref{monotonicity_P_T}) to be non-negative. To show this, recall the definition of
$w_i$ and $\widetilde{K}_i$, use the monotonicity of the sequences $u^\ell_i$ and $w^\ell_i$,
as well as (\ref{choice_A_i_2}) to find
\begin{align}
 K (u^{\ell}_i)^{\frac{n+2}{n-2}} - \widetilde{K}_i (w^{\ell}_i)^{\frac{n+2}{n-2}}
\geq  K (u^{0}_i)^{\frac{n+2}{n-2}} - \widetilde{K}_i (w^{0}_i)^{\frac{n+2}{n-2}}
=  K u_i^{\frac{n+2}{n-2}} - K A_i  \de^{\frac{n+2}{n-2}} \leq 0 ,
\nonumber
\end{align}
It follows that $u^\ell_i \leq w^\ell_i$.

Now a standard argument produces the desired solution $u^\infty_i$ of (\ref{bvp_lower_hemisphere})
such that $u_i \leq u^\infty_i$. The proof also yields a $w^\infty_i$ solving (\ref{bvp_lower_hemisphere})
with $\widetilde{K}_i$ in place of $K$, and
such that and $w^\infty_i \leq w_i$, but this is not the solution we are looking for due to
the different $i$-dependent constant $\widetilde{K}_i$.
\end{proof}

\section{Leray-Schauder degree of solutions \label{Leray}}

Here we discuss some consequences of theorem \ref{compactness_theorem}. Throughout
this section we assume $3 \leq n \leq 24$. The results here are very
similar to the cases of manifolds without boundary and
locally conformally flat with boundary, so we refer the reader to \cite{KMS} and
\cite{HL} for details.

As we pointed out in the introduction, one obvious consequence of
theorem \ref{compactness_theorem} is to give an alternative
proof of the solution to the Yamabe problem. This follows from the fact
that standard variational methods can be used to give solutions to the
subcritical problem
\begin{align}
\begin{cases}
L_{g}u + K u^p = 0, & \text{ in } M, \displaybreak[0] \\
B_g u = 0, & \text{ on } \partial M,
\end{cases}
\label{Yamabe_eq_bry_p}
\end{align}
with $1 < p < \frac{n+2}{n-2}$.
More generally, the compactness theorem allows us to compute
the total Leray-Schauder degree of all solutions to equation (\ref{Yamabe_eq_bry_p}), and
to obtain more refined existence theorems which we now discuss.

Without loss of generality we can assume that $R_g > 0$ and $\kappa_g = 0$.
Then we can write (\ref{Yamabe_eq_bry_p}) as
\begin{align}
\begin{cases}
L_{g}u + E(u)u^p = 0, & \text{ in } M, \displaybreak[0] \\
\frac{ \partial u}{\partial \nu_g} = 0, & \text{ on } \partial M,
\end{cases}
\label{Yamabe_eq_bry_p_E}
\end{align}
where
\begin{gather}
 E(u) = \int_M (|\nabla_g u|^2 + c(n) R_g u^2 ) dV_g
\nonumber
\end{gather}
is the energy of $u$ (there is no boundary term since $\kappa_g = 0$). Notice that the Neumann problem for the conformal Laplacian
is invertible in that $R_g > 0$. Defining
\begin{gather}
\Om_\La = \{ u \in C^{2,\al}(M) ~ | ~ \parallel u \parallel_{C^{2,\al}}(M) < \La,~ u > \La^{-1} \}
\nonumber
\end{gather}
we obtain a map
$F_p: \overline{\Om}_\La \rar C^{2,\al}(M)$ given by $F_p(u) = u + L_g^{-1}(E(u)u^p)$.

From elliptic theory, we know that the map $u \mapsto L_g^{-1}(E(u)u^p)$ is a
compact map from $\overline{\Om}_\La$ into $C^{2,\al}(M)$. Thus $F_p$ is of the form $I+$compact,
and we may define the Leray-Schauder degree (see \cite{Ni}) of $F_p$ in the
region $\Om_\La$ with respect to $0 \in C^{2,\al}(M)$, denoted by $\operatorname{deg}(F_p,\Om_\La, 0)$,
provided that $0 \notin F_p(\partial \Om_\La)$. The degree is an integer which counts
with multiplicity the number of times that the value $0$ is taken on by
the map $F_p$. Notice that $F_p(u) = 0$ if and only if $u$ is a solution of
(\ref{Yamabe_eq_bry_p_E}). Furthermore, the homotopy invariance of the degree tells us that
$\operatorname{deg}(F_p, \Om_\La, 0)$ is constant for all $p \in [1, \frac{n+2}{n-2} ] $
provided that $0 \notin F_p(\partial \Om_\La)$ for all $p \in [1, \frac{n+2}{n-2} ] $.
Moreover, in the linear case when $p = 1$, it is not
difficult to calculate, by an argument similar to what is done in \cite{S0}, that
$\operatorname{deg}(F_1,\Om_\La, 0) = -1$ for all
$\La$ sufficiently large. Therefore, theorem \ref{compactness_theorem} allows us to calculate the
degree for all $p \in [1, \frac{n+2}{n-2} ] $. Since it follows
from the a priori estimates we derived that $0$ does not belong to $F_p(\partial \Om_\La)$, we obtain

\begin{theorem}
Let ($M^n, g)$ satisfy the assumptions of theorem \ref{compactness_theorem}.
Then for all $\La$ sufficiently large and all $p \in [1, \frac{n+2}{n-2}]$, we have
$\operatorname{deg}(F_p,\Om_\La, 0) = -1$.
\end{theorem}

In the case that all solutions of the Yamabe problem are nondegenerate,
our previous results assert that there will be a finite number
of solutions of the variational problem. Moreover, the strong Morse inequalities
will hold for the Yamabe problem since these inequalities hold
for subcritical equations, and theorem \ref{compactness_theorem} shows that all
critical points converge as $p \rar \frac{n+2}{n-2}$. It follows that
\begin{gather}
 (-1)^\la \leq \sum_{\mu=0}^\la (-1)^{\la-\mu} C_\mu,~\la = 0,1,2,\dots
\nonumber
\end{gather}
where $C_\mu$ denotes the number of solutions of Morse index $\mu$. Since there
is a finite number of solutions, we then obtain:
\begin{theorem}
Let $(M^n, g)$ satisfy the assumptions of theorem \ref{compactness_theorem},
and suppose that all critical points in $[g]$ are nondegenerate. Then there
is a finite number of critical points $g_1, \dots, g_k$, and we have
\begin{gather}
1=\sum_{j=1}^k (-1)^{I(g_j) }
\nonumber
\end{gather}
where $I(g_j)$ denotes the Morse index of the variational problem with
volume constraint.
\end{theorem}

\appendix

\section{Auxiliary results\label{auxiliary}}
In this section we state several auxiliary results that are either well known
or slight modifications of standard
results. Therefore proofs, when provided, will be rather short.

The following proposition is analogous to a well known theorem of
Caffarelli, Gidas, and Spruck (\cite{CGS}):

\begin{prop} Let $T \geq 0$ and $\RR_{-T}^n  = \{ y \in \RR^n~ |~ y^n > -T \}$.
Consider the problem
\begin{gather}
\begin{cases}
\Delta u + n(n-2) u^{p} = 0,~u > 0 & \text{ in } \RR^n_{-T}, \displaybreak[0] \\
\frac{\partial u}{\partial x_n} = 0 & \text{ on } \partial \RR^n_{-T},
\displaybreak[0] \\
u(0) = 1,~~0 \text{ is a local maximum of } u,
\end{cases} \nonumber
\label{Euclidean_prop_bubble}
\end{gather}
where $p \in (1,\frac{n+2}{n-2}]$. If $p < \frac{n+2}{n-2}$ then this problem has no
solution. If $p=\frac{n+2}{n-2}$
then
\begin{gather}
u(x^\prime, x_n) = \Big ( \frac{1}{1 + |(x^\prime,x_n) | ^2} \Big )^\frac{n-2}{2} =
U(x) \nonumber
\end{gather}
in which case $T=0$ necessarily.
\end{prop}
\begin{proof} \cite{LZ} (see also the proof of proposition 2.4 in \cite{FA}, and
\cite{HL} p. 498).
\end{proof}

Now we recall some transformation laws.
\begin{prop} Let $(M,g)$ be a Riemannian manifold with boundary and $\phi > 0$ a
smooth function.
Let $\tilde{g} = \phi^{\frac{4}{n-2}} g$, then
\label{transformation_laws}
\begin{align}
L_{\tilde{g}}(\phi^{-1} u ) & = \phi^{-\frac{n+2}{n-2}} L_g u \label{conf_inv_eq}
\displaybreak[0] \\
R_{\tilde{g}} & = -c(n)^{-1} \phi^{-\frac{n+2}{n-2}}L_g \phi  \label{conf_inv_scalar}
\displaybreak[0] \\
B_{\tilde{g}} (\phi^{-1} u) & = \phi^{-\frac{n}{n-2}} B_g u \label{conf_inv_bc}
\displaybreak[0] \\
\tilde{\kappa}_{ij} & = \phi^{\frac{2}{n-2}} \kappa_{ij} +
\frac{2}{n-2}\phi^{\frac{4-n}{n-2}}\frac{\partial \phi}{\partial \nu_g} g_{ij}
\label{conf_sec_fundamental_form} \displaybreak[0] \\
\tilde{\kappa} & =  \frac{2}{n-2}  \phi^{-\frac{n}{n-2}} B_g \phi
\label{trans_law_mean_curvature}
\end{align}
where quantities with $\tilde{~}$ refer to the metric $\tilde{g}$, $\kappa_{ij}$ and
$\kappa$ are the second fundamental form and
the mean curvature, respectively.
\end{prop}
\begin{proof} Direct calculation (see \cite{KMS,Es,M} for example).
\end{proof}

\begin{prop} Up to a conformal change we can assume that in small balls the scalar
curvature
is positive and that the mean curvature of $\partial M$ vanishes.
\label{conf_change_negative_R}
\end{prop}
\begin{proof}
The idea of the proof is to perform two conformal changes on the metric, one to
produce a metric
with zero mean curvature and a further one to achieve positive scalar curvature.
Denote by $\phi_1 > 0$, the first eigenfunction of the conformal Laplacian
with boundary
condition $B_g \phi_1 = 0$, i.e.,
\begin{gather}
\begin{cases}
L_g \phi_1 + \la_1 \phi_1 = 0, & \text{ in } M, \displaybreak[0] \\
B_g \phi_1 = 0, &\text{ on } \partial M.
\end{cases} \nonumber
\end{gather}
See \cite{Es} for the existence of $\phi_1$; the fact that $\phi_1 > 0$ follows from
a standard calculus of variation argument. By transformation law
(\ref{trans_law_mean_curvature}),
the metric $g_1 = \phi_1^{\frac{4}{n-2}} g$ has zero mean curvature.

Now let $x_0 \in \partial M$ and consider a small ball $B_{2\delta}(x_0)$ near the
boundary.
Denote by $\psi_1 > 0$, the first eigenfunction
of the Laplacian $\Delta_{g_1}$ with the boundary condition as below:
\begin{gather}
\begin{cases}
\Delta_{g_1} \psi_1 + \mu_1 \psi_1 = 0,  & \text{ in } B_{2\delta}(x_0),
\displaybreak[0] \\
\psi_1 = 0, & \text{ on } \partial^+ B_{2\delta}(x_0), \displaybreak[0] \\
B_{g_1} \psi_1 = \frac{\partial \psi_1}{\partial \nu_{g_1}} = 0, & \text{ on }
\partial^\prime B_{2\delta}(x_0).
\end{cases} \nonumber
\end{gather}
The existence and positivity of $\psi_1$ again follows from a standard calculus of
variations argument.
Consider the metric $\tilde{g} = \psi_1^{\frac{4}{n-2}} g_1$ on $B_{2\delta}(x_0)$.
Then from (\ref{conf_inv_scalar}),
\begin{gather}
R_{\tilde{g}} = -c(n)^{-1} \psi^{\frac{n+2}{n-2}} L_{g_1} \psi_1 =
-c(n)^{-1} \psi_1^{\frac{n+2}{n-2}} (\Delta_{g_1}\psi_1 - R_{g_1} \psi_1). \nonumber
\end{gather}
Since $\mu_1 \rar \infty$ as $\delta \rar 0$ we can choose $\de > 0$ so small that
\begin{gather}
 \Delta_{g_1}\psi_1 - R_{g_1} \psi_1  = -\mu_1 \psi_1 - R_{g_1} \psi_1 < 0, \nonumber
\end{gather}
and therefore $R_{\tilde{g}} > 0$ on $B_{\delta}(x_0)$. Notice that shrinking
$B_{2\de}(x_0)$
does not affect $R_{g_1}$ as $\phi_1$ is defined on the whole of $M$. Finally, the
mean curvature for $\tilde{g}$ is $\tilde{\kappa} =  \frac{2}{n-2}
\psi^{-\frac{n}{n-2}} B_{g_1} \psi = 0$.
\end{proof}

The next result immediately follows.

\begin{coro} Up to a conformal change the maximum principle holds for the conformal
Laplacian in small balls. More precisely, if
$L_g u \geq 0$ in $B_\sigma(x_0)$, $u > 0$, then there exists a constant $C > 0$,
independent of $u$, such that
$\sup_{B_\si(x_0)} u \leq C \sup_{\partial B_\si(x_0)} u$, provided $\si$ is small
enough.
\label{coro_max_principle}
\end{coro}

\begin{lemma} Let $\psi$ be a solution of
\begin{gather}
\begin{cases}
\Delta \psi + n(n+2) U^{\frac{4}{n-2}} \psi = 0,  & \text{ in  }\RR_{+}^n, \nonumber
\displaybreak[0] \\
\frac{\partial \psi}{\partial y^n} = 0, & \text{ on } \partial \RR^{n-1}, \nonumber
\displaybreak[0] \\
\lim_{|y| \rar \infty} \psi(y) = 0. \nonumber
\end{cases} \nonumber
\end{gather}
Then it takes the form
\begin{gather}
\psi(y) = c_0 \Big( \frac{n-2}{2} U + y \cdot \nabla U \Big ) + \sum_{j=1}^{n-1} c_j
\partial_j U, \nonumber
\end{gather}
for some constants $c_0,\dots,c_{n-1}$.
\label{class_sol}
\end{lemma}
\begin{proof}
 Since $\partial_n \psi = 0$ on $\RR^{n-1}$, we can make a $C^2$ reflection across
$\RR^{n-1}$ and then the result
follows from \cite{CC}.
\end{proof}

The following is a Harnack-type inequality.

\begin{lemma} Let $x_i \rar \bar{x}$ be an isolated blow-up point and assume that
$\bar{r}$ is sufficiently small.
Then for all $r$ such that $0<r<\bar{r}$ we have
\begin{gather}
\sup_{B_{r}(x_i) \backslash B_{r/2}(x_i)} u_i \leq C \inf_{B_{r}(x_i) \backslash
B_{r/2}(x_i)} u_i, \nonumber
\end{gather}
for some constant $C$ independent of $i$ and $r$.
\label{Harnack}
\end{lemma}
\begin{proof}
It follows from a combination of lemma A.1 of \cite{HL}, the Harnack inequality, and
the definition of isolated
blow-up points.
\end{proof}

\begin{prop} (Pohozaev identity) Let $u>0$ be a solution of
$L_{g}u + K f_i^{-\de}u^{p} = 0$ on
$B_\rho^+ = \{ x \in B_\rho(0) ~ \big | ~ x^n \geq 0 \}$. Then
{ \allowdisplaybreaks
\begin{eqnarray}
  & & \int_{\partial B_\rho^+} \Big ( (\frac{n-2}{2} u + x^k \partial_k u
)\frac{\partial u}{\partial \nu_0}
-\frac{1}{2} x^k \nu_0^k |\nabla_0 u|^2 + \frac{1}{p+1} K(x) x^k\nu_0^k u^{p+1} \Big
) d\si
\label{Pohozaev}
 \\
& & = - \int_{B_\rho^+} (\frac{n-2}{2} u + x^k \partial_k u )\big( (g^{ij} -
\de^{ij})\partial_{ij}u + \partial_jg^{ij}\partial_i u \big ) dx \nonumber
 \\
& & + \int_{B_\rho^+} c(n) (\frac{n-2}{2} u + x^k \partial_k u )Ru dx +
\frac{1}{p+1} \int_{B_\rho^+} x^k\partial_k K(x) u^{p+1} dx \nonumber
\\
& & + \Big (\frac{n}{p+1} - \frac{n-2}{2} \Big )
\int_{B_\rho^+} K(x) u^{p+1} dx,
\nonumber
 \end{eqnarray} }
where quantities with$~_0$ refer to the Euclidean metric and $K(x) = Kf^{-\de}(x)$.
\label{Pohozaev_prop}
\end{prop}
\begin{proof}
Standard integration by parts argument.
\end{proof}

\end{document}